\documentclass[letter,10pt]{amsart}
\usepackage{amssymb,amsmath,amsfonts}
\usepackage{mathrsfs}
\usepackage[left=1 in, right=1 in,top=1 in, bottom=1 in]{geometry}
\usepackage{mathenv}

\newtheorem{theorem}{Theorem}[section]
\newtheorem{lemma}[theorem]{Lemma}

\newtheorem{corollary}[theorem]{Corollary}

\newtheorem{remark}[theorem]{Remark}
\newtheorem{definition}[theorem]{Definition}

\makeatletter
\renewcommand \theequation {%
\ifnum \c@section>\z@ \@arabic\c@section.%
\fi\@arabic\c@equation} \@addtoreset{equation}{section}
\makeatother

\providecommand{\ud}[1]{\mathrm{d}{#1}}
\providecommand{\abs}[1]{\left\vert#1\right\vert}
\providecommand{\nm}[1]{\left\Vert#1\right\Vert}

\providecommand{\br}[1]{\left\langle #1 \right\rangle}
\providecommand{\tm}[2]{\left\Vert#1\right\Vert_{L^2(#2)}}
\providecommand{\im}[2]{\left\Vert#1\right\Vert_{L^{\infty}(#2)}}

\providecommand{\lnnm}[1]{{\left\Vert#1\right\Vert}_{L^{\infty}L^{\infty}}}
\providecommand{\lnm}[1]{\left\Vert#1\right\Vert_{L^{\infty}}}
\providecommand{\tnm}[1]{\left\Vert#1\right\Vert_{L^{2}}}
\providecommand{\tnnm}[1]{{\left\Vert#1\right\Vert}_{L^{2}L^2}}

\def\p{\partial}
\def\ls{\lesssim}

\def\half{\frac{1}{2}}
\def\rt{\rightarrow}

\def\no{\nonumber}
\def\ue{\mathrm{e}}

\def\ds{\displaystyle}

\def\u{U^{\e}}
\def\ub{\mathscr{U}^{\e}}
\def\bu{\bar U^{\e}}
\def\bub{\bar{\mathscr{U}}^{\e}}
\def\e{\epsilon}
\def\s{\mathcal{S}}
\def\vx{\vec x}
\def\vw{\vec w}

\def\nx{\nabla_{x}}
\def\vn{\vec\nu}

\def\px{\p_{\eta}}

\def\l{\lambda}

\def\ll{\mathcal{L}}

\def\k{\kappa}
\def\q{Q}
\def\qb{\mathscr{Q}}

\def\v{\mathscr{V}}
\def\w{\mathscr{W}}
\def\d{\delta}
\def\t{\mathcal{T}}
\def\k{\mathcal{K}}

\def\pp{\mathcal{P}}
\def\xc{X_{cl}}
\def\wc{W_{cl}}

\def\vt{\vec t}
\def\vr{\vec r}

\def\id{{\bf{1}}}

\def\tf{\tilde F}
\def\tv{\tilde V}

\def\a{\mathscr{A}}
\def\b{\mathscr{B}}
\def\rk{R}

\def\rr{\mathscr{R}}

\begin{document}
\title{Regularity of Milne Problem with Geometric Correction in 3D}

\author[Yan Guo]{Yan Guo}
\address[Yan Guo]{\newline\indent
Division of Applied Mathematics, Brown University,
\newline\indent Providence, RI 02912, USA }
\email{Yan\_Guo@brown.edu}

\author[Lei Wu]{Lei Wu}
\address[Lei Wu]{
   \newline\indent Department of Mathematical Sciences, Carnegie Mellon University
\newline\indent Pittsburgh, PA 15213, USA}
\email{lwu2@andrew.cmu.edu}

\thanks{
L. Wu is supported by NSF grant
0967140. Y. Guo is supported in part by NSFC grant 10828103, NSF grant 1209437, Simon Research Fellowship and BICMR.}


\begin{abstract}
Consider the Milne problem with geometric correction in a 3D convex domain. Via bootstrapping arguments, we establish $W^{1,\infty}$ regularity for its solutions. Combined with a uniform $L^6$ estimate, such regularity leads to the validity of diffusive expansion for the neutron transport equation with diffusive boundary conditions.
\end{abstract}

\maketitle

\tableofcontents

\newpage


\pagestyle{myheadings} \thispagestyle{plain} \markboth{YAN GUO AND LEI WU}{REGULARITY OF MILNE PROBLEM IN 3D}

\section{Introduction}

\subsection{Motivation and Formulation}

Milne problem is the main tool to study boundary layer effect in kinetic equations. Here to motivate 3D $\e$-Milne problem with geometric correction, we consider the steady neutron transport equation in a
three-dimensional convex domain with diffusive boundary. In the space
domain $\vx=(x_1,x_2,x_3)\in\Omega$ where $\p\Omega\in C^3$ and the velocity domain
$\vw=(w_1,w_2,w_3)\in\s^2$, the neutron density $u^{\e}(\vx,\vw)$
satisfies
\begin{eqnarray}\label{transport}
\left\{
\begin{array}{rcl}\displaystyle
\e \vw\cdot\nabla_x u^{\e}+u^{\e}-\bar u^{\e}&=&0\ \ \text{in}\ \ \Omega,\\
u^{\e}(\vx_0,\vw)&=&\pp[u^{\e}](\vx_0)+\e g(\vx_0,\vw)\ \ \text{for}\
\ \vw\cdot\vn<0\ \ \text{and}\ \ \vx_0\in\p\Omega,
\end{array}
\right.
\end{eqnarray}
where
\begin{eqnarray}\label{average}
\bar u^{\e}(\vx)=\frac{1}{4\pi}\int_{\s^2}u^{\e}(\vx,\vw)\ud{\vw},
\end{eqnarray}
\begin{eqnarray}\label{diffusive}
\pp[u^{\e}](\vx_0)=\frac{1}{4\pi}\int_{\vw\cdot\vn>0}u^{\e}(\vx_0,\vw)(\vw\cdot\vn)\ud{\vw},
\end{eqnarray}
$\vn$ is the outward unit normal vector, with the Knudsen number $0<\e<<1$. Also, $u^{\e}$ satisfies the normalization condition
\begin{eqnarray}\label{normalization}
\int_{\Omega\times\s^2}u^{\e}(\vx,\vw)\ud{\vw}\ud{\vx}=0,
\end{eqnarray}
and $g$ satisfies the compatibility condition
\begin{eqnarray}\label{compatibility}
\int_{\p\Omega}\int_{\vw\cdot\vn<0}g(\vx_0,\vw)(\vw\cdot\vn)\ud{\vw}\ud{\vx_0}=0.
\end{eqnarray}

A classical problem is to study the diffusive limit of (\ref{transport}) as $\e\rt0$.
Generally speaking, the solution $u^{\e}$ varies smoothly and slowly in the interior of $\Omega$, and behaves like $u^{\e}-\bar u^{\e}=0$ which ignores $\e\vw\cdot\nabla_xu^{\e}$. However, its value changes severely when approaching boundary $\p\Omega$ and $\e\vw\cdot\nabla_xu^{\e}$ becomes non-negligible. The smaller $\e$ is, the more severe $u^{\e}$ changes. This indicates that $u^{\e}$ actually contains two separate parts with different scalings, i.e. interior solution and boundary layer. In particular, the boundary layer is a function with scaled variable defined in a thin layer of thickness $O(\e)$ close to boundary in $\Omega$.

In the region of boundary layer, assume $\mu$ denotes the distance to the boundary in the inward normal direction and $(\tau_1,\tau_2)$ denote a local orthogonal curvilinear coordinate system for $\p\Omega$. Then we have
\begin{eqnarray}
\e\vw\cdot\nabla_x= -\e(\vw\cdot\vn)\frac{\p}{\p\mu}+\frac{\e}{R_1-\mu}(\vw\cdot\vt_1)\frac{\p}{\p\tau_1}+\frac{\e}{R_2-\mu}(\vw\cdot\vt_2)\frac{\p}{\p\tau_2},
\end{eqnarray}
where $(\vt_1,\vt_2)$ are orthogonal tangential vectors associated with $(\tau_1,\tau_2)$, and $\Big(R_1(\tau_1,\tau_2),R_2(\tau_1,\tau_2)\Big)$ denote two radium of principle curvature. Since $\vw\in\s^2$, we define the spherical velocity substitution as
\begin{eqnarray}
\left\{
\begin{array}{rcl}
-\vw\cdot\vn&=&\sin\phi,\\
\vw\cdot\vt_1&=&\cos\phi\sin\psi,\\
\vw\cdot\vt_2&=&\cos\phi\cos\psi.
\end{array}
\right.
\end{eqnarray}
With the rescaled distance $\eta=\dfrac{\mu}{\e}$, we may represent
\begin{eqnarray}\label{it 01}
\e\vw\cdot\nabla_x=\sin\phi\frac{\p
}{\p\eta}-\e\bigg(\dfrac{\sin^2\psi}{R_1-\e\eta}+\dfrac{\cos^2\psi}{R_2-\e\eta}\bigg)\cos\phi\frac{\p
}{\p\phi}+\text{higher-order}\ \text{terms}.
\end{eqnarray}
Therefore, in order to construct boundary layer, it suffices to study the 3D $\e$-Milne problem with geometric correction for $f^{\e}(\eta,\phi,\psi)$ in $(\eta,\phi,\psi)\in [0,L]\times[-\pi/2,\pi/2]\times[-\pi,\pi]$ as
\begin{eqnarray}\label{mt}
\left\{ \begin{array}{rcl}\displaystyle \sin\phi\frac{\p
f^{\e}}{\p\eta}-\e\bigg(\dfrac{\sin^2\psi}{R_1-\e\eta}+\dfrac{\cos^2\psi}{R_2-\e\eta}\bigg)\cos\phi\frac{\p
f^{\e}}{\p\phi}+f^{\e}-\bar f^{\e}&=&S^{\e}(\eta,\phi,\psi),\\
f^{\e}(0,\phi,\psi)&=& h^{\e}(\phi,\psi)\ \ \text{for}\
\ \sin\phi>0,\\
f^{\e}(L,\phi,\psi)&=&f^{\e}(L,-\phi,\psi),
\end{array}
\right.
\end{eqnarray}
where $h^{\e}$ and $S^{\e}$ are two functions given a priori and
\begin{eqnarray}
\bar f^{\e}(\eta)=\frac{1}{4\pi}\int_{-\pi}^{\pi}\int_{-\pi/2}^{\pi/2}f^{\e}(\eta,\phi,\psi)\cos\phi\ud{\phi}\ud{\psi}
\end{eqnarray}
in which $\cos\phi$ shows up as the Jacobian of spherical coordinates in integration, and $L=\e^{-n}$ for some $n$ to be specified later. Note that actually $f^{\e}$, $S^{\e}$, $h^{\e}$ and $R_1$, $R_2$ are all related to $(\tau_1,\tau_2)$. Since boundary layer is defined locally on $\p\Omega$ and our analysis focuses on the case for fixed $(\tau_1,\tau_2)$, we do not need to specify such dependence explicitly unless necessary.

\subsection{Main Result}

Define norms
\begin{eqnarray}
\lnnm{f}&=&\sup_{(\eta,\phi,\psi)}\abs{f(\eta,\phi,\psi)},\\
\lnm{f}(\eta)&=&\sup_{(\phi,\psi)\ \text{with}\ \sin\phi>0}\abs{f(\eta,\phi,\psi)},
\end{eqnarray}
\begin{eqnarray}
\tnnm{f}&=&\bigg(\int_0^{L}\int_{-\pi}^{\pi}\int_{-\pi/2}^{\pi/2}\abs{f(\eta,\phi,\psi)}^2\cos\phi\ \ud{\phi}\ud{\psi}\ud{\eta}\bigg)^{1/2},\\
\tnm{f}(\eta)&=&\bigg(\int_{-\pi}^{\pi}\int_{-\pi/2}^{\pi/2}\abs{f(\eta,\phi,\psi)}^2\cos\phi\ \ud{\phi}\ud{\psi}\bigg)^{1/2},
\end{eqnarray}
and the inner product as
\begin{eqnarray}
\br{f,g}(\eta)&=&\int_{-\pi}^{\pi}\int_{-\pi/2}^{\pi/2}f(\eta,\phi,\psi)g(\eta,\phi,\psi)\cos\phi\
\ud{\phi}\ud{\psi}.
\end{eqnarray}

\begin{theorem}(Well-Posedness and Decay)\label{main 1}
Assume $0<n<\dfrac{2}{5}$ and there exist some constants $M,K>0$ uniform in $\e$, such that
\begin{eqnarray}
\lnm{h^{\e}}\leq M,
\end{eqnarray}
and
\begin{eqnarray}
\lnnm{\ue^{K\eta}S^{\e}}&\leq& M.\no
\end{eqnarray}
Then for $K_0>0$ sufficiently small, there exists a constant $f_L^{\e}$ and the solution $f^{\e}(\eta,\phi,\psi)$ to the
$\e$-Milne problem (\ref{mt}) satisfies
\begin{eqnarray}
\lnnm{\ue^{K_0\eta}(f^{\e}-f_L^{\e})}\leq C.
\end{eqnarray}
Here $C\geq0$ denotes a universal constant independent of $\e$.
\end{theorem}

\begin{theorem}(Weighted Regularity)\label{main 2}
Assume $0<n<\dfrac{2}{5}$ and there exist some constants $M,K>0$ uniform in $\e$, such that
\begin{eqnarray}
\lnm{h^{\e}}+\lnm{\frac{\p h^{\e}}{\p\phi}}\leq M,
\end{eqnarray}
and
\begin{eqnarray}
\lnnm{\ue^{K\eta}S^{\e}}+\lnnm{\ue^{K\eta}\frac{\p S^{\e}}{\p\eta}}
+\lnnm{\ue^{K\eta}\frac{\p S^{\e}}{\p\phi}}&\leq& M.\no
\end{eqnarray}
Then for $K_0>0$ sufficiently small, we have
\begin{eqnarray}
\lnnm{\ue^{K_0\eta}\zeta\frac{\p(f^{\e}-f_L^{\e})}{\p\eta}}+\lnnm{\ue^{K_0\eta}\zeta\frac{\p(f^{\e}-f_L^{\e})}{\p\phi}}\leq C\abs{\ln(\e)}^8,
\end{eqnarray}
where the weight function
\begin{eqnarray}\label{weight}
\zeta(\eta,\phi,\psi)=\Bigg(1-\bigg(\frac{\rk_1-\e\eta}{\rk_1}\bigg)^{2\sin^2\psi}
\bigg(\frac{\rk_2-\e\eta}{\rk_2}\bigg)^{2\cos^2\psi}\cos^2\phi\Bigg)^{1/2}.
\end{eqnarray}
If further for $i=1,2$,
\begin{eqnarray}
\lnm{\frac{\p h^{\e}}{\p\psi}}+\lnm{\frac{\p h^{\e}}{\p\tau_i}}+\lnnm{\ue^{K\eta}\frac{\p S^{\e}}{\p\psi}}
+\lnnm{\ue^{K\eta}\frac{\p S^{\e}}{\p\tau_i}}&\leq& M,
\end{eqnarray}
we have
\begin{eqnarray}\label{it 001}
\lnnm{\ue^{K_0\eta}\frac{\p(f^{\e}-f_L^{\e})}{\p\psi}}+\lnnm{\ue^{K_0\eta}\frac{\p(f^{\e}-f_L^{\e})}{\p\tau_i}}\leq C\abs{\ln(\e)}^8.
\end{eqnarray}
\end{theorem}

\begin{remark}
It is easy to see $\zeta\geq\abs{\sin\phi}$. Then Theorem \ref{main 2} naturally implies
\begin{eqnarray}
\lnnm{\ue^{K_0\eta}\sin\phi\frac{\p(f^{\e}-f_L^{\e})}{\p\eta}}\leq C\abs{\ln(\e)}^8,
\end{eqnarray}
which is bounded away from the grazing set. More importantly, due to the half-line geometry of the Milne problem, the tangential derivatives are bounded in (\ref{it 001}) up to the grazing set. This is a sharp contrast to \cite{Guo.Kim.Tonon.Trescases2013} in a bounded domain.
\end{remark}
\ \\
As an application, thanks to the uniform bounds for tangential derivatives (\ref{it 001}), we finally establish the diffusive limit of neutron transport equation.
\begin{corollary}\label{main 3}
Assume $g(\vx_0,\vw)\in C^2(\Gamma^-)$ satisfying (\ref{compatibility}). Then for the steady neutron
transport equation (\ref{transport}), there exists a unique solution
$u^{\e}(\vx,\vw)\in L^{\infty}(\Omega\times\s^2)$ satisfying (\ref{normalization}). Moreover, for any $0<\d<<1$, the solution obeys the estimate
\begin{eqnarray}
\im{u^{\e}(\vx,\vw)-\u_0(\vx)}{\Omega\times\s^2}\leq C(\d,\Omega)\e^{\frac{1}{3}-\d},
\end{eqnarray}
where $\u_0(\vx)$ satisfies
\begin{eqnarray}
\left\{
\begin{array}{rcl}
\Delta_x\u_0&=&0\ \ \text{in}\
\ \Omega,\\\rule{0ex}{2em}\dfrac{\p\u_0}{\p\vec
\nu}&=&\dfrac{1}{\pi^2}\displaystyle
\int_{\vw\cdot\vn<0}g(\vx,\vw)\abs{\vw\cdot\vn}\ud{\vw}\ \ \text{on}\ \
\p\Omega,\\\rule{0ex}{1em}
\displaystyle\int_{\Omega}\u_0(\vx)\ud{\vx}&=&0,
\end{array}
\right.
\end{eqnarray}
in which $C(\d,\Omega)>0$ denotes a constant that depends on $\d$ and $\Omega$.
\end{corollary}

\subsection{Background and Methods}

At the core of boundary layer analysis, the study of Milne problem is consistent with the development of asymptotic analysis of kinetic equations in bounded domains. Since 1960s, people have discovered several methods to study the well-posedness of Milne problem, and apply them to asymptotic expansion. We refer to the references \cite{Esposito.Guo.Kim.Marra2015}, \cite{Arkeryd.Esposito.Marra.Nouri2011}, \cite{Arkeryd.Nouri2007}, \cite{Yang2012}, \cite{Huang.Wang.Yang2010}, \cite{Ukai.Yang.Yu2003}, \cite{Duan.Yang2013}, \cite{Coron.Golse.Sulem1988}, \cite{Aoki.Bardos.Takata2003}, \cite{Bardos.Yang2012}, \cite{Chen.Liu.Yang2004}, \cite{Golse.Perthame.Sulem1988}, \cite{Huang.Qin2009}, \cite{Wang.Yang.Yang2006},
\cite{Arkeryd.Nouri2013}, \cite{Arkeryd.Esposito.Marra.Nouri2010}, \cite{Cercignani.Marra.Esposito1998},
\cite{Larsen1974=}, \cite{Larsen1974}, \cite{Larsen.D'Arruda1976}, and \cite{Larsen.Zweifel1976} for more details. In 1979, diffusive limit of steady neutron transport equation was systematically investigated in \cite{Bensoussan.Lions.Papanicolaou1979} (see also \cite{Bardos.Caflisch.Nicolaenko1986} and \cite{Bardos.Santos.Sentis1984}).

The key idea of \cite{Bensoussan.Lions.Papanicolaou1979}, \cite{Bardos.Caflisch.Nicolaenko1986} and \cite{Bardos.Santos.Sentis1984} is to study the classical Milne problem as
\begin{eqnarray}\label{it 002}
\left\{ \begin{array}{rcl}\displaystyle \sin\phi\frac{\p
f^{\e}}{\p\eta}+f^{\e}-\bar f^{\e}&=&S^{\e}(\eta,\phi,\psi),\\
f^{\e}(0,\phi,\psi)&=& h^{\e}(\phi,\psi)\ \ \text{for}\
\ \sin\phi>0,\\
\lim_{\eta\rt\infty}f^{\e}(\eta,\phi,\psi)&=&f^{\e}_{\infty}.
\end{array}
\right.
\end{eqnarray}
In \cite{Bensoussan.Lions.Papanicolaou1979}, the authors proved that $f^{\e}$ is well-posed and decays exponentially fast to some constant $f^{\e}_{\infty}$ in $L^{\infty}$.

Unfortunately, as discovered recently in \cite{AA003}, the lack of regularity of such classical Milne problem (\ref{it 002}) has been overlooked for non-flat bounded domains. The solutions of (\ref{it 002}) are singular in the normal direction, which leads to singularity in the tangential directions, resulting in break-down of diffusive expansion with classical Milne boundary layers.

The regularity of the Milne problem is the central issue. In \cite{AA003} and \cite{AA006}, a new approach with geometric correction from the next-order diffusive expansion has been introduced to ensure regularity in the cases of 2D plate and annulus, i.e. to solve for $f^{\e}(\eta,\phi)$ satisfying
\begin{eqnarray}\label{mt2}
\left\{ \begin{array}{rcl}\displaystyle \sin\phi\frac{\p
f^{\e}}{\p\eta}-\dfrac{\e}{R-\e\eta}\cos\phi\frac{\p
f^{\e}}{\p\phi}+f^{\e}-\bar f^{\e}&=&S^{\e}(\eta,\phi),\\
f^{\e}(0,\phi)&=& h^{\e}(\phi)\ \ \text{for}\
\ \sin\phi>0,\\
f^{\e}(L,\phi)&=&f^{\e}(L,-\phi),
\end{array}
\right.
\end{eqnarray}
where $R$ denotes the radius of curvature. Also, in \cite{AA007}, weighted $W^{1,\infty}$ estimates was proved to treat more general 2D convex domains.

There are three main ingredients to generalize our previous results to 3D convex domains.

The first difficulty is the lack of conserved energy. Consider the simplest case that $S=0$ and we omit $\e$ temporarily. Assume
\begin{eqnarray}
F(\eta,\psi)=-\e\bigg(\dfrac{\sin^2\psi}{R_1-\e\eta}+\dfrac{\cos^2\psi}{R_2-\e\eta}\bigg).
\end{eqnarray}
Taking inner product with $f$ on both sides of (\ref{mt}), we obtain
\begin{eqnarray}
\frac{1}{4}\frac{\p}{\p\eta}\br{f,f\sin(2\phi)}+\half\br{F(\eta,\psi),\frac{\p
(f^2)}{\p\phi}\cos^2\phi}
+\tnm{f-\bar f}^2&=&0.
\end{eqnarray}
We may integrate by parts to get
\begin{eqnarray}
\half\br{F(\eta,\psi),\frac{\p
(f^2)}{\p\phi}\cos^2\phi}&=&\half\br{F(\eta,\psi)f,f\sin(2\phi)}.
\end{eqnarray}
Since $F(\eta,\psi)$
depends on $\psi$ when $R_1\neq R_2$ in 3D, we cannot further pull $F$ out of the integral, i.e.
\begin{eqnarray}\label{it 003}
\half\br{F(\eta,\psi)f,f\sin(2\phi)}\overset{?}{=}\half F(\eta,\psi)\br{f,f\sin(2\phi)}.
\end{eqnarray}
This important equality is true in (\ref{mt2}) for 2D domains and yields an ordinary differential equation for $\ds\br{f,f\sin(2\phi)}$, leading to the closure of the $L^2$ estimate of the microscopic part $f-\bar f$. Similarly, taking inner product with $1$ on both sides of (\ref{mt2}) and integrating by parts, we easily get the orthogonality relation \begin{eqnarray}\label{it 004}
\br{f,\sin\phi}=0,
\end{eqnarray}
which plays a crucial role in estimating hydrodynamical part $\bar f$. Unfortunately, both (\ref{it 003}) and (\ref{it 004}) break down in 3D domains.

To circumvent these two major difficulties, as Lemma \ref{Milne finite LT} reveals, we decompose
\begin{eqnarray}
F(\eta,\psi)\cos\phi\frac{\p f}{\p\phi}=\tf(\eta)\cos\phi\frac{\p
f}{\p\phi}+G(\eta)\cos^2\psi\cos\phi\frac{\p
f}{\p\phi},
\end{eqnarray}
where
\begin{eqnarray}
\tf(\eta)&=&-\frac{\e}{R_1-\e\eta},\\
G(\eta)&=&-\frac{\e(R_1-R_2)}{(R_1-\e\eta)(R_2-\e\eta)},
\end{eqnarray}
in which $\tf$ behaves like 2D force (independent of $\psi$) and $G$ can be regarded as a source term. Roughly speaking, taking inner product with $f$ in (\ref{mt}), we obtain
\begin{eqnarray}\label{it 005}
\tnnm{f-\bar f}^2&\ls& C\abs{\int_0^LG(y)\br{f\cos^2\psi,f\sin\phi}(y)\ud{y}}+\text{lower-order}\ \text{terms}\\
&\ls& C+C\lnnm{G}\tnnm{f}^2,\no
\end{eqnarray}
which means we cannot close the estimate for $f-\bar f$ alone without invoking $\tnnm{\bar f}$. On the other hand, taking inner product with $\sin\phi$ in (\ref{mt}) indicates
\begin{eqnarray}\label{it 006}
\tnnm{\bar f}^2&\ls& C\int_0^L\abs{\int_s^L\int_z^LG(y)\br{\cos^2\psi,(f-\bar f)\sin\phi}(y)\ud{y}\ud{z}}^2\ud{s}+\text{lower-order}\ \text{terms}\\
&\ls& C+CL^3\tnnm{G}\tnnm{f-\bar f}.\no
\end{eqnarray}
(\ref{it 005}) and (\ref{it 006}) form a coupled system for $f-\bar f$ and $\bar f$ and require careful analysis of the interplay between microscopic and hydrodynamic parts. Also, we have to delicately choose $L=\e^{-n}$ with $0<n<\dfrac{2}{5}$ to create a small constant such that an intricate bootstrapping argument can finally close the $L^2$ estimates.

The second key ingredient in our analysis is to establish the regularity estimate of $\e$-Milne problem. Proving diffusive limit in transport equations requires boundary layer expansion higher than leading-order term, which means we need $L^{\infty}$ estimate of the tangential derivatives
\begin{eqnarray}
\frac{\p f}{\p\tau_1},\ \ \frac{\p f}{\p\tau_2},\ \ \text{and}\ \ \frac{\p f}{\p\psi}.
\end{eqnarray}
In the case when $R_1=R_2$ are constant independent of $\tau_1$ and $\tau_2$, as in a perfect ball $\Big\{\abs{\vx}=1\Big\}$, $\dfrac{\p f}{\p\tau_i}$ for $i=1,2$ is smooth, since the tangential derivative commutes with the equation. On the other hand, when $R_1$ and $R_2$ are functions of $\tau_i$, then $\dfrac{\p f}{\p\tau_i}$ relates to the normal derivative $\dfrac{\p f}{\p\eta}$, and $\dfrac{\p f}{\p\psi}$ relates to the velocity derivative $\dfrac{\p f}{\p\phi}$.

Our main contribution is to show $\dfrac{\p f}{\p\tau_i}$ and $\dfrac{\p f}{\p\psi}$ are bounded even if $R_1$ and $R_2$ are not identical constant for a general convex domain(see Theorem \ref{Milne tangential}). Our proof is intricate and delicate, which relies on the weighted $L^{\infty}$ estimates for the normal derivative with detailed analysis along the characteristic curves in the presence of non-local operator $\bar f$. The convexity and invariant kinetic distance $\zeta$ defined in (\ref{weight}) play the key role.

The third ingredient is a new $L^6-L^{\infty}$ framework developed to improve remainder estimates in
\begin{eqnarray}
\left\{
\begin{array}{rcl}\displaystyle
\e\vw\cdot\nx u+u-\bar
u&=&S(\vx,\vw)\ \ \text{in}\ \ \Omega,\\
u(\vx_0,\vw)&=&\pp[u](\vx_0)+h(\vx_0,\vw)\ \ \text{for}\ \
\vw\cdot\vn<0\ \ \text{and}\ \ \vx_0\in\p\Omega,
\end{array}
\right.
\end{eqnarray}
The main idea is to introduce special test functions in the weak formulation to treat kernel and non-kernel parts separately. In principle, we get $L^6$ estimate
\begin{eqnarray}
&&\frac{1}{\e^{\frac{1}{2}}}\nm{(1-\pp)[u]}_{L^2(\Gamma^+)}+\nm{
\bar u}_{L^{6}(\Omega\times\s^2)}+\frac{1}{\e}\nm{u-\bar
u}_{L^2(\Omega\times\s^2)}\\
&\leq&
C\bigg(o(1)\e^{\frac{1}{2}}\nm{u}_{L^{\infty}(\Gamma^+)}+\frac{1}{\e}\tm{S}{\Omega\times\s^2}+
\frac{1}{\e^2}\nm{S}_{L^{\frac{6}{5}}(\Omega\times\s^2)}+\frac{1}{\e}\nm{h}_{L^2(\Gamma^-)}+\nm{h}_{L^{4}(\Gamma^-)}\bigg),\nonumber
\end{eqnarray}
where $o(1)$ denotes a sufficiently small constant (see Theorem \ref{LN estimate}). The proof relies on a careful analysis using sharp interpolation and Young's inequality. Finally, the utilization of modified double Duhamel's principle and a bootstrapping argument yield the $L^{\infty}$ estimate as
\begin{eqnarray}
\im{u}{\Omega\times\s^2}\leq C\bigg(
\frac{1}{\e^{\frac{1}{2}}}\nm{\bar u}_{L^{6}(\Omega\times\s^2)}+\im{S}{\Omega\times\s^2}+\im{g}{\Gamma^-}\bigg).
\end{eqnarray}

Our methods are currently being applied to the study of hydrodynamic limit of Boltzmann equation in the bounded domains with boundary layer corrections.

\subsection{Notation and Structure}

Throughout this paper, unless specified, $C>0$ denotes a universal constant which does not depend on the data and
can change from one inequality to another.
When we write $C(z)$, it means a positive constant depending
on the quantity $z$.

Our paper is organized as follows: in Section 2, we present the asymptotic analysis of the equation (\ref{transport});  in Section 3, we prove the well-posedness and decay of $\e$-Milne problem, i.e. Theorem \ref{main 1}; in Section 4, we prove the weighted $W^{1,\infty}$ estimates in $\e$-Milne problem, i.e. Theorem \ref{main 2}; finally, in appendix, we prove the improved $L^{\infty}$ estimate of remainder equation and the diffusive limit, i.e. Corollary \ref{main 3}.

\newpage

\section{Asymptotic Analysis}

\subsection{Interior Expansion}

We first try to approximate the solution of neutron transport equation (\ref{transport}). We define the interior expansion as follows:
\begin{eqnarray}\label{interior expansion}
\u(\vx,\vw)\sim\u_0(\vx,\vw)+\e\u_1(\vx,\vw)+\e^2\u_2(\vx,\vw),
\end{eqnarray}
where $\u_k$ can be determined by comparing the order of $\e$ by
plugging (\ref{interior expansion}) into the equation
(\ref{transport}). Thus we have
\begin{eqnarray}
\u_0-\bu_0&=&0,\label{expansion temp 1}\\
\u_1-\bu_1&=&-\vw\cdot\nx\u_0,\label{expansion temp 2}\\
\u_2-\bu_2&=&-\vw\cdot\nx\u_1.\label{expansion temp 3}
\end{eqnarray}
Plugging (\ref{expansion temp 1}) into (\ref{expansion temp 2}),
we obtain
\begin{eqnarray}
\u_1=\bu_1-\vw\cdot\nx\bu_0.\label{expansion temp 4}
\end{eqnarray}
Plugging (\ref{expansion temp 4}) into (\ref{expansion temp 3}),
we get
\begin{eqnarray}\label{expansion temp 5}
\u_2-\bu_2=-\vw\cdot\nx(\bu_1-\vw\cdot\nx\bu_0)=-\vw\cdot\nx\bu_1+w_1^2\p_{x_1x_1}\bu_0+w_2^2\p_{x_2x_2}\bu_0+2w_1w_2\p_{x_1x_2}\bu_0.
\end{eqnarray}
Integrating (\ref{expansion temp 5}) over $\vw\in\s^1$, we achieve
the final form
\begin{eqnarray}
\Delta_x\bu_0=0.
\end{eqnarray}
which further implies $\u_0(\vx,\vw)$ satisfies the equation
\begin{eqnarray}\label{interior 1}
\left\{ \begin{array}{rcl} \u_0&=&\bu_0,\\
\Delta_x\bu_0&=&0.
\end{array}
\right.
\end{eqnarray}
In a similar fashion, for $k=1,2$, $\u_k$ satisfies
\begin{eqnarray}\label{interior 2}
\left\{ \begin{array}{rcl} \u_k&=&\bu_k-\vw\cdot\nx\u_{k-1},\\
\Delta_x\bu_k&=&\displaystyle-\int_{\s^2}\vw\cdot\nx\u_{k-1}\ud{\vw}.\end{array}
\right.
\end{eqnarray}
It is easy to see $\bu_k$ satisfies an elliptic equation. However, the boundary condition of $\bu_k$ is unknown at this stage, since generally $\u_k$ does not necessarily satisfy the diffusive boundary condition of (\ref{transport}). Therefore, we have to resort to boundary layer.

\subsection{Local Coordinate System}

Basically, we use two types of coordinate systems: Cartesian coordinate
system for interior solution, which is stated above, and a local coordinate system in a neighborhood of the boundary for boundary layer. We need several substitution to describe solution near boundary.\\
\ \\
Substitution 1: spacial substitution:\\
We consider the three-dimensional transport operator $\vw\cdot\nx$. In the boundary surface, locally we can always define an orthogonal curvilinear coordinates system $(\tau_1,\tau_2)$ and the surface is described as $\vr(\tau_1,\tau_2)$. From the differential geometry, we know $\p_1\vr$ and $\p_2\vr$ denote two orthogonal tangential vectors. Then assume the outward unit normal vector is
\begin{eqnarray}
\vn=\frac{\p_1\vr\times\p_2\vr}{\abs{\p_1\vr\times\p_2\vr}}.
\end{eqnarray}
Here $\abs{\cdot}$ denotes the length and $\p_i$ denotes derivative with respect to $\tau_i$. Let
\begin{eqnarray}
P=\abs{\p_1\vr\times\p_2\vr}=\abs{\p_1\vr}\abs{\p_2\vr}=P_1P_2,
\end{eqnarray}
with the unit tangential vectors are
\begin{eqnarray}
\vt_1=\frac{\p_1\vr}{P_1},\ \ \vt_2=\frac{\p_2\vr}{P_2}.
\end{eqnarray}
Then in the new coordinates $(\mu,\tau_1,\tau_2)$ where $\mu$ denotes the normal distance to boundary surface, we have
\begin{eqnarray}
\vx=\vr-\mu\vn.
\end{eqnarray}
which further implies the operator becomes
\begin{eqnarray}
\vw\cdot\nx&=&-\frac{\bigg(\left(\p_1\vr-\mu\p_1\vn\right)\times\left(\p_2\vr-\mu\p_2\vn\right)\bigg)\cdot\vw}
{\bigg(\left(\p_1\vr-\mu\p_1\vn\right)\times\left(\p_2\vr-\mu\p_2\vn\right)\bigg)\cdot\vn}\frac{\p f}{\p\mu}\\
&&+\frac{\bigg(\left(\p_2\vr-\mu\p_2\vn\right)\times\vn\bigg)\cdot\vw}{\bigg(\left(\p_1\vr-\mu\p_1\vn\right)\times\left(\p_2\vr-\mu\p_2\vn\right)\bigg)\cdot\vn}\frac{\p f}{\p\tau_1}-\frac{\bigg(\left(\p_1\vr-\mu\p_1\vn\right)\times\vn\bigg)\cdot\vw}{\bigg(\left(\p_1\vr-\mu\p_1\vn\right)\times\left(\p_2\vr-\mu\p_2\vn\right)\bigg)\cdot\vn}\frac{\p f}{\p\tau_2}.\no
\end{eqnarray}
Based on differential geometry, we define the first fundamental form as $(E,F,G)$ and second fundamental form as $(L,M,N)$, then we have $F=M=0$ and the principal curvatures are given by
\begin{eqnarray}
\kappa_1=\frac{L}{E},\ \ \kappa_2=\frac{N}{G},
\end{eqnarray}
and also
\begin{eqnarray}
\p_1\vn=\kappa_1\p_1\vr,\ \ \p_2\vn=\kappa_2\p_2\vr.
\end{eqnarray}
Hence, we know
\begin{eqnarray}
\bigg(\left(\p_1\vr-\mu\p_1\vn\right)\times\left(\p_2\vr-\mu\p_2\vn\right)\bigg)\cdot\vn&=&\bigg(\kappa_1\kappa_2\mu^2-(\kappa_1+\kappa_2)\mu+1\bigg)P\\
&=&\bigg((\kappa_1\mu-1)(\kappa_2\mu-1)\bigg)P.\no
\end{eqnarray}
\begin{eqnarray}
\bigg(\left(\p_1\vr-\mu\p_1\vn\right)\times\left(\p_2\vr-\mu\p_2\vn\right)\bigg)\cdot\vw&=&\bigg(\kappa_1\kappa_2\mu^2-(\kappa_1+\kappa_2)\mu+1\bigg)P(\vn\cdot\vw)\\
&=&\bigg((\kappa_1\mu-1)(\kappa_2\mu-1)\bigg)P(\vw\cdot\vn).\no
\end{eqnarray}
\begin{eqnarray}
\bigg(\left(\p_2\vr-\mu\p_2\vn\right)\times\vn\bigg)\cdot\vw&=&(1-\kappa_2\mu)\frac{P_2}{P_1}(\vw\cdot\p_1\vr).
\end{eqnarray}
\begin{eqnarray}
\bigg(\left(\p_1\vr-\mu\p_1\vn\right)\times\vn\bigg)\cdot\vw&=&-(1-\kappa_1\mu)\frac{P_1}{P_2}(\vw\cdot\p_2\vr).
\end{eqnarray}
Hence, we have the transport operator as
\begin{eqnarray}
\vw\cdot\nx=-(\vw\cdot\vn)\frac{\p}{\p\mu}-\frac{\vw\cdot\vt_1}{P_1(\kappa_1\mu-1)}\frac{\p}{\p\tau_1}-\frac{\vw\cdot\vt_2}{P_2(\kappa_2\mu-1)}\frac{\p}{\p\tau_2}.
\end{eqnarray}
Therefore, under substitution $(x_1,x_2,x_3)\rt(\mu,\tau_1,\tau_2)$, the equation (\ref{transport}) is transformed into
\begin{eqnarray}
\left\{
\begin{array}{rcl}\displaystyle
\e\bigg(-(\vw\cdot\vn)\frac{\p u^{\e}}{\p\mu}-\frac{\vw\cdot\vt_1}{P_1(\kappa_1\mu-1)}\frac{\p u^{\e}}{\p\tau_1}-\frac{\vw\cdot\vt_2}{P_2(\kappa_2\mu-1)}\frac{\p u^{\e}}{\p\tau_2}\bigg)+u^{\e}-\bar u^{\e}=0\ \ \text{in}\ \ \Omega,\\\rule{0ex}{2.0em}
u^{\e}(0,\tau_1,\tau_2,\vw)=\pp[u^{\e}](0,\tau_1,\tau_2)+\e g(\tau_1,\tau_2,\vw)\ \ \text{for}\
\ \vw\cdot\vn<0,
\end{array}
\right.
\end{eqnarray}
where
\begin{eqnarray}
\pp[u^{\e}](0,\tau_1,\tau_2)=\frac{1}{2\pi}\int_{\vw\cdot\vn>0}u^{\e}(0,\tau_1,\tau_2,\vw)(\vw\cdot\vn)\ud{\vw},
\end{eqnarray}
\ \\
Substitution 2: velocity substitution.\\
Define the orthogonal velocity substitution
\begin{eqnarray}
\left\{
\begin{array}{rcl}
-\vw\cdot\vn&=&\sin\phi,\\
\vw\cdot\vt_1&=&\cos\phi\sin\psi,\\
\vw\cdot\vt_2&=&\cos\phi\cos\psi.
\end{array}
\right.
\end{eqnarray}
Then we have
\begin{eqnarray}
\frac{\p}{\p\tau_1}&\rt&\frac{\p}{\p\tau_1}-\kappa_1P_1\sin\psi\frac{\p}{\p\phi}
+\bigg(\frac{(\p_2\vr\times(\p_{21}\vr\times\p_2\vr))\cdot\vt_2}{P_2^3}-\kappa_1P_1\tan\phi\cos\psi\bigg)\frac{\p}{\p\psi},\\
\frac{\p}{\p\tau_2}&\rt&\frac{\p}{\p\tau_2}-\kappa_2P_2\cos\psi\frac{\p}{\p\phi}
+\bigg(\frac{(\p_1\vr\times(\p_{12}\vr\times\p_1\vr))\cdot\vt_1}{P_1^3}+\kappa_2P_2\tan\phi\sin\psi\bigg)\frac{\p}{\p\psi}.
\end{eqnarray}
Then the transport operator is as
\begin{eqnarray}
\vw\cdot\nx&=&\sin\phi\frac{\p}{\p\mu}-\bigg(\frac{\sin^2\psi}{R_1-\mu}+\frac{\cos^2\psi}{R_2-\mu}\bigg)\cos\phi\frac{\p}{\p\phi}\\
&&+\bigg(\frac{\cos\phi\sin\psi}{P_1(1-\kappa_1\mu)}\frac{\p}{\p\tau_1}+\frac{\cos\phi\cos\psi}{P_2(1-\kappa_2\mu)}\frac{\p}{\p\tau_2}\bigg)\no\\
&&+\bigg(\frac{\sin\psi}{1-\kappa_1\mu}(\vt_2\times(\p_{21}\vr\times\vt_2))\cdot\vt_2
+\frac{\cos\psi}{1-\kappa_2\mu}(\vt_1\times(\p_{12}\vr\times\vt_1))\cdot\vt_1\bigg)\frac{\cos\phi}{P_1P_2}\frac{\p}{\p\psi},\no
\end{eqnarray}
where $R_1=\dfrac{1}{\kappa_1}$ and $R_2=\dfrac{1}{\kappa_2}$. Hence, under substitution $(w_1,w_2,w_3)\rt(\phi,\psi)$,
the equation (\ref{transport}) is transformed into
\begin{eqnarray}\label{transport.}
\left\{
\begin{array}{l}\displaystyle
\e\sin\phi\frac{\p u^{\e}}{\p\mu}-\e\bigg(\dfrac{\sin^2\psi}{R_1-\mu}+\dfrac{\cos^2\psi}{R_2-\mu}\bigg)\cos\phi\dfrac{\p u^{\e}}{\p\phi}\\\rule{0ex}{2.0em}
+\e\bigg(\dfrac{\cos\phi\sin\psi}{P_1(1-\kappa_1\mu)}\dfrac{\p u^{\e}}{\p\tau_1}+\dfrac{\cos\phi\cos\psi}{P_2(1-\kappa_2\mu)}\dfrac{\p u^{\e}}{\p\tau_2}\bigg)\\\rule{0ex}{2.0em}
+\e\bigg(\dfrac{\sin\psi}{1-\kappa_1\mu}(\vt_2\times(\p_{21}\vr\times\vt_2))\cdot\vt_2
+\dfrac{\cos\psi}{1-\kappa_2\mu}(\vt_1\times(\p_{12}\vr\times\vt_1))\cdot\vt_1\bigg)\dfrac{\cos\phi}{P_1P_2}\dfrac{\p u^{\e}}{\p\psi}+u^{\e}-\bar u^{\e}=0,\\\rule{0ex}{2.0em}
u^{\e}(0,\tau_1,\tau_2,\phi,\psi)=\pp[u^{\e}](0,\tau_1,\tau_2)+\e g(\tau_1,\tau_2,\phi,\psi)\ \
\text{for}\ \ \sin\phi>0,
\end{array}
\right.
\end{eqnarray}
where
\begin{eqnarray}
\pp[u^{\e}](0,\tau_1,\tau_2)=\frac{1}{4\pi}\iint_{\sin\phi>0}u^{\e}(0,\tau_1,\tau_2,\phi,\psi)\sin\phi\cos\phi\ud{\phi}\ud{\psi},
\end{eqnarray}
due to Jacobian $J=\cos\phi$, in a neighborhood of the boundary. \\
\ \\
Substitution 3: scaling substitution.\\
We define the scaled variable $\eta=\dfrac{\mu}{\e}$, which implies $\dfrac{\p}{\p\mu}=\dfrac{1}{\e}\dfrac{\p}{\p\eta}$. Then, under the substitution $\mu\rt\eta$, the equation (\ref{transport}) is transformed into
\begin{eqnarray}\label{transport temp}
\left\{
\begin{array}{l}\displaystyle
\sin\phi\frac{\p u^{\e}}{\p\eta}-\e\bigg(\dfrac{\sin^2\psi}{R_1-\e\eta}+\dfrac{\cos^2\psi}{R_2-\e\eta}\bigg)\cos\phi\dfrac{\p u^{\e}}{\p\phi}\\\rule{0ex}{2.0em}
+\e\bigg(\dfrac{\cos\phi\sin\psi}{P_1(1-\e\kappa_1\eta)}\dfrac{\p u^{\e}}{\p\tau_1}+\dfrac{\cos\phi\cos\psi}{P_2(1-\e\kappa_2\eta)}\dfrac{\p u^{\e}}{\p\tau_2}\bigg)\\\rule{0ex}{2.0em}
+\e\bigg(\dfrac{\sin\psi}{1-\e\kappa_1\eta}(\vt_2\times(\p_{21}\vr\times\vt_2))\cdot\vt_2
+\dfrac{\cos\psi}{1-\e\kappa_2\eta}(\vt_1\times(\p_{12}\vr\times\vt_1))\cdot\vt_1\bigg)\dfrac{\cos\phi}{P_1P_2}\dfrac{\p u^{\e}}{\p\psi}+u^{\e}-\bar u^{\e}=0,\\\rule{0ex}{2.0em}
u^{\e}(0,\tau_1,\tau_2,\phi,\psi)=\pp[u^{\e}](0,\tau_1,\tau_2)+\e g(\tau_1,\tau_2,\phi,\psi)\ \
\text{for}\ \ \sin\phi>0,
\end{array}
\right.
\end{eqnarray}
where
\begin{eqnarray}
\pp[u^{\e}](0,\tau_1,\tau_2)=\frac{1}{4\pi}\iint_{\sin\phi>0}u^{\e}(0,\tau_1,\tau_2,\phi,\psi)\sin\phi\cos\phi\ud{\phi}\ud{\psi},
\end{eqnarray}
in a neighborhood of the boundary.

\subsection{Boundary Layer Expansion with Geometric Correction}

We define the boundary layer expansion as follows:
\begin{eqnarray}\label{boundary layer expansion}
\ub(\eta,\tau_1,\tau_2,\phi,\psi)\sim\ub_0(\eta,\tau_1,\tau_2,\phi,\psi)+\e\ub_1(\eta,\tau_1,\tau_2,\phi,\psi),
\end{eqnarray}
where $\ub_k$ can be defined by comparing the order of $\e$ via
plugging (\ref{boundary layer expansion}) into the equation
(\ref{transport temp}). Thus, in a neighborhood of the boundary, we have
\begin{eqnarray}
\sin\phi\frac{\p\ub_0}{\p\eta}-\e\bigg(\dfrac{\sin^2\psi}{R_1-\e\eta}+\dfrac{\cos^2\psi}{R_2-\e\eta}\bigg)\cos\phi\dfrac{\p \ub_0}{\p\phi}+\ub_0-\bub_0&=&0,\label{expansion temp 6}\\
\sin\phi\frac{\p\ub_1}{\p\eta}-\e\bigg(\dfrac{\sin^2\psi}{R_1-\e\eta}+\dfrac{\cos^2\psi}{R_2-\e\eta}\bigg)\cos\phi\dfrac{\p \ub_1}{\p\phi}+\ub_1-\bub_1&=&-G_0,\label{expansion temp 7}
\end{eqnarray}
where
\begin{eqnarray}
G_{0}&=&\bigg(\dfrac{\cos\phi\sin\psi}{P_1(1-\e\kappa_1\eta)}\dfrac{\p \ub_0}{\p\tau_1}+\dfrac{\cos\phi\cos\psi}{P_2(1-\e\kappa_2\eta)}\dfrac{\p \ub_0}{\p\tau_2}\bigg)\\
&&+\bigg(\dfrac{\sin\psi}{1-\e\kappa_1\eta}(\vt_2\times(\p_{21}\vr\times\vt_2))\cdot\vt_2
+\dfrac{\cos\psi}{1-\e\kappa_2\eta}(\vt_1\times(\p_{12}\vr\times\vt_1))\cdot\vt_1\bigg)\dfrac{\cos\phi}{P_1P_2}\dfrac{\p\ub_0}{\p\psi},\no
\end{eqnarray}
and
\begin{eqnarray}
\bub_k(\eta,\tau_1,\tau_2)=\frac{1}{4\pi}\int_{-\pi}^{\pi}\int_{-\pi/2}^{\pi/2}\ub_k(\eta,\tau_1,\tau_2,\phi,\psi)\cos\phi\ud{\phi}\ud{\psi}.
\end{eqnarray}
Note that this formulation is always valid locally on the boundary. By open covering theorem, we can find finite open domains to cover the whole surface. For convenience, we will not change our notation in each open domain.

\subsection{Matching of Interior Solution and Boundary Layer}

The bridge between the interior solution and the boundary layer
solution is the boundary condition of (\ref{transport}), so we
first consider the boundary condition expansion:
\begin{eqnarray}
(\u_0+\ub_0)&=&\pp[\u_0+\ub_0],\\
(\u_1+\ub_1)&=&\pp[\u_1+\ub_1]+g.
\end{eqnarray}
Note the fact that $\bu_k=\pp[\bu_k]$, we can simplify above
conditions as follows:
\begin{eqnarray}
\ub_0&=&\pp[\ub_0],\\
\ub_1&=&\pp[\ub_1]+(\vw\cdot\u_0-\pp[\vw\cdot\u_0])+g.
\end{eqnarray}
The construction of $\u_k$ and $\ub_k$ are as follows:\\
\ \\
Step 0: Preliminaries.\\
Assume the cut-off function $\Upsilon_0\in C^{\infty}[0,\infty)$ are defined as
\begin{eqnarray}\label{cut-off}
\Upsilon_0(\mu)=\left\{
\begin{array}{ll}
1&0\leq\mu\leq\dfrac{1}{4}R_{\min},\\
0&\dfrac{1}{2}R_{\min}\leq\mu\leq\infty.
\end{array}
\right.
\end{eqnarray}
where
\begin{eqnarray}
R_{\min}=\min_{\tau_1,\tau_2}\{R_1(\tau_1,\tau_2),R_2(\tau_1,\tau_2)\}.
\end{eqnarray}
Define the length of boundary layer $L=\e^{-n}$ for $0<n<\dfrac{2}{5}$ and the force as
\begin{eqnarray}\label{force}
F(\e;\eta,\psi)=-\e\bigg(\dfrac{\sin^2\psi}{R_1-\e\eta}+\dfrac{\cos^2\psi}{R_2-\e\eta}\bigg).
\end{eqnarray}
Also, denote $\rr\phi=-\phi$.\\
\ \\
Step 1: Construction of $\ub_0$.\\
Define the zeroth-order boundary layer as
\begin{eqnarray}\label{expansion temp 9}
\left\{
\begin{array}{rcl}
\ub_0(\eta,\tau_1,\tau_2,\phi,\psi)&=&\Upsilon_0(\e^{n}\eta)\bigg(f_0^{\e}(\eta,\tau_1,\tau_2,\phi,\psi)-f_{0,L}^{\e}(\tau_1,\tau_2)\bigg),\\
\sin\phi\dfrac{\p f_0^{\e}}{\p\eta}+F(\e;\eta,\psi)\cos\phi\dfrac{\p
f_0^{\e}}{\p\phi}+f_0^{\e}-\bar f_0^{\e}&=&0,\\\rule{0ex}{1em}
f_0^{\e}(0,\tau_1,\tau_2,\phi,\psi)&=&\pp[f_0^{\e}](0,\tau_1,\tau_2)\ \ \text{for}\ \
\sin\phi>0,\\\rule{0ex}{1em}
f_0^{\e}(L,\tau_1,\tau_2,\phi,\psi)&=&f_0^{\e}(L,\tau_1,\tau_2,\rr\phi,\psi),
\end{array}
\right.
\end{eqnarray}
with
\begin{eqnarray}
\pp[f_0^{\e}](0,\tau_1,\tau_2)=0.
\end{eqnarray}
By Theorem \ref{Milne theorem 3.}, $\ub_0$ is well-defined. It is
obvious to see $f_0^{\e}=f_{0,L}^{\e}=0$ is the only solution.\\
\ \\
Step 2: Construction of $\ub_1$ and $\u_0$.\\
Define the first-order boundary layer as
\begin{eqnarray}\label{expansion temp 10}
\left\{
\begin{array}{rcl}
\ub_1(\eta,\tau_1,\tau_2,\phi,\psi)&=&\Upsilon_0(\e^{n}\eta)\bigg(f_1^{\e}(\eta,\tau_1,\tau_2,\phi,\psi)-f_{1,L}^{\e}(\tau_1,\tau_2)\bigg),\\
\sin\phi\dfrac{\p f_1^{\e}}{\p\eta}+F(\e;\eta,\psi)\cos\phi\dfrac{\p
f_1^{\e}}{\p\phi}+f_1^{\e}-\bar
f_1^{\e}&=&-G_0,\\\rule{0ex}{1.0em} f_1^{\e}(0,\tau_1,\tau_2,\phi,\psi)&=&\pp
[f_1^{\e}](0,\tau_1,\tau_2)+g_1(\tau_1,\tau_2,\phi,\psi)\ \ \text{for}\ \
\sin\phi>0,\\\rule{0ex}{1em}
f_1^{\e}(L,\tau_1,\tau_2,\phi,\psi)&=&f_1^{\e}(L,\tau_1,\tau_2,\rr\phi,\psi),
\end{array}
\right.
\end{eqnarray}
with
\begin{eqnarray}
\pp[f_1^{\e}](0,\tau_1,\tau_2)=0,
\end{eqnarray}
where
\begin{eqnarray}
g_1&=&(\vw\cdot\nx\u_0(\vx_0)-\pp[\vw\cdot\nx\u_0(\vx_0)])+g,
\end{eqnarray}
with $\vx_0$ is the same boundary point as $(0,\tau_1,\tau_2)$.
To solve (\ref{expansion temp 10}), we require the compatibility
condition (\ref{Milne compatibility condition}) for the boundary
data
\begin{eqnarray}
\\
\iint_{\sin\phi>0}\bigg(g+\vw\cdot\nx\u_0(\vx_0)-\pp[\vw\cdot\nx\u_0(\vx_0)]\bigg)\sin\phi\cos\phi\ud{\phi}\ud{\psi}
-\int_0^{L}\int_{-\pi}^{\pi}\int_{0}^{\pi/2}\ue^{-V(s)}G_0\cos\phi\ud{\phi}\ud{\psi}\ud{s}\nonumber\\
=0,&&\nonumber
\end{eqnarray}
where $V(0)=0$ and $\dfrac{\p V}{\p\eta}=-F$. Note the fact $\vw=(\sin\phi)\vn+(\cos\phi\sin\psi)\vt_1+(\cos\phi\cos\psi)\vt_2$ and
\begin{eqnarray}
&&\iint_{\sin\phi>0}\bigg(\vw\cdot\nx\u_0(\vx_0)-\pp[\vw\cdot\nx\u_0(\vx_0)]\bigg)\sin\phi\cos\phi\ud{\phi}\ud{\psi}\\
&=&
\iint_{\sin\phi>0}(\vw\cdot\nx\u_0(\vx_0))\sin\phi\cos\phi\ud{\phi}\ud{\psi}-2\pi\pp[\vw\cdot\nx\u_0(\vx_0)]\nonumber\\
&=&\iint_{\sin\phi>0}(\vw\cdot\nx\u_0(\vx_0))\sin\phi\cos\phi\ud{\phi}\ud{\psi}+\iint_{\sin\phi<0}(\vw\cdot\nx\u_0(\vx_0))\sin\phi\cos\phi\ud{\phi}\ud{\psi}\nonumber\\
&=&\int_{-\pi}^{\pi}\int_{-\pi/2}^{\pi/2}(\vw\cdot\nx\u_0(\vx_0))\sin\phi\cos\phi\ud{\phi}\ud{\psi}\nonumber\\
&=&\int_{-\pi}^{\pi}\int_{-\pi/2}^{\pi/2}(\vn\cdot\nx\u_0(\vx_0))\sin\phi\cos\phi\ud{\phi}\ud{\psi}\nonumber\\
&=&-\pi^2\nx\bu_0(\vx_0)\cdot\vn=-\pi^2\frac{\p\bu_0(\vx_0)}{\p\vn}.\nonumber
\end{eqnarray}
We can simplify the compatibility condition as follows:
\begin{eqnarray}
\iint_{\sin\phi>0}g(\tau_1,\tau_2,\phi,\psi)\sin\phi\cos\phi\ud{\phi}\ud{\psi}-\pi^2\frac{\p\bu_0(\vx_0)}{\p\vn} -\int_0^{L}\int_{-\pi}^{\pi}\int_{-\pi/2}^{\pi/2}\ue^{-V(s)}G_0\cos\phi\ud{\phi}\ud{\psi}\ud{s}=0.
\end{eqnarray}
Then we have
\begin{eqnarray}
\frac{\p\bu_0(\vx_0)}{\p\vec
n}&=&\frac{1}{\pi^2}\iint_{\sin\phi>0}g(\tau_1,\tau_2,\phi,\psi)\sin\phi\cos\phi\ud{\phi}\ud{\psi}
+\frac{1}{\pi^2}\int_0^{L}\int_{-\pi}^{\pi}\int_{-\pi/2}^{\pi/2}\ue^{-V(s)}G_0\cos\phi\ud{\phi}\ud{\psi}\ud{s}\\
&=&\frac{1}{\pi^2}\iint_{\sin\phi>0}g(\tau_1,\tau_2,\phi,\psi)\sin\phi\cos\phi\ud{\phi}\ud{\psi}.\nonumber
\end{eqnarray}
Hence, we define the zeroth order interior solution $\u_0(\vx,\vw)$ as
\begin{eqnarray}\label{expansion temp 11}
\left\{
\begin{array}{rcl}
\u_0&=&\bu_0 ,\\\rule{0ex}{1em} \Delta_x\bu_0&=&0\ \ \text{in}\
\ \Omega,\\\rule{0ex}{1em}\dfrac{\p\bu_0}{\p\vn}&=&\displaystyle\dfrac{1}{\pi^2}\iint_{\sin\phi>0}g(\tau_1,\tau_2,\phi,\psi)\sin\phi\cos\phi\ud{\phi}\ud{\psi}\ \ \text{on}\ \
\p\Omega,\\\rule{0ex}{1em}
\ds\int_{\Omega}\bu_0(\vx_0)\ud{\vx_0}&=&0.
\end{array}
\right.
\end{eqnarray}
\ \\
Step 3: Construction of $\u_1$.\\
We do not expand the boundary layer to $\ub_2$ and just terminate at $\ub_1$. Then we define the first order interior solution $\u_1(\vx)$ as
\begin{eqnarray}\label{expansion temp 12.}
\left\{
\begin{array}{rcl}
\u_1&=&\bu_1-\vw\cdot\nx\u_0,\\\rule{0ex}{1em}
\Delta_x\bu_1&=&-\displaystyle\int_{\s^2}(\vw\cdot\nx\u_{0})\ud{\vw}\
\ \text{in}\ \ \Omega,\\\rule{0ex}{1.0em} \dfrac{\p\bu_1}{\p\vn}&=&-\displaystyle\int_{\Omega}\int_{\s^2}(\vw\cdot\nx\u_{0})\ud{\vw}\ud{\vx}\ \ \text{on}\ \
\p\Omega,\\\rule{0ex}{1.5em}
\displaystyle\int_{\Omega}\bu_1(\vx)\ud{\vx}&=&\ds\int_{\Omega}\int_{\s^2}(\vw\cdot\nx\u_0-\ub_1)\ud{\vw}\ud{\vx}.
\end{array}
\right.
\end{eqnarray}
Note that here we only require the trivial boundary condition since we cannot resort to the compatibility condition in $\e$-Milne problem with geometric correction. Based on \cite{AA003}, this might lead to $O(\e^2)$ error to the boundary approximation. Thanks to the improved remainder estimate, this error is acceptable.\\
\ \\
Step 4: Construction of $\u_2$.\\
By a similar fashion, we define the second order interior solution as
\begin{eqnarray}
\left\{
\begin{array}{rcl}
\u_{2}&=&\bu_{2}-\vw\cdot\nx\u_{1},\\\rule{0ex}{1em}
\Delta_x\bu_{2}&=&-\displaystyle\int_{\s^2}(\vw\cdot\nx\u_{1})\ud{\vw}\
\ \text{in}\ \ \Omega,\\\rule{0ex}{1.0em} \dfrac{\p\bu_{2}}{\p\vn}&=&-\displaystyle\int_{\Omega}\int_{\s^2}(\vw\cdot\nx\u_{1})\ud{\vw}\ud{\vx}\ \ \text{on}\ \
\p\Omega,\\\rule{0ex}{1.5em}
\displaystyle\int_{\Omega}\bu_2(\vx)\ud{\vx}&=&\ds\int_{\Omega}\int_{\s^2}\vw\cdot\nx\u_1\ud{\vw}\ud{\vx}.
\end{array}
\right.
\end{eqnarray}
As the case of $\u_1$, we might have $O(\e^3)$ error in this step due to the trivial boundary data. However, it will not affect the diffusive limit.

\newpage

\section{Well-Posedness and Decay of $\e$-Milne Problem}

We consider the $\e$-Milne problem for $f^{\e}(\eta,\phi,\psi)$ in
the domain $(\eta,\phi,\psi)\in[0,L]\times[-\pi/2,\pi/2]\times[-\pi,\pi]$ as
\begin{eqnarray}\label{Milne problem}
\left\{ \begin{array}{rcl}\displaystyle \sin\phi\frac{\p
f^{\e}}{\p\eta}+F(\e;\eta,\psi)\cos\phi\frac{\p
f^{\e}}{\p\phi}+f^{\e}-\bar f^{\e}&=&S^{\e}(\eta,\phi,\psi),\\
f^{\e}(0,\phi,\psi)&=& h^{\e}(\phi,\psi)\ \ \text{for}\
\ \sin\phi>0,\\
f^{\e}(L,\phi,\psi)&=&f^{\e}(L,\rr\phi,\psi),
\end{array}
\right.
\end{eqnarray}
where
\begin{eqnarray}
\bar f^{\e}(\eta)=\frac{1}{4\pi}\int_{-\pi}^{\pi}\int_{-\pi/2}^{\pi/2}f^{\e}(\eta,\phi,\psi)\cos\phi\ud{\phi}\ud{\psi}
\end{eqnarray}
in which $\cos\phi$ shows up as the Jacobian of spherical coordinates in integration. Note that for $\phi\in[-\pi/2,\pi/2]$, we always have $\cos\phi\geq0$, which means this will not destroy the positivity of integral. Also, we have
\begin{eqnarray}
F(\e;\eta,\psi)=-\e\bigg(\dfrac{\sin^2\psi}{R_1-\e\eta}+\dfrac{\cos^2\psi}{R_2-\e\eta}\bigg).
\end{eqnarray}
for $R_1$ and $R_2$ radium of two principle curvature, and $L=\e^{-n}$ for some $n>0$ which will be specified later.

In this section, for convenience, we temporarily ignore the superscript $\e$. Note that all the estimates we get will be uniform in $\e$. We define the norms in the
space $(\eta,\phi,\psi)\in[0,L]\times[-\pi/2,\pi/2]\times[-\pi,\pi)$ as follows:
\begin{eqnarray}
\tnnm{f}&=&\bigg(\int_0^{L}\int_{-\pi}^{\pi}\int_{-\pi/2}^{\pi/2}\abs{f(\eta,\phi,\psi)}^2\cos\phi\ud{\phi}\ud{\psi}\ud{\eta}\bigg)^{1/2},\\
\lnnm{f}&=&\sup_{(\eta,\phi,\psi)}\abs{f(\eta,\phi,\psi)}.
\end{eqnarray}
Also, we define the inner product as
\begin{eqnarray}
\br{f,g}(\eta)&=&\int_{-\pi}^{\pi}\int_{-\pi/2}^{\pi/2}f(\eta,\phi,\psi)g(\eta,\tau_1,\tau_2,\phi,\psi)\cos\phi
\ud{\phi}\ud{\psi}.
\end{eqnarray}
Similarly, we can define the norm at in-flow boundary as
\begin{eqnarray}
\tnm{f}(\eta)&=&\bigg(\iint_{\sin\phi>0}\abs{f(\eta,\phi,\psi)}^2\cos\phi\ud{\phi}\ud{\psi}\bigg)^{1/2},\\
\lnm{f}(\eta)&=&\sup_{(\phi,\psi)\ \text{with}\ \sin\phi>0}\abs{f(\eta,\phi,\psi)},
\end{eqnarray}
We further assume
\begin{eqnarray}\label{Milne bounded}
\lnm{h}\leq M,
\end{eqnarray}
and
\begin{eqnarray}\label{Milne decay}
\lnnm{\ue^{K\eta}S}&\leq& M,\no
\end{eqnarray}
for $M>0$ and $K>0$ uniform in $\e$.

\subsection{$L^2$ Estimates}

\subsubsection{$L^2$ Estimates when $\bar S=0$}

Consider the equation
\begin{eqnarray}\label{LT Milne problem}
\left\{ \begin{array}{rcl}\displaystyle \sin\phi\frac{\p
f}{\p\eta}+F(\eta,\psi)\cos\phi\frac{\p
f}{\p\phi}+f-\bar f&=&S(\eta,\phi,\psi),\\
f(0,\phi,\psi)&=& h(\phi,\psi)\ \ \text{for}\
\ \sin\phi>0,\\
f(L,\phi,\psi)&=&f(L,\rr\phi,\psi).
\end{array}
\right.
\end{eqnarray}
where
\begin{eqnarray}
F(\eta,\psi)\cos\phi\frac{\p f}{\p\phi}=\tf(\eta)\cos\phi\frac{\p
f}{\p\phi}+G(\eta)\cos^2\psi\cos\phi\frac{\p
f}{\p\phi},
\end{eqnarray}
for
\begin{eqnarray}
\tf(\eta)&=&-\frac{\e}{R_1-\e\eta},\\
G(\eta)&=&-\frac{\e(R_1-R_2)}{(R_1-\e\eta)(R_2-\e\eta)}.
\end{eqnarray}
Also, $\rr\phi=-\phi$.
We may decompose the solution
\begin{eqnarray}
f(\eta,\phi,\psi)=q(\eta)+r(\eta,\phi,\psi),
\end{eqnarray}
where the hydrodynamical part $q$ is in the null space of the
operator $f-\bar f$, and the microscopic part $r$ is
the orthogonal complement, i.e.
\begin{eqnarray}\label{hydro}
q(\eta)=\frac{1}{4\pi}\int_{-\pi}^{\pi}\int_{-\pi/2}^{\pi/2}f(\eta,\phi,\psi)\cos\phi\ud{\phi}\ud{\psi},\quad
r(\eta,\phi,\psi)=f(\eta,\phi,\psi)-q(\eta).
\end{eqnarray}
Furthermore, we define a potential function $\tv(\eta)$ satisfying $\tv(0)=0$ and $\dfrac{\p \tv}{\p\eta}=-\tf(\eta)$. It is easy to compute $\tv(\eta)=\ln\left(\dfrac{R_1}{R_1-\e\eta}\right)$.

\begin{lemma}\label{Milne finite LT}
Assume $\bar S=0$ satisfying (\ref{Milne bounded}) and (\ref{Milne decay}) and $0<n<\dfrac{2}{5}$. There exists a solution $f(\eta,\phi,\psi)$ to the equation
(\ref{LT Milne problem}), satisfying for some constant $\abs{f_L}\leq C$,
\begin{eqnarray}\label{LT estimate}
\tnnm{f-f_L}\leq C.
\end{eqnarray}
The solution is unique among functions such that (\ref{LT estimate}) holds.
\end{lemma}
\begin{proof}
As in \cite[Section 6]{AA003}, the existence can be proved using a standard approximation argument, so we will only focus on the a priori estimates. We divide the proof into several steps:\\
\ \\
Step 1: $r$ Estimates.\\
Multiplying $f\cos\phi$
on both sides of (\ref{LT Milne problem}) and
integrating over $(\phi,\psi)\in[-\pi/2,\pi/2]\times[-\pi,\pi)$, we get the energy estimate
\begin{eqnarray}\label{mt 01}
\half\frac{\ud{}}{\ud{\eta}}\br{f
,f\sin\phi}(\eta)&=&-\tnm{r(\eta)}^2-\tf(\eta)\br{\frac{\p
f}{\p\phi},f\cos\phi}(\eta)\\
&&-G(\eta)\br{\frac{\p
f}{\p\phi}\cos^2\psi,f\cos\phi}(\eta)+\br{S,f}(\eta).\no
\end{eqnarray}
A further integration by parts in $\phi$ reveals
\begin{eqnarray}
-\tf(\eta)\br{\frac{\p
f}{\p\phi},f\cos\phi}(\eta)&=&-\tf(\eta)
\br{f,f\sin\phi}(\eta),\\
-G(\eta)\br{\frac{\p
f}{\p\phi}\cos^2\psi,f\cos\phi}(\eta)&=&-G(\eta)\br{f\cos^2\psi,f\sin\phi}(\eta).
\end{eqnarray}
Hence, we can simplify (\ref{mt 01}) as
\begin{eqnarray}\label{mt 02}
\half\frac{\ud{}}{\ud{\eta}}\br{
f,f\sin\phi}(\eta)&=&-\tnm{r(\eta)}^2-
\tf(\eta)\br{
f,f\sin\phi}(\eta)\\
&&-G(\eta)\br{f\cos^2\psi,f\sin\phi}(\eta)+\br{S,f}(\eta).\no
\end{eqnarray}
Define
\begin{eqnarray}
\alpha(\eta)=\half\br{f,f\sin\phi}(\eta).
\end{eqnarray}
Then (\ref{mt 02}) can be rewritten as
\begin{eqnarray}
\frac{\ud{\alpha}}{\ud{\eta}}=-\tnm{r(\eta)}^2-
2\tf(\eta)\alpha(\eta)-G(\eta)\br{f\cos^2\psi,f\sin\phi}(\eta)+\br{S,f}(\eta).
\end{eqnarray}
We can solve above in $[\eta,L]$ and $[0,\eta]$ respectively to obtain
\begin{eqnarray}
\label{mt 03}
\\
\alpha(\eta)&=&\ue^{2\tv(\eta)-2\tv(L)}\alpha(L)+\int_{\eta}^L\ue^{2\tv(\eta)-2\tv(y)}
\bigg(\tnm{r(y)}^2+G(y)\br{f\cos^2\psi,f\sin\phi}(y)-\br{S,f}(y)\bigg)\ud{y},\no\\
\label{mt 04}
\\
\alpha(\eta)&=&\ue^{2\tv(\eta)}\alpha(0)+\int_{0}^{\eta}\ue^{2\tv(\eta)-2\tv(y)}
\bigg(-\tnm{r(y)}^2-G(y)\br{f\cos^2\psi,f\sin\phi}(y)+\br{S,f}(y)\bigg)\ud{y}.\no
\end{eqnarray}
The specular reflexive boundary $f(L,\phi)=f(L,\rr\phi)$
ensures $\alpha(L)=0$. Hence, based on (\ref{mt 03}), we have
\begin{eqnarray}
\alpha(\eta)\geq\int_{\eta}^L\ue^{2\tv(\eta)-2\tv(y)}\bigg(G(y)\br{f\cos^2\psi,f\sin\phi}(y)-\br{S,f}(y)\bigg)\ud{y}.
\end{eqnarray}
Also, (\ref{mt 04}) implies
\begin{eqnarray}
\alpha(\eta)&\leq&C\alpha(0)+\int_{0}^{\eta}\ue^{2\tv(\eta)-2\tv(y)}
\bigg(-G(y)\br{f\cos^2\psi,f\sin\phi}(y)+\br{S,f}(y)\bigg)\ud{y}\\
&\leq&C\tnm{h}^2+\int_{0}^{\eta}\ue^{2\tv(\eta)-2\tv(y)}
\bigg(-G(y)\br{f\cos^2\psi,f\sin\phi}(y)+\br{S,f}(y)\bigg)\ud{y}\nonumber,
\end{eqnarray}
due to the fact
\begin{eqnarray}
\alpha(0)=\half\br{f\sin\phi
,f}(0)\leq\half\bigg(\int_{\sin\phi>0}h(\phi)^2\sin\phi\cos\phi
\ud{\phi}\bigg)\leq \half\tnm{h}^2.
\end{eqnarray}
Then in (\ref{mt 04}) taking $\eta=L$, from $\alpha(L)=0$, we have
\begin{eqnarray}
&&\int_{0}^L\ue^{-2\tv(y)}\tnm{r(y)}^2\ud{y}\\
&\leq&\alpha(0)+\int_{0}^L\ue^{-2\tv(y)}\bigg(-G(y)\br{f\cos^2\psi,f\sin\phi}(y)+\br{S,f}(y)\bigg)\ud{y}\no\\
&\leq&
C\tnm{h}^2+\int_{0}^L\ue^{-2\tv(y)}\bigg(-G(y)\br{f\cos^2\psi,f\sin\phi}(y)+\br{S,f}(y)\bigg)\ud{y}\nonumber.
\end{eqnarray}
On the other hand, we can directly
estimate
\begin{eqnarray}
\int_{0}^L\ue^{-2\tv(y)}\tnm{r(y)}^2\ud{y}\geq C\tnnm{r}^2.
\end{eqnarray}
Combining above yields
\begin{eqnarray}
\tnnm{r}^2&\leq&
C\Bigg(\tnm{h}^2+\int_{0}^L\ue^{-2\tv(y)}\bigg(-G(y)\br{f\cos^2\psi,f\sin\phi}(y)+\br{S,f}(y)\bigg)\ud{y}\Bigg).
\end{eqnarray}
Since $\br{S,f}=\br{S,r}$ due to $\bar S=0$,
by Cauchy's inequality, we have
\begin{eqnarray}
\abs{\int_{0}^L\ue^{-2\tv(y)}\br{S,r}(y)\ud{y}}&\leq&C'\tnnm{r}^2+C\tnnm{S}^2,
\end{eqnarray}
for $C'>0$ sufficiently small. Therefore, absorbing $C'\tnnm{r}^2$ term, we deduce
\begin{eqnarray}
\tnnm{r}^2&\leq&C\left(
\tnm{h}^2+\tnnm{S}^2+\abs{\int_{0}^LG(y)\br{f\cos^2\psi,f\sin\phi}(y)\ud{y}}\right)\\
&\leq&C\left(1+\abs{\int_0^LG(y)\br{f\cos^2\psi,f\sin\phi}(y)\ud{y}}\right)\no\\
&\leq&C\left(1+\lnnm{G}\tnnm{f}^2\right)\no\\
&\leq&C\left(1+\e\tnnm{f}^2\right).\no
\end{eqnarray}
Note that this estimate is not closed since it depends on $f$.\\
\ \\
Step 2: Quasi-Orthogonality relation.\\
Multiplying $\cos\phi$ on both sides of (\ref{LT Milne problem}) and integrating over $(\phi,\psi)\in[-\pi/2,\pi/2)\times[-\pi,\pi)$ imply
\begin{eqnarray}
\frac{\ud{}}{\ud{\eta}}\br{\sin\phi,f}(\eta)&=&-\tf\br{\cos\phi,\frac{\ud{f}}{\ud{\phi}}}(\eta)
-G\br{\cos\phi\cos^2\psi,\frac{\ud{f}}{\ud{\phi}}}(\eta)
+\bar S(\eta)\\
&=&-2\tf\br{\sin\phi,f}(\eta)-2G\br{\sin\phi\cos^2\psi,f}(\eta).\no
\end{eqnarray}
The specular reflexive boundary
$f(L,\phi)=f(L,\rr\phi)$ implies
$\br{\sin\phi,f}(L)=0$. Then we have
\begin{eqnarray}
\br{\sin\phi,f}(\eta)=-2\int_{\eta}^L\ue^{2\tv(\eta)-\tv(y)}G(y)\br{\sin\phi\cos^2\psi,f}(y)\ud{y}.
\end{eqnarray}
It is easy to see
\begin{eqnarray}
\br{\sin\phi,q}(\eta)=0.
\end{eqnarray}
Hence, we may derive
\begin{eqnarray}
\br{\sin\phi,r}(\eta)&=&-2\int_{\eta}^L\ue^{2\tv(\eta)-\tv(y)}G(y)\br{\sin\phi\cos^2\psi,f}(y)\ud{y},\\
&=&-2\int_{\eta}^L\ue^{2\tv(\eta)-\tv(y)}G(y)\br{\sin\phi\cos^2\psi,r}(y)\ud{y}.\no
\end{eqnarray}
\ \\
Step 3: $q$ Estimates.\\
Multiplying $\sin\phi\cos\phi$ on
both sides of (\ref{LT Milne problem}) and
integrating over $(\phi,\psi)\in[-\pi/2,\pi/2)\times[-\pi,\pi)$ lead to
\begin{eqnarray}
\label{mt 05}
\frac{\ud{}}{\ud{\eta}}\br{\sin^2\phi,f}(\eta)&=&
-\br{\sin\phi,r}(\eta)-\tf(\eta)\br{\sin\phi\cos\phi,\frac{\p
f}{\p\phi}}(\eta)\\
&&-G(\eta)\br{\sin\phi\cos\phi\cos^2\psi,\frac{\p
f}{\p\phi}}(\eta)+\br{\sin\phi,S}(\eta).\no
\end{eqnarray}
We can further integrate by parts in $\phi$ as
\begin{eqnarray}
-\tf(\eta)\br{\sin\phi\cos\phi,\frac{\p
f}{\p\phi}}(\eta)&=&\tf(\eta)\br{1-3\sin^2\phi,f}(\eta)=\tf(\eta)\br{1-3\sin^2\phi,r}(\eta),\\
\\
-G(\eta)\br{\sin\phi\cos\phi\cos^2\psi,\frac{\p
f}{\p\phi}}(\eta)&=&G(\eta)\br{1-3\sin^2\phi,f\cos^2\psi}(\eta)=G(\eta)\br{1-3\sin^2\phi,r\cos^2\psi}(\eta),\no
\end{eqnarray}
to obtain
\begin{eqnarray}
\frac{\ud{}}{\ud{\eta}}\br{\sin^2\phi,f}(\eta)&=&
-\br{\sin\phi,r}(\eta)+\tf(\eta)\br{1-3\sin^2\phi,r}(\eta)\\
&&+G(\eta)\br{1-3\sin^2\phi,r\cos^2\psi}(\eta)+\br{\sin\phi,S}(\eta).\no
\end{eqnarray}
Define
\begin{eqnarray}
\beta(\eta)=\br{\sin^2\phi,f}(\eta).
\end{eqnarray}
Then we can simplify (\ref{mt 05}) as
\begin{eqnarray}\label{mt 06}
\frac{\ud{\beta}}{\ud{\eta}}=D(\eta),
\end{eqnarray}
where
\begin{eqnarray}
D(\eta)&=&-\br{\sin\phi,r}(\eta)+\tf(\eta)\br{1-3\sin^2\phi,r}(\eta)\\
&&+G(\eta)\br{1-3\sin^2\phi,r\cos^2\psi}(\eta)+\br{\sin\phi,S}(\eta).\no
\end{eqnarray}
We can integrate over $[0,\eta]$ in (\ref{mt 06}) to obtain
\begin{eqnarray}
\beta(\eta)=\beta(0)+\int_0^{\eta}D(y)\ud{y}.
\end{eqnarray}
The quasi-orthogonal relation implies
\begin{eqnarray}
D(\eta)&=&2\int_{\eta}^L\ue^{2\tv(\eta)-\tv(y)}G(y)\br{\sin\phi\cos^2\psi,f}(y)\ud{y}+\tf(\eta)\br{1-3\sin^2\phi,r}(\eta)\\
&&+G(\eta)\br{1-3\sin^2\phi,r\cos^2\psi}(\eta)+\br{\sin\phi,S}(\eta).\no
\end{eqnarray}
Hence, we deduce
\begin{eqnarray}
\beta(\eta)-\beta(0)&=&2\int_0^{\eta}\int_{z}^L\ue^{2\tv(z)-2\tv(y)}G(y)\br{\sin\phi\cos^2\psi,r}\ud{y}\ud{z}
+\int_0^{\eta}\tf(y)\br{1-3\sin^2\phi,r}(y)\ud{y}\\
&&+\int_0^{\eta}G(y)\br{1-3\sin^2\phi,r\cos^2\psi}(y)\ud{y}+\int_0^{\eta}\br{\sin\phi,S}(y)\ud{y}.\no
\end{eqnarray}
For the boundary data,
\begin{eqnarray}
\beta(0)=\br{\sin^2\phi,f}(0)\leq \bigg(\br{
f,f\abs{\sin\phi}}(0)\bigg)^{1/2}\tnm{\sin\phi}^{3/2}\leq C
\bigg(\br{ f,f\abs{\sin\phi}}(0)\bigg)^{1/2}.
\end{eqnarray}
Obviously, we have
\begin{eqnarray}
\br{f, f\abs{\sin\phi}
}(0)=\int_{\sin\phi>0}h^2(\phi)\sin\phi\cos\phi
\ud{\phi}-\int_{\sin\phi<0}\bigg(f(0,\phi)\bigg)^2\sin\phi\cos\phi\ud{\phi}.
\end{eqnarray}
However, based on the definition of $\alpha(\eta)$, we can obtain
\begin{eqnarray}
&&\int_{\sin\phi>0} h^2(\phi)\sin\phi\cos\phi\ud{\phi}+\int_{\sin\phi<0}
\bigg(f(0,\phi)\bigg)^2\sin\phi\cos\phi\ud{\phi}=2\alpha(0)\\
&\geq&
2\int_{0}^L\ue^{-2\tv(y)}\bigg(G(y)\br{f\cos^2\psi,f\sin\phi}(y)-\br{S,r}(y)\bigg)\ud{y}.\nonumber
\end{eqnarray}
Hence, we can deduce
\begin{eqnarray}
&&-\int_{\sin\phi<0}
\bigg(f(0,\phi)\bigg)^2\sin\phi\cos\phi\ud{\phi}\\
&\leq&\int_{\sin\phi>0}
h^2(\phi)\sin\phi\cos\phi\ud{\phi}-2\int_0^L\ue^{-2\tv(y)}\bigg(G(y)\br{f\cos^2\psi,f\sin\phi}(y)-\br{S,r}(y)\bigg)\ud{y}\no\\
&\leq&
\tnm{h}^2+C\abs{\int_0^L\bigg(G(y)\br{f\cos^2\psi,f\sin\phi}(y)-\br{S,r}(y)\bigg)\ud{y}}\nonumber.
\end{eqnarray}
Hence, we obtain
\begin{eqnarray}
\Big(\beta(0)\Big)^2&\leq& \br{f,f\abs{\sin\phi}}(0)\leq
C\tnm{h}^2+C\abs{\int_0^L\bigg(G(y)\br{f\cos^2\psi,f\sin\phi}(y)-\br{S,r}(y)\bigg)\ud{y}}.
\end{eqnarray}
Note that $\lnnm{G}\leq C\e$. Since $\tf\in L^1[0,L]\cap
L^2[0,L]$, $r\in L^2([0,L]\times[-\pi,\pi))$, and $S$
exponentially decays, using Cauchy's inequality, we have
\begin{eqnarray}
\abs{\beta(L)}&\leq&C\tnm{h}+C\abs{\int_0^LG(y)\br{f\cos^2\psi,f\sin\phi}(y)\ud{y}}^{1/2}+C\abs{\int_0^L\br{S,r}(y)\ud{y}}^{1/2}\\
&&+\abs{\int_0^L\tf(y)\br{1-3\sin^2\phi,r}(y)\ud{y}}
+\abs{\int_0^{L}\int_{z}^L\ue^{2\tv(z)-2\tv(y)}G(y)\br{\sin\phi\cos^2\psi,r}\ud{y}\ud{z}}\no\\
&&\abs{\int_0^{L}G(y)\br{1-3\sin^2\phi,r\cos^2\psi}(y)\ud{y}}+\abs{\int_0^{L}\br{\sin\phi,S}(y)\ud{y}}\no\\
&\leq&C\abs{\int_0^LG(y)\br{f\cos^2\psi,f\sin\phi}(y)\ud{y}}^{1/2}+C\tnm{h}+C\tnnm{r}\no\\
&&+C\tnnm{\tf}
\tnnm{r}+\tnnm{G}\tnnm{r}+L\tnnm{G}\tnnm{r}+\tnnm{S}\no\\
&\leq&C+C(1+\e^{1-\frac{3n}{2}})\tnnm{r}+\abs{\int_0^LG(y)\br{f\cos^2\psi,f\sin\phi}(y)\ud{y}}^{1/2}\no\\
&\leq&C(1+\e^{1-\frac{3n}{2}})\left(1+\abs{\int_0^LG(y)\br{f\cos^2\psi,f\sin\phi}(y)\ud{y}}^{1/2}\right)\no\\
&\leq&C(1+\e^{1-\frac{3n}{2}})\left(1+\e^{\frac{1}{2}}\tnnm{f}\right)\no\\
&\leq&C(1+\e^{1-\frac{3n}{2}})\left(1+\e^{\frac{1}{2}}\tnnm{f-f_L}+\e^{\frac{1}{2}}\tnnm{f_L}\right)\no\\
&\leq&C(1+\e^{1-\frac{3n}{2}})\left(1+\e^{\frac{1}{2}}\tnnm{f-f_L}+\e^{\frac{1}{2}-\frac{n}{2}}\abs{f_L}\right)\no,
\end{eqnarray}
where we define
\begin{eqnarray}
f_L=q_L=\frac{\beta(L)}{\tnm{\sin\phi}^2}.
\end{eqnarray}
Therefore, for $0<n<\dfrac{2}{3}$ and $\e$ sufficiently small, absorbing $\abs{f_L}$, we have
\begin{eqnarray}
\abs{f_L}\leq C+C\e^{\frac{1}{2}}\tnnm{f-f_L}.
\end{eqnarray}
Thus, we naturally have
\begin{eqnarray}
\tnnm{r}&\leq&C\left(1+\e^{\frac{1}{2}}\tnnm{f}\right)\\
&\leq&C\left(1+\e^{\frac{1}{2}}\tnnm{f-f_L}+\e^{\frac{1}{2}}\tnnm{f_L}\right)\no\\
&\leq&C\left(1+\e^{\frac{1}{2}}\tnnm{f-f_L}\right)\no,
\end{eqnarray}
Furthermore, we have
\begin{eqnarray}
\beta(L)-\beta(\eta)&=&\int_{\eta}^LD(y)\ud{y}\\
&=&\abs{\int_{\eta}^L\tf(y)\br{1-3\sin^2\phi,r}(y)\ud{y}}
+\abs{\int_{\eta}^{L}\int_{z}^L\ue^{2\tv(z)-2\tv{y}}G(y)\br{\sin\phi\cos^2\psi,r}\ud{y}\ud{z}}\no\\
&&\abs{\int_{\eta}^{L}G(y)\br{1-3\sin^2\phi,r\cos^2\psi}(y)\ud{y}}+\abs{\int_{\eta}^{L}\br{\sin\phi,S}(y)\ud{y}}\no
\end{eqnarray}
Note
\begin{eqnarray}
\beta(\eta)=\br{\sin^2\phi,f}(\eta)=\br{\sin^2\phi,q}(\eta)+\br{\sin^2\phi,r}(\eta)
=q(\eta)\tnm{\sin\phi}^2+\br{\sin^2\phi,r}(\eta).
\end{eqnarray}
Thus considering $S$ decays exponentially, we can estimate
\begin{eqnarray}
\tnnm{q-q_L}^2&\leq&C\tnnm{r}^2+C\tnnm{\beta(\eta)-\beta(L)}^2\\
&\leq&C\tnnm{r}^2+\int_0^L\abs{\int_{\eta}^L\tf(y)\br{1-3\sin^2\phi,r}(y)\ud{y}}^2\ud{\eta}\no\\
&&+\int_0^L\abs{\int_{\eta}^{L}\int_{z}^L\ue^{2\tv(z)-2\tv{y}}G(y)\br{\sin\phi\cos^2\psi,r}\ud{y}\ud{z}}^2\ud{\eta}\no\\
&&+\int_0^L\abs{\int_{\eta}^{L}G(y)\br{1-3\sin^2\phi,r\cos^2\psi}(y)\ud{y}}^2\ud{\eta}+\int_0^L\abs{\int_{\eta}^{L}\br{\sin\phi,S}(y)\ud{y}}^2\ud{\eta}\no\\
&\leq&C\tnnm{r}^2+\tnnm{r}^2\int_0^L\int_{\eta}^L\abs{\tf(y)}^2\ud{y}\ud{\eta}+L^3\tnnm{G}^2\tnnm{r}^2+L\tnnm{G}^2\tnnm{r}^2\no\\
&&+\int_0^L\abs{\int_{\eta}^{L}S(y)\ud{y}}^2\ud{\eta}\no\\
&\leq&C(1+\e^{2-5n}\tnnm{r}^2)\no\\
&\leq&C(1+\e^{2-5n})\left(1+\e^{\frac{1}{2}}\tnnm{f-f_L}\right).\no
\end{eqnarray}
Therefore, for $0<n<\dfrac{2}{5}$, we have
\begin{eqnarray}
\tnnm{q-q_L}\leq C\left(1+\e^{\frac{1}{2}}\tnnm{f-f_L}\right).
\end{eqnarray}
\ \\
Step 4: Synthesis.\\
For $0<n<\dfrac{2}{5}$, we have
\begin{eqnarray}
\tnnm{r}&\leq& C\left(1+\e^{\frac{1}{2}}\tnnm{f-f_L}\right),\\
\tnnm{q-q_L}&\leq& C\left(1+\e^{\frac{1}{2}}\tnnm{f-f_L}\right),
\end{eqnarray}
which further implies
\begin{eqnarray}
\tnnm{f-f_L}\leq \tnnm{r}+\tnnm{q-q_L}\leq C\left(1+\e^{\frac{1}{2}}\tnnm{f-f_L}\right).
\end{eqnarray}
Hence, for $\e$ sufficiently small, we have $\abs{f_L}\leq C$ and
\begin{eqnarray}
\tnnm{f-f_L}\leq C.
\end{eqnarray}
In order to show the uniqueness of the solution, we assume there are
two solutions $f_1$ and $f_2$ to the equation (\ref{LT Milne problem}) satisfying above estimates. Then $f'=f_1-f_2$ satisfies the equation
\begin{eqnarray}
\left\{
\begin{array}{rcl}\displaystyle
\sin\phi\frac{\p f'}{\p\eta}+F(\eta)\cos\phi\frac{\p
f'}{\p\phi}+f'-\bar f'&=&0,\\
f'(0,\phi,\psi)&=&0\ \ \text{for}\ \ \sin\phi>0,\\
f'(L,\phi,\psi)&=&f'(L,\rr\phi,\psi).
\end{array}
\right.
\end{eqnarray}
Assume $\abs{f'_L}\leq C$ and
\begin{eqnarray}
\tnnm{f'-f'_L}\leq C.
\end{eqnarray}
Then we can repeat the proof procedure and obtain
\begin{eqnarray}
\tnnm{f'-f'_L}\leq C+C\e^{\frac{1}{2}}\tnnm{f'-f'_L}.
\end{eqnarray}
Note that in this proof, $O(1)$ term $C$ purely comes from the boundary data and source term. Since all data are zero in $f'$ equation, we have
\begin{eqnarray}
\tnnm{f'-f'_L}\leq C\e^{\frac{1}{2}}\tnnm{f'-f'_L},
\end{eqnarray}
which implies $f'=f'_L$ is a constant. Then based on zero boundary data, we must have $f'=0$.

\end{proof}

\subsubsection{$\bar S\neq0$ Case}

Consider the $\e$-Milne problem for $f(\eta,\phi)$ in $(\eta,\phi,\psi)\in[0,L]\times[-\pi/2,\pi/2)\times[-\pi,\pi)$ with a general source term
\begin{eqnarray}\label{LT Milne problem.}
\left\{
\begin{array}{rcl}\displaystyle
\sin\phi\frac{\p f}{\p\eta}+F(\eta)\cos\phi\frac{\p
f}{\p\phi}+f-\bar f&=&S(\eta,\phi,\psi),\\
f(0,\phi,\psi)&=&h(\phi,\psi)\ \ \text{for}\ \ \sin\phi>0,\\
f(L,\phi,\psi)&=&f(L,\rr\phi,\psi),
\end{array}
\right.
\end{eqnarray}
where $F=\tf+G$.
\begin{lemma}\label{Milne finite LT.}
Assume $S=0$ satisfying (\ref{Milne bounded}) and (\ref{Milne decay}) and $0<n<\dfrac{2}{5}$. There exists a solution $f(\eta,\phi,\psi)$ of the problem
(\ref{LT Milne problem}), satisfying for some constant $\abs{f_L}\leq C$,
\begin{eqnarray}\label{LT estimate.}
\tnnm{f-f_L}&\leq&C.
\end{eqnarray}
The solution is unique among functions such that (\ref{LT estimate.}) holds
\end{lemma}
\begin{proof}
We can utilize superposition property for this linear problem, i.e.
write $S=\bar S+(S-\bar S)=S_Q+S_R$. Then we solve the problem by
the following steps. \\
\ \\
Step 1: Construction of auxiliary function $f^1$.\\
We first solve $f^1$ as the solution to
\begin{eqnarray}
\left\{
\begin{array}{rcl}\displaystyle
\sin\phi\frac{\p f^1}{\p\eta}+F(\eta)\cos\phi\frac{\p
f^1}{\p\phi}+f^1-\bar f^1&=&S_R(\eta,\phi,\psi),\\
f^1(0,\phi,\psi)&=&h(\phi,\psi)\ \ \text{for}\ \ \sin\phi>0,\\
f^1(L,\phi,\psi)&=&f^1(L,\rr\phi,\psi).
\end{array}
\right.
\end{eqnarray}
Since $\bar S_R=0$, by Lemma \ref{Milne finite LT}, we know there
exists a unique solution $f^1$
satisfying the $L^2$ estimate.\\
\ \\
Step 2: Construction of auxiliary function $f^2$.\\
We seek a function $f^{2}$ satisfying
\begin{eqnarray}\label{mt 07}
-\frac{1}{4\pi}\int_{-\pi}^{\pi}\int_{-\pi/2}^{\pi/2}\bigg(\sin\phi\frac{\p
f^{2}}{\p\eta}+F(\eta)\cos\phi\frac{\p
f^{2}}{\p\phi}\bigg)\cos\phi\ud{\phi}\ud{\psi}+S_Q=0.
\end{eqnarray}
The following analysis shows this type of function can always be
found. An integration by parts transforms the equation (\ref{mt 07}) into
\begin{eqnarray}\label{mt 08}
-\int_{-\pi}^{\pi}\int_{-\pi/2}^{\pi/2}\frac{\p
f^{2}}{\p\eta}\sin\phi\cos\phi\ud{\phi}\ud{\psi}-\int_{-\pi}^{\pi}\int_{-\pi/2}^{\pi/2}F(\eta)f^{2}\sin\phi
\cos\phi\ud{\phi}\ud{\psi}+4\pi S_Q=0.
\end{eqnarray}
Setting
\begin{eqnarray}
f^{2}(\phi,\eta)=a(\eta)\sin\phi.
\end{eqnarray}
and plugging this ansatz into (\ref{mt 08}), we have
\begin{eqnarray}
-\frac{\ud{a}}{\ud{\eta}}\int_{-\pi}^{\pi}\int_{-\pi/2}^{\pi/2}\sin^2\phi \cos\phi\ud{\phi}\ud{\psi}-a(\eta)\int_{-\pi}^{\pi}\int_{-\pi/2}^{\pi/2}F(\eta)\sin^2\phi\cos\phi\ud{\phi}\ud{\psi}+4\pi
S_Q=0.
\end{eqnarray}
Hence, we have
\begin{eqnarray}
-\frac{\ud{a}}{\ud{\eta}}-\bar F(\eta)a(\eta)+2S_Q=0,
\end{eqnarray}
where
\begin{eqnarray}
\bar F(\eta)=\int_{-\pi}^{\pi}\int_{-\pi/2}^{\pi/2}F(\eta)\sin^2\phi\cos\phi\ud{\phi}\ud{\psi}\sim\bigg(\dfrac{\e}{R_1-\e\eta}+\dfrac{\e}{R_2-\e\eta}\bigg)
\end{eqnarray}
This is a first order linear ordinary differential equation, which
possesses infinite solutions. We can directly solve it to obtain
\begin{eqnarray}
a(\eta)=\ue^{-\int_0^{\eta}\bar F(y)\ud{y}}\bigg(a(0)+\int_0^{\eta}\ue^{\int_0^y\bar F(z)\ud{z}}2S_Q(y)\ud{y}\bigg).
\end{eqnarray}
We may take
\begin{eqnarray}
a(0)=-\int_0^{L}\ue^{\int_0^y\bar F(z)\ud{z}}2S_Q(y)\ud{y}.
\end{eqnarray}
Based on the exponential decay of $S_Q$, we can directly verify
$a(\eta)$ decays exponentially to zero as $\eta\rt L$ and $f^2$
satisfies the $L^2$ estimate.\\
\ \\
Step 3: Construction of auxiliary function $f^3$.\\
Based on above construction, we can directly verify
\begin{eqnarray}\label{mt 09}
\int_{-\pi}^{\pi}\int_{-\pi/2}^{\pi/2}\bigg(-\sin\phi\frac{\p
f^{2}}{\p\eta}-F(\eta)\cos\phi\frac{\p f^{2}}{\p\phi}-f^{2}+\bar
f^{2}+S_Q\bigg)\cos\phi\ud{\phi}\ud{\psi}=0.
\end{eqnarray}
Then we can solve $f^3$ as the solution to
\begin{eqnarray}
\left\{
\begin{array}{rcl}\displaystyle
\sin\phi\frac{\p f^3}{\p\eta}+F(\eta)\cos\phi\frac{\p
f^3}{\p\phi}+f^3-\bar f^3&=&-\sin\phi\dfrac{\p
f^{2}}{\p\eta}-F(\eta)\cos\phi\dfrac{\p
f^{2}}{\p\phi}-f^{2}+\bar f^{2}+S_Q,\\
f^3(0,\phi,\psi)&=&-a(0)\sin\phi\ \ \text{for}\ \ \sin\phi>0,\\
f^3(L,\phi,\psi)&=&f^3(L,\rr\phi,\psi).
\end{array}
\right.
\end{eqnarray}
By (\ref{mt 09}), we can apply Lemma \ref{Milne finite LT}
to obtain a unique solution $f^3$
satisfying the $L^2$ estimate.\\
\ \\
Step 4: Construction of auxiliary function $f^4$.\\
We now define $f^4=f^2+f^3$ and an explicit verification shows
\begin{eqnarray}
\left\{
\begin{array}{rcl}\displaystyle
\sin\phi\frac{\p f^4}{\p\eta}+F(\eta)\cos\phi\frac{\p
f^4}{\p\phi}+f^4-\bar f^4&=&S_Q(\eta,\phi,\psi),\\
f^4(0,\phi,\psi)&=&0\ \ \text{for}\ \ \sin\phi>0,\\
f^4(L,\phi,\psi)&=&f^4(L,\rr\phi,\psi),
\end{array}
\right.
\end{eqnarray}
and $f^4$
satisfies the $L^2$ estimate.\\
\ \\
In summary, we deduce that $f^1+f^4$ is the solution of (\ref{LT Milne
problem.}) and satisfies the $L^2$ estimate.
\end{proof}
Combining all above, we have the following theorem.
\begin{theorem}\label{Milne LT Theorem}
For the $\e$-Milne problem (\ref{Milne problem}), there exists a unique
solution $f(\eta,\phi,\psi)$ satisfying the estimates
\begin{eqnarray}
\tnnm{f-f_L}\leq C
\end{eqnarray}
for some constant $f_L$ satisfying
\begin{eqnarray}
\abs{f_L}\leq C.
\end{eqnarray}
\end{theorem}

\subsection{$L^{\infty}$ Estimates}

This section is similar to Section 3 of \cite{AA007} with obvious modifications, so we omit the proof here and only present the main results.

\begin{theorem}\label{Milne LT-LI Theorem}
The solution $f(\eta,\phi,\psi)$ to the Milne problem (\ref{Milne
problem}) satisfies
\begin{eqnarray}
\lnnm{f-f_L}\leq C\bigg(1+\tnnm{f-f_L}\bigg).
\end{eqnarray}
\end{theorem}

\begin{theorem}\label{Milne LI Theorem}
There exists a unique solution $f(\eta,\phi,\psi)$ to the $\e$-Milne problem
(\ref{Milne problem}) satisfying
\begin{eqnarray}
\lnnm{f-f_L}\leq C.
\end{eqnarray}
\end{theorem}

\subsection{Exponential Decay}

In this section, we prove the spatial decay of the solution to the
Milne problem.
\begin{theorem}\label{Milne Decay Theorem}
Assume (\ref{Milne bounded}) and (\ref{Milne decay}) hold and $0<n<\dfrac{2}{5}$. For $K_0>0$ sufficiently small, the solution $f(\eta,\phi,\psi)$ to the
$\e$-Milne problem (\ref{Milne problem}) satisfies
\begin{eqnarray}
\lnnm{\ue^{K_0\eta}(f-f_L)}\leq C.
\end{eqnarray}
\end{theorem}
\begin{proof}
Let $\v=f-f_L$. Then $\v$ satisfies
\begin{eqnarray}\label{decay equation}
\left\{
\begin{array}{rcl}\displaystyle
\sin\phi\frac{\p \v}{\p\eta}+F(\eta,\psi)\cos\phi\frac{\p
\v}{\p\phi}+\v-\bar\v&=&S,\\
\v(0,\phi,\psi)&=&p(\phi,\psi)=h(\phi,\psi)-f_L\ \ \text{for}\ \ \sin\phi>0,\\
\v(L,\phi,\psi)&=&\v(L,\rr\phi,\psi).
\end{array}
\right.
\end{eqnarray}
We divide the analysis into several steps:\\
\ \\
Step 1: $L^2$ Estimates.\\
Assume $\bar S=0$. We continue using the notation $F=\tf+G$ and the decomposition $\v=r_{\v}+q_{\v}$. Now we naturally have $(q_{\v})_{L}=0$. The quasi-orthogonal property
reveals
\begin{eqnarray}
\br{\v,\v\sin\phi}_{\phi}(\eta)&=&\br{r_{\v},r_{\v}\sin\phi}_{\phi}(\eta)+2\br{r_{\v},q_{\v}\sin\phi}_{\phi}(\eta)+\br{q_{\v},q_{\v}\sin\phi}_{\phi}(\eta)\\
&=&\br{r_{\v},r_{\v}\sin\phi}_{\phi}(\eta)-4q_{\v}(\eta)\int_{\eta}^L\ue^{2\tv(\eta)-\tv(y)}G(y)\br{\sin\phi\cos^2\psi,r_{\v}}(y)\ud{y}.\no
\end{eqnarray}
Multiplying $\ue^{2K_0\eta}\v\cos\phi$ on both sides of equation (\ref{decay equation}) and integrating over $(\phi,\psi)\in[-\pi/2,\pi/2)\times[-\pi,\pi)$, we obtain
\begin{eqnarray}\label{mt 91}
&&\half\frac{\ud{}}{\ud{\eta}}\bigg(\ue^{2K_0\eta}\br{\v,\v\sin\phi}_{\phi}(\eta)\bigg)
+F(\eta)\bigg(\ue^{2K_0\eta}\br{\v,\v\sin\phi}_{\phi}(\eta)\bigg)\\
&=&\ue^{2K_0\eta}K_0\br{\v,\v\sin\phi}_{\phi}(\eta)-\br{r_{\v},
r_{\v}}_{\phi}(\eta)\no\\
&&-G(\eta)\bigg(\ue^{2K_0\eta}\br{\v\cos^2\psi,\v\sin\phi}_{\phi}(\eta)\bigg)+\ue^{2K_0\eta}\br{S,r_{\v}}_{\phi}(\eta)\nonumber\\
&=&\ue^{2K_0\eta}\bigg(K_0\br{r_{\v},r_{\v}\sin\phi}_{\phi}(\eta)-\br{r_{\v},
r_{\v}}_{\phi}(\eta)\bigg)\no\\
&&+4\ue^{2K_0\eta}K_0q_{\v}(\eta)\int_{\eta}^L\ue^{2\tv(\eta)-\tv(y)}G(y)\br{\sin\phi\cos^2\psi,r_{\v}}(y)\ud{y}\no\\
&&-G(\eta)\bigg(\ue^{2K_0\eta}\br{\v\cos^2\psi,\v\sin\phi}_{\phi}(\eta)\bigg)+\ue^{2K_0\eta}\br{S,r_{\v}}_{\phi}(\eta)\nonumber.
\end{eqnarray}
For $K_0<\min\{1/2,K\}$, we have
\begin{eqnarray}\label{mt 92}
\frac{3}{2}\tnm{r_{\v}(\eta)}^2\geq-K_0\br{r_{\v},r_{\v}\sin\phi}_{\phi}(\eta)+\br{r_{\v},r_{\v}}_{\phi}(\eta)\geq
\half\tnm{r_{\v}(\eta)}^2.
\end{eqnarray}
Similar to the proof of Lemma \ref{Milne finite LT}, formula as (\ref{mt 91}) and
(\ref{mt 92}) imply
\begin{eqnarray}\label{mt 93}
\tnnm{\ue^{K_0\eta}r_{\v}}^2&\leq&\abs{\int_0^L4\ue^{2K_0\eta}K_0q_{\v}(\eta)\int_{\eta}^L\ue^{2\tv(\eta)-\tv(y)}G(y)\br{\sin\phi\cos^2\psi,r_{\v}}(y)\ud{y}\ud{\eta}}\\
&&+\abs{\int_0^LG(\eta)\bigg(\ue^{2K_0\eta}\br{\v\cos^2\psi,\v\sin\phi}_{\phi}(\eta)\bigg)\ud{\eta}}+\abs{\int_0^L\ue^{2K_0\eta}\br{S,r_{\v}}_{\phi}(\eta)\ud{\eta}}\no\\
&\leq&CL\tnnm{G}\tnnm{\ue^{K_0\eta}r_{\v}}\tnnm{\ue^{K_0\eta}q_{\v}}+C\e\tnnm{\ue^{K_0\eta}\v}^2\no\\
&&+C\tnnm{\ue^{K_0\eta}S}\tnnm{\ue^{K_0\eta}r_{\v}}\no\\
&\leq&C\e\tnnm{\ue^{K_0\eta}q_{\v}}^2+C\e^{1-\frac{3}{2}n}\tnnm{\ue^{K_0\eta}r_{\v}}^2+\e\tnnm{\ue^{K_0\eta}\v}^2+C\tnnm{S}^2\no\\
&\leq&C+C\tnnm{\ue^{K_0\eta}q_{\v}}^2+C\e^{1-\frac{3}{2}n}\tnnm{\ue^{K_0\eta}r_{\v}}^2+\e\tnnm{\ue^{K_0\eta}\v}^2.\no
\end{eqnarray}
Hence, for $\e$ sufficiently small, we know
\begin{eqnarray}
\tnnm{\ue^{K_0\eta}r_{\v}}^2&\leq&C+C\e\tnnm{\ue^{K_0\eta}q_{\v}}^2+C\e\tnnm{\v}^2.
\end{eqnarray}
Then similar to the proof of Lemma \ref{Milne finite LT}, we deduce
\begin{eqnarray}
&&\tnnm{\ue^{K_0\eta}q_{\v}}^2\\
&\leq&\tnnm{\ue^{K_0\eta}r_{\v}}^2+\int_0^L\ue^{2K_0\eta}\abs{\int_{\eta}^L\tf(y)\br{1-3\sin^2\phi,r_{\v}}(y)\ud{y}}^2\ud{\eta}\no\\
&&+\int_0^L\ue^{2K_0\eta}\abs{\int_{\eta}^{L}\int_{z}^L\ue^{2\tv(z)-2\tv{y}}G(y)\br{\sin\phi\cos^2\psi,r_{\v}}\ud{y}\ud{z}}^2\ud{\eta}\no\\
&&+\int_0^L\ue^{2K_0\eta}\abs{\int_{\eta}^{L}G(y)\br{1-3\sin^2\phi,r_{\v}\cos^2\psi}(y)\ud{y}}^2\ud{\eta}
+\int_0^L\ue^{2K_0\eta}\abs{\int_{\eta}^{L}\br{\sin\phi,S}(y)\ud{y}}^2\ud{\eta}\no\\
&\leq&C+C\tnnm{\ue^{K_0\eta}r_{\v}}^2+C\tnnm{\ue^{K_0\eta}r_{\v}}
\bigg(\int_0^{L}\int_{\eta}^{L}\ue^{2K_0(\eta-y)}F^2(y)\ud{y}\ud{\eta}\bigg)\no\\
&&+L^3\tnnm{G}^2\tnnm{\ue^{K_0\eta}r_{\v}}^2+L\tnnm{G}^2\tnnm{\ue^{K_0\eta}r_{\v}}^2
+\int_0^{L}\ue^{2K_0\eta}\bigg(\int_{\eta}^{L}\lnm{S(y)}\ud{y}\bigg)^2\ud{\eta}
\nonumber\\
&\leq&C+C(1+\e^{2-5n})\tnnm{\ue^{K_0\eta}r_{\v}}^2\no\\
&\leq&C+C\tnnm{\ue^{K_0\eta}r_{\v}}^2\no\\
&\leq&C+C\e\tnnm{\ue^{K_0\eta}q_{\v}}^2+C\e\tnnm{\ue^{K_0\eta}\v}^2,
\end{eqnarray}
which implies
\begin{eqnarray}
\tnnm{\ue^{K_0\eta}q_{\v}}^2\leq C+C\e\tnnm{\ue^{K_0\eta}\v}^2.
\end{eqnarray}
In summary, we have
\begin{eqnarray}
\tnnm{\ue^{K_0\eta}\v}^2&\leq&\tnnm{\ue^{K_0\eta}q_{\v}}^2+\tnnm{\ue^{K_0\eta}r_{\v}}^2\leq C+C\e\tnnm{\ue^{K_0\eta}\v}^2,
\end{eqnarray}
which yields
\begin{eqnarray}\label{mt 94}
\tnnm{\ue^{K_0\eta}\v}&\leq&C.
\end{eqnarray}
This is the desired result when $\bar S=0$.
By the method introduced in Lemma \ref{Milne finite LT.},
we can extend above $L^2$ estimates to the general $S$ case. Note
all the auxiliary functions
constructed in Lemma \ref{Milne finite LT.} satisfy the desired estimates. \\
\ \\
Step 2: $L^{\infty}$ Estimates.\\
This is similar to the proof of exponential decay in \cite{AA007}, so we omit the details here. We have
\begin{eqnarray}\label{mt 95}
\lnnm{\ue^{K_0\eta}\v}\leq C+C\tnnm{\ue^{K_0\eta}\v}.
\end{eqnarray}
Combining (\ref{mt 94}) and (\ref{mt 95}), we
deduce the desired result
\begin{eqnarray}
\lnnm{\ue^{K_0\eta}(f-f_L)}\leq C.
\end{eqnarray}
\end{proof}

\subsection{Diffusive Boundary}

In this subsection, we consider the $\e$-Milne problem with diffusive boundary as
\begin{eqnarray}\label{Milne problem.}
\left\{ \begin{array}{rcl}\displaystyle \sin\phi\frac{\p
f}{\p\eta}+F(\eta,\psi)\cos\phi\frac{\p
f}{\p\phi}+f-\bar f&=&S(\eta,\phi,\psi),\\
f(0,\phi,\psi)&=& h(\phi,\psi)+\pp[f](0)\ \ \text{for}\
\ \sin\phi>0,\\
f(L,\phi,\psi)&=&f(L,\rr\phi,\psi),
\end{array}
\right.
\end{eqnarray}
where
\begin{eqnarray}
\pp
[f](0)=-\frac{1}{4\pi}\iint_{\sin\phi<0}f(0,\phi,\psi)\sin\phi\cos\phi\ud{\phi}\ud{\psi},
\end{eqnarray}
Similar to \cite[Section 6]{AA003}, we can easily prove that
\begin{lemma}\label{Milne lemma 1.}
In order for the equation (\ref{Milne problem.}) to have a solution
$f(\eta,\phi,\psi)\in L^{\infty}([0,L]\times[-\pi,\pi)\times[-\pi,\pi)\times[-\pi/2,\pi/2))$, the boundary data $h$
and the source term $S$ must satisfy the compatibility condition
\begin{eqnarray}\label{Milne compatibility condition}
\iint_{\sin\phi>0}h(\phi,\psi)\sin\phi\cos\phi\ud{\phi}\ud{\psi}
+\int_0^{L}\int_{-\pi}^{\pi}\int_{-\pi/2}^{\pi/2}\ue^{-V(s)}S(s,\phi,\psi)\cos\phi\ud{\phi}\ud{\psi}\ud{s}=0.
\end{eqnarray}
In particular, if $S=0$, then the compatibility condition reduces to
\begin{eqnarray}\label{Milne reduced compatibility condition}
\iint_{\sin\phi>0}h(\phi,\psi)\sin\phi\cos\phi\ud{\phi}\ud{\psi}=0.
\end{eqnarray}
\end{lemma}
It is easy to see if $f$ is a solution to (\ref{Milne problem.}),
then $f+C$ is also a solution for any constant $C$. Hence, in order
to obtain a unique solution, we need a normalization condition
\begin{eqnarray}\label{Milne normalization}
\pp[f](0)=0.
\end{eqnarray}
The following lemma in \cite[Section 6]{AA003} tells us the problem (\ref{Milne problem.}) can
be reduced to the $\e$-Milne problem with in-flow boundary (\ref{Milne problem}).
\begin{lemma}\label{Milne lemma 2.}
If the boundary data $h$ and $S$ satisfy the compatibility condition
(\ref{Milne compatibility condition}), then the solution $f$ to the
$\e$-Milne problem (\ref{Milne problem}) with in-flow boundary as
$f=h$ on $\sin\phi>0$ is also a solution to the $\e$-Milne problem
(\ref{Milne problem.}) with diffusive boundary, which satisfies the
normalization condition (\ref{Milne normalization}). Furthermore,
this is the unique solution to (\ref{Milne problem.}) among the
functions satisfying (\ref{Milne normalization}) and
$\tnnm{f(\eta,\phi,\psi)-f_L}\leq C$.
\end{lemma}
In summary, based on above analysis, we can utilize the known result
for $\e$-Milne problem (\ref{Milne problem}) to obtain the
desired results of the solution to the
$\e$-Milne problem (\ref{Milne problem.}).
\begin{theorem}\label{Milne theorem 1.}
There exists a unique solution $f(\eta,\phi,\psi)$ to the $\e$-Milne problem
(\ref{Milne problem.}) with the normalization condition (\ref{Milne
normalization}) satisfying for some constant $\abs{f_L}<C$,
\begin{eqnarray}
\tnnm{f(\eta,\phi,\psi)-f_L}\leq C.
\end{eqnarray}
\end{theorem}
\begin{theorem}\label{Milne theorem 2.}
The unique solution $f(\eta,\phi,\psi)$ to the $\e$-Milne problem
(\ref{Milne problem.}) with the normalization condition (\ref{Milne
normalization}) satisfying for some constant $\abs{f_L}<C$,
\begin{eqnarray}
\lnnm{f(\eta,\phi,\psi)-f_L}\leq C.
\end{eqnarray}
\end{theorem}
\begin{theorem}\label{Milne theorem 3.}
There exists $K_0>0$ such that the solution $f(\eta,\phi,\psi)$ to the
$\e$-Milne problem (\ref{Milne problem.}) with the normalization
condition (\ref{Milne normalization}) satisfies
\begin{eqnarray}
\lnnm{\ue^{K_0\eta}\bigg(f(\eta,\phi,\psi)-f_L\bigg)}\leq C.
\end{eqnarray}
\end{theorem}

\newpage

\section{Regularity of $\e$-Milne Problem}

We continue studying the $\e$-Milne problem with in-flow boundary as
\begin{eqnarray}
\left\{ \begin{array}{rcl}\displaystyle \sin\phi\frac{\p
f}{\p\eta}+F(\eta,\psi)\cos\phi\frac{\p
f}{\p\phi}+f-\bar f&=&S(\eta,\phi,\psi),\\
f(0,\phi,\psi)&=& h(\phi,\psi)\ \ \text{for}\
\ \sin\phi>0,\\
f(L,\phi,\psi)&=&f(L,\rr\phi,\psi).
\end{array}
\right.
\end{eqnarray}
Here we already omit the superscript $\e$ and dependence on $(\tau_1,\tau_2)$.
Besides (\ref{Milne bounded}) and (\ref{Milne decay}), we further assume
\begin{eqnarray}\label{Milne bounded.}
\lnm{\frac{\p h}{\p\phi}}+\lnm{\frac{\p h}{\p\psi}}+\lnm{\frac{\p h}{\p\tau_1}}+\lnm{\frac{\p h}{\p\tau_2}}\leq M,
\end{eqnarray}
and
\begin{eqnarray}\label{Milne decay.}
\lnnm{\ue^{K\eta}\frac{\p S}{\p\eta}}
+\lnnm{\ue^{K\eta}\frac{\p S}{\p\phi}}+\lnnm{\ue^{K\eta}\frac{\p S}{\p\psi}}
+\lnnm{\ue^{K\eta}\frac{\p S}{\p\tau_1}}+\lnnm{\ue^{K\eta}\frac{\p S}{\p\tau_2}}&\leq& M,
\end{eqnarray}
for some $M,K>0$. Define a potential function $V(\eta,\psi)$ satisfying $V(0,\psi)=0$ and $\dfrac{\p V}{\p\eta}=-F(\eta,\psi)$. Also, we know $L=\e^{-n}$ for $0<n<\dfrac{2}{5}$.
\begin{lemma}\label{rt lemma 1}
We have
$\ue^{-V(0,\psi)}=1$ and
\begin{eqnarray}
\ue^{-V(L,\psi)}=\bigg(1-\frac{\e^{1-n}}{\rk_1}\bigg)^{\sin^2\psi}
\bigg(1-\frac{\e^{1-n}}{\rk_2}\bigg)^{\cos^2\psi}.
\end{eqnarray}
Also, for $\rk=\max\{R_1, R_2\}$ and $\rk'=\min\{R_1, R_2\}$ which are the maximum and minimum of $\rk_1$ and $\rk_2$, we have
\begin{eqnarray}
\frac{\rk'-\e\eta}{\rk'}\leq\ue^{-V(\eta,\psi)}\leq\frac{\rk-\e\eta}{\rk}.
\end{eqnarray}
\end{lemma}
\begin{proof}
We directly compute
\begin{eqnarray}
V(\eta,\psi)=\sin^2\psi\ln\bigg(\frac{\rk_1}{\rk_1-\e\eta}\bigg)
+\cos^2\psi\ln\bigg(\frac{\rk_2}{\rk_2-\e\eta}\bigg),
\end{eqnarray}
and
\begin{eqnarray}
\ue^{-V(\eta,\psi)}=\bigg(\frac{\rk_1-\e\eta}{\rk_1}\bigg)^{\sin^2\psi}
\bigg(\frac{\rk_2-\e\eta}{\rk_2}\bigg)^{\cos^2\psi}.
\end{eqnarray}
Hence, our result naturally follows.
\end{proof}

\subsection{Preliminaries}

It is easy to see $\v(\eta,\phi,\psi)=f(\eta,\phi,\psi)-f_L$ satisfies the equation
\begin{eqnarray}\label{Milne difference problem}
\left\{
\begin{array}{rcl}\displaystyle
\sin\phi\frac{\p \v}{\p\eta}+F(\eta,\psi)\cos\phi\frac{\p
\v}{\p\phi}+\v-\bar\v&=&S(\eta,\phi,\psi),\\
\v(0,\phi,\psi)&=&p(\phi,\psi)\ \ \text{for}\ \ \sin\phi>0,\\\rule{0ex}{1.0em}
\v(L,\phi,\psi)&=&\v(L,\rr\phi,\psi).
\end{array}
\right.
\end{eqnarray}
where
\begin{eqnarray}
p(\phi,\psi)=h(\phi,\psi)-f_L.
\end{eqnarray}
We intend to estimate the normal, tangential and velocity derivative. This idea is motivated by \cite{Guo.Kim.Tonon.Trescases2013} and \cite{AA007}.
Define a distance function $\zeta(\eta,\phi,\psi)$ as
\begin{eqnarray}\label{weight function}
\zeta(\eta,\phi,\psi)=\Bigg(1-\bigg(\ue^{-V(\eta,\psi)}\cos\phi\bigg)^2\Bigg)^{1/2}.
\end{eqnarray}
Note that the closer $(\eta,\phi,\psi)$ is to the grazing set which satisfies $\eta=0$ and $\sin\phi=0$, the smaller $\zeta$ is. In particular, at grazing set, $\zeta=0$. Also, we have $0\leq\zeta\leq 1$.
\begin{lemma}\label{rt lemma 2}
We have
\begin{eqnarray}
\sin\phi\frac{\p \zeta}{\p\eta}+F(\eta,\psi)\cos\phi\frac{\p
\zeta}{\p\phi}=0.
\end{eqnarray}
\end{lemma}
\begin{proof}
We may directly compute
\begin{eqnarray}
\frac{\p \zeta}{\p\eta}&=&\frac{1}{2}\Bigg(1-\bigg(\ue^{-V(\eta,\psi)}\cos\phi\bigg)^2\Bigg)^{-1/2}\bigg(-2\ue^{-2V(\eta,\psi)}\cos^2\phi\bigg)F(\eta,\psi)
=-\frac{\ue^{-2V(\eta,\psi)}F(\eta,\psi)\cos^2\phi}{\zeta},\\
\frac{\p\zeta}{\p\phi}&=&\frac{1}{2}\Bigg(1-\bigg(\ue^{-V(\eta,\psi)}\cos\phi\bigg)^2\Bigg)^{-1/2}\bigg(-2\ue^{-2V(\eta,\psi)}\cos\phi\bigg)(-\sin\phi)
=\frac{\ue^{-2V(\eta,\psi)}\cos\phi\sin\phi}{\zeta}.
\end{eqnarray}
Hence, we know
\begin{eqnarray}
\\
\sin\phi\frac{\p \zeta}{\p\eta}+F(\eta,\psi)\cos\phi\frac{\p\zeta}{\p\phi}&=&
\frac{-\sin\phi\bigg(\ue^{-2V(\eta,\psi)}F(\eta,\psi)\cos^2\phi\bigg)+F(\eta,\psi)\cos\phi\bigg(\ue^{-2V(\eta,\psi)}\cos\phi\sin\phi\bigg)}{\zeta}=0.\no
\end{eqnarray}
\end{proof}

As a matter of fact, we are able to prove some preliminary estimates that are based on the characteristics of $\v$ itself instead of the derivative. In the following, let $0<\d_0<<1$ be a small quantity.

\begin{lemma}\label{pt lemma 1}
Assume (\ref{Milne bounded}), (\ref{Milne decay}), (\ref{Milne bounded.}) and (\ref{Milne decay.}). For $\sin\phi>\d_0$, we have
\begin{eqnarray}
\abs{\sin\phi\frac{\p\v}{\p\eta}(\eta,\phi)}\leq C\bigg(1+\frac{1}{\d_0^3}\bigg).
\end{eqnarray}
\end{lemma}

\begin{lemma}\label{pt lemma 2}
Assume (\ref{Milne bounded}), (\ref{Milne decay}), (\ref{Milne bounded.}) and (\ref{Milne decay.}). For $\sin\phi<0$ with $\abs{E(\eta,\phi)}\leq \ue^{-V(L)}$, if it satisfies $\min_{\phi'}\sin\phi'\geq\d_0$ where $(\eta',\phi')$ are on the same characteristics as $(\eta,\phi)$ with $\sin\phi'\geq0$, then we have
\begin{eqnarray}
\abs{\sin\phi\frac{\p\v}{\p\eta}(\eta,\phi)}\leq C\bigg(1+\frac{1}{\d_0^3}\bigg).
\end{eqnarray}
\end{lemma}

\begin{lemma}\label{pt lemma 6}
Assume (\ref{Milne bounded}), (\ref{Milne decay}), (\ref{Milne bounded.}), (\ref{Milne decay.}) and $\abs{\dfrac{\p\bar\v}{\p\eta}}\leq C(1+\abs{\ln(\e)}+\abs{\ln(\eta)})$. For $\sin\phi\leq0$ and $\abs{E(\eta,\phi)}\geq \ue^{-V(L)}$, we have
\begin{eqnarray}
\abs{\sin\phi\frac{\p\v}{\p\eta}(\eta,\phi)}\leq C(1+\abs{\ln(\e)}).
\end{eqnarray}
\end{lemma}

The proofs of Lemma \ref{pt lemma 1}, Lemma \ref{pt lemma 2} and Lemma \ref{pt lemma 6} are similar to those in \cite{AA007} with obvious modifications, so we omit the details here.

\begin{remark}
Estimates in Lemma \ref{pt lemma 1}, Lemma \ref{pt lemma 2} and Lemma \ref{pt lemma 6} can provide pointwise bounds of derivatives. However, they are not uniform estimates due to presence of $\d_0$ and $\ln(\e)$. We need weighted $L^{\infty}$ estimates of derivatives to close the proof. Also, the estimate $\abs{\dfrac{\p\bar\v}{\p\eta}}\leq C(1+\abs{\ln(\e)}+\abs{\ln(\eta)})$ are not known a priori, so we need an iteration argument.
\end{remark}

\subsection{Mild Formulation of Normal Derivative}

In this and next subsection, we will prove stronger a priori estimates of derivatives. Consider the $\e$-transport problem for $\a=\zeta\dfrac{\p\v}{\p\eta}$ as
\begin{eqnarray}\label{Milne infinite problem LI}
\left\{
\begin{array}{rcl}\displaystyle
\sin\phi\frac{\p\a}{\p\eta}+F(\eta,\psi)\cos\phi\frac{\p
\a}{\p\phi}+\a&=&\tilde\a+S_{\a},\\
\a(0,\phi,\psi)&=&p_{\a}(\phi,\psi)\ \ \text{for}\ \ \sin\phi>0,\\
\a(L,\phi,\psi)&=&\a(L,\rr\phi,\psi),
\end{array}
\right.
\end{eqnarray}
where $p_{\a}$ and $S_{\a}$ will be specified later with
\begin{eqnarray}
\tilde\a(\eta,\phi,\psi)=\frac{1}{4\pi}\int_{-\pi}^{\pi}\int_{-\pi/2}^{\pi/2}\frac{\zeta(\eta,\phi,\psi)}{\zeta(\eta,\phi_{\ast},\psi)}
\a(\eta,\phi_{\ast},\psi)\cos\phi_{\ast}\ud{\phi_{\ast}}\ud{\psi}.
\end{eqnarray}

\begin{lemma}\label{pt lemma 3}
We have
\begin{eqnarray}
\lnnm{\a}&\leq&C\bigg(\lnm{p_{\a}}+\lnnm{S_{\a}}\bigg)\\
&&+C\abs{\ln(\e)}^8\bigg(\lnnm{\v}+\lnm{\frac{\p p}{\p\phi}}+\lnnm{S}+\lnnm{\frac{\p S}{\p\phi}}+\lnnm{\frac{\p S}{\p\eta}}\bigg).\no
\end{eqnarray}
\end{lemma}
The rest of this subsection will be devoted to the proof of this lemma. We first introduce some notation.
Define the energy as before
\begin{eqnarray}
E(\eta,\phi,\psi)=\ue^{-V(\eta,\psi)}\cos\phi .
\end{eqnarray}
Along the characteristics, where this energy is conserved and $\zeta$ is a constant, the equation can be simplified as follows:
\begin{eqnarray}
\sin\phi\frac{\ud{\a}}{\ud{\eta}}+\a=\tilde\a+S_{\a}.
\end{eqnarray}
An implicit function
$\eta^+(\eta,\phi,\psi)$ can be determined through
\begin{eqnarray}
\abs{E(\eta,\phi,\psi)}=\ue^{-V(\eta^+,\psi)}.
\end{eqnarray}
which means $(\eta^+,\phi_0,\psi)$ with $\sin\phi_0=0$ is on the same characteristics as $(\eta,\phi,\psi)$.
Define the quantities for $0\leq\eta'\leq\eta^+$ as follows:
\begin{eqnarray}
\phi'(\phi,\eta,\eta',\psi)&=&\cos^{-1}(\ue^{V(\eta',\psi)-V(\eta,\psi)}\cos\phi ),\\
\rr\phi'(\phi,\eta,\eta',\psi)&=&-\cos^{-1}(\ue^{V(\eta',\psi)-V(\eta,\psi)}\cos\phi )=-\phi'(\phi,\eta,\eta',\psi),
\end{eqnarray}
where the inverse trigonometric function can be defined
single-valued in the domain $[0,\pi/2]$ and the quantities are always well-defined due to the monotonicity of $V$. Note that $\sin\phi'\geq0$, even if $\sin\phi<0$.
Finally we put
\begin{eqnarray}
G_{\eta,\eta',\psi}(\phi)&=&\int_{\eta'}^{\eta}\frac{1}{\sin(\phi'(\phi,\eta,\xi,\psi))}\ud{\xi}.
\end{eqnarray}
Similar to $\e$-Milne problem, we can define the solution along the characteristics as follows:
\begin{eqnarray}
\a(\eta,\phi,\psi)=\k[p_{\a}]+\t[\tilde\a+S_{\a}],
\end{eqnarray}
where\\
\ \\
Region I:\\
For $\sin\phi>0$,
\begin{eqnarray}
\k[p_{\a}]&=&p_{\a}(\phi'(0),\psi)\exp(-G_{\eta,0,\psi})\\
\t[\tilde\a+S_{\a}]&=&\int_0^{\eta}\frac{(\tilde\a+S_{\a})(\eta',\phi'(\eta'),\psi)}{\sin(\phi'(\eta'))}\exp(-G_{\eta,\eta',\psi})\ud{\eta'}.
\end{eqnarray}
\ \\
Region II:\\
For $\sin\phi<0$ and $\abs{E(\eta,\phi,\psi)}\leq \ue^{-V(L,\psi)}$,
\begin{eqnarray}
\k[p_{\a}]&=&p_{\a}(\phi'(0),\psi)\exp(-G_{L,0,\psi}-G_{L,\eta,\psi})\\
\t[\tilde\a+S_{\a}]&=&\int_0^{L}\frac{(\tilde\a+S)(\eta',\phi'(\eta'),\psi)}{\sin(\phi'(\eta'))}
\exp(-G_{L,\eta',\psi}-G_{L,\eta,\psi})\ud{\eta'}\\
&&+\int_{\eta}^{L}\frac{(\tilde\a+S)(\eta',\rr\phi'(\eta'),\psi)}{\sin(\phi'(\eta'))}\exp(-G_{\eta',\eta,\psi})\ud{\eta'}\nonumber.
\end{eqnarray}
\ \\
Region III:\\
For $\sin\phi<0$ and $\abs{E(\eta,\phi,\psi)}\geq \ue^{-V(L,\psi)}$,
\begin{eqnarray}
\k[p_{\a}]&=&p_{\a}(\phi'(\phi,\eta,0),\psi)\exp(-G_{\eta^+,0,\psi}-G_{\eta^+,\eta,\psi})\\
\t[\tilde\a+S_{\a}]&=&\int_0^{\eta^+}\frac{(\tilde\a+S_{\a})(\eta',\phi'(\eta'),\psi)}{\sin(\phi'(\eta'))}
\exp(-G_{\eta^+,\eta',\psi}-G_{\eta^+,\eta,\psi})\ud{\eta'}\\
&&+
\int_{\eta}^{\eta^+}\frac{(\tilde\a+S_{\a})(\eta',\rr\phi'(\eta'),\psi)}{\sin(\phi'(\eta'))}\exp(-G_{\eta',\eta,\psi})\ud{\eta'}\nonumber.
\end{eqnarray}
Then we need to estimate $\k[p_{\a}]$ and $\t[\tilde\a+S_{\a}]$ in each region. We assume $0<\d<<1$ and $0<\d_0<<1$ are small quantities which will be determined later.

\subsubsection{Region I: $\sin\phi>0$}

A direct computation reveals
\begin{eqnarray}
\abs{\k[p_{\a}]}&\leq&\lnm{p_{\a}},\\
\abs{\t[S_{\a}]}&\leq&\lnnm{S_{\a}}.
\end{eqnarray}
Hence, we only need to estimate $I=\t[\tilde\a]$. We divide it into several steps:\\
\ \\
Step 0: Preliminaries.\\
We have
\begin{eqnarray}
E(\eta',\phi')=\bigg(\frac{\rk_1-\e\eta'}{\rk_1}\bigg)^{\sin^2\psi}\bigg(\frac{\rk_2-\e\eta'}{\rk_2}\bigg)^{\cos^2\psi}\cos\phi'.
\end{eqnarray}
We can directly obtain
\begin{eqnarray}\label{pt 01}
\zeta(\eta',\phi',\psi)&\leq&\frac{1}{\rk'}\sqrt{\rk'^2-\bigg((\rk'-\e\eta')\cos\phi'\bigg)^2}
=\frac{1}{\rk'}\sqrt{\rk'^2-(\rk'-\e\eta')^2+(\rk'-\e\eta')^2\sin^2\phi'},\\
&\leq& \frac{\sqrt{\rk'^2-(\rk'-\e\eta')^2}+\sqrt{(\rk'-\e\eta')^2\sin^2\phi'}}{\rk'}\leq C\bigg(\sqrt{\e\eta'}+\sin\phi'\bigg),\no
\end{eqnarray}
and
\begin{eqnarray}\label{pt 02}
\zeta(\eta',\phi',\psi)\geq\frac{1}{\rk}\sqrt{\rk^2-(\rk-\e\eta')^2}\geq C\sqrt{\e\eta'}.
\end{eqnarray}
Also, we know for $0\leq\eta'\leq\eta$,
\begin{eqnarray}
\sin\phi'&=&\sqrt{1-\cos^2\phi'}\leq\sqrt{1-\bigg(\frac{\rk'-\e\eta}{\rk'-\e\eta'}\bigg)^2\cos^2\phi}\\
&=&\frac{\sqrt{(\rk'-\e\eta')^2\sin^2\phi+(2\rk'-\e\eta-\e\eta')(\e\eta-\e\eta')\cos^2\phi}}{\rk'-\e\eta'}.
\end{eqnarray}
Since
\begin{eqnarray}
0\leq(2\rk'-\e\eta-\e\eta')(\e\eta-\e\eta')\cos^2\phi\leq 2\rk'\e(\eta-\eta'),
\end{eqnarray}
we have
\begin{eqnarray}
\sin\phi\leq\sin\phi'
\leq2\sqrt{\sin^2\phi+\e(\eta-\eta')},
\end{eqnarray}
which means
\begin{eqnarray}
\frac{1}{2\sqrt{\sin^2\phi+\e(\eta-\eta')}}\leq\frac{1}{\sin\phi'}
\leq\frac{1}{\sin\phi}.
\end{eqnarray}
Therefore,
\begin{eqnarray}\label{pt 03}
-\int_{\eta'}^{\eta}\frac{1}{\sin\phi'(y)}\ud{y}&\leq& -\int_{\eta'}^{\eta}\frac{1}{2\sqrt{\sin^2\phi+\e(\eta-y)}}\ud{y}\\
&=&\frac{1}{\e}\bigg(\sin\phi-\sqrt{\sin^2\phi+\e(\eta-\eta')}\bigg)\no\\
&=&-\frac{\eta-\eta'}{\sin\phi+\sqrt{\sin^2\phi+\e(\eta-\eta')}}\no\\
&\leq&-\frac{\eta-\eta'}{2\sqrt{\sin^2\phi+\e(\eta-\eta')}}.\no
\end{eqnarray}
Define a cut-off function $\chi\in C^{\infty}[-\pi,\pi]$ satisfying
\begin{eqnarray}
\chi(\phi)=\left\{
\begin{array}{ll}
1&\text{for}\ \ \abs{\sin\phi}\leq\d,\\
0&\text{for}\ \ \abs{\sin\phi}\geq2\d,
\end{array}
\right.
\end{eqnarray}
In the following, we will divide the estimate of $I$ into several cases based on the value of $\sin\phi$, $\sin\phi'$, $\e\eta'$ and $\e(\eta-\eta')$. Let $\id$ denote the indicator function. We write
\begin{eqnarray}
I&=&\int_0^{\eta}\id_{\{\sin\phi\geq\d_0\}}+\int_0^{\eta}\id_{\{0\leq\sin\phi\leq\d_0\}}\id_{\{\chi(\phi_{\ast})<1\}}
+\int_0^{\eta}\id_{\{0\leq\sin\phi\leq\d_0\}}\id_{\{\chi(\phi_{\ast})=1\}}\id_{\{ \sqrt{\e\eta'}\geq\sin\phi'\}}\\
&&+\int_0^{\eta}\id_{\{0\leq\sin\phi\leq\d_0\}}\id_{\{\chi(\phi_{\ast})=1\}}\id_{\{\sqrt{\e\eta'}\leq\sin\phi'\}}\id_{\{\sin^2\phi\leq\e(\eta-\eta')\}}\no\\
&&+\int_0^{\eta}\id_{\{0\leq\sin\phi\leq\d_0\}}\id_{\{\chi(\phi_{\ast})=1\}}\id_{\{\sqrt{\e\eta'}\leq\sin\phi'\}}\id_{\{\sin^2\phi\geq\e(\eta-\eta')\}}\no\\
&=&I_1+I_2+I_3+I_4+I_5.\no
\end{eqnarray}
\ \\
Step 1: Estimate of $I_1$ for $\sin\phi\geq\d_0$.\\
Based on Lemma \ref{pt lemma 1}, we know
\begin{eqnarray}
\abs{\sin\phi\frac{\p\v}{\p\eta}}\leq C\bigg(1+\frac{1}{\d_0^3}\bigg)\bigg(\lnnm{\v}+\lnm{\frac{\p p}{\p\phi}}+\lnnm{\frac{\p S}{\p\phi}}+\lnnm{\frac{\p S}{\p\eta}}\bigg).
\end{eqnarray}
Hence, we have
\begin{eqnarray}
\abs{I_1}\leq C\abs{\frac{\p\v}{\p\eta}}\leq \frac{C}{\d_0^4}\bigg(\lnnm{\v}+\lnm{\frac{\p p}{\p\phi}}+\lnnm{\frac{\p S}{\p\phi}}+\lnnm{\frac{\p S}{\p\eta}}\bigg).
\end{eqnarray}
\ \\
Step 2: Estimate of $I_2$ for $0\leq\sin\phi\leq\d_0$ and $\chi(\phi_{\ast})<1$.\\
We have
\begin{eqnarray}
\\
I_2&=&\frac{1}{4\pi}\int_0^{\eta}\bigg(\int_{-\pi}^{\pi}\int_{-\pi/2}^{\pi/2}\frac{\zeta(\eta',\phi',\psi)}{\zeta(\eta',\phi_{\ast},\psi)}(1-\chi(\phi_{\ast}))
\a(\eta',\phi_{\ast},\psi)\cos\phi_{\ast}\ud{\phi_{\ast}}\ud{\psi}\bigg)
\frac{1}{\sin\phi'}\exp(-G_{\eta,\eta',\psi})\ud{\eta'}\no\\
&=&\frac{1}{4\pi}\int_0^{\eta}\bigg(\int_{-\pi}^{\pi}\int_{-\pi/2}^{\pi/2}\zeta(\eta',\phi',\psi)(1-\chi(\phi_{\ast}))
\frac{\v(\eta',\phi_{\ast},\psi)}{\p\eta'}\cos\phi_{\ast}\ud{\phi_{\ast}}\ud{\psi}\bigg)\frac{1}{\sin\phi'}\exp(-G_{\eta,\eta',\psi})\ud{\eta'}.\no
\end{eqnarray}
Based on the $\e$-Milne problem of $\v$ as
\begin{eqnarray}
\sin\phi_{\ast}\frac{\p\v(\eta',\phi_{\ast},\psi)}{\p\eta'}+F(\eta',\psi)\cos\phi_{\ast}\frac{\p\v(\eta',\phi_{\ast},\psi)}{\p\phi_{\ast}}
+\v(\eta',\phi_{\ast},\psi)-\bar\v(\eta')=S(\eta',\phi_{\ast},\psi),
\end{eqnarray}
we have
\begin{eqnarray}
\frac{\p\v(\eta',\phi_{\ast},\psi)}{\p\eta'}=-\frac{1}{\sin\phi_{\ast}}
\bigg(F(\eta',\psi)\cos\phi_{\ast}\frac{\p\v(\eta',\phi_{\ast},\psi)}{\p\phi_{\ast}}+\v(\eta',\phi_{\ast},\psi)-\bar\v(\eta')-S(\eta',\phi_{\ast},\psi)\bigg).
\end{eqnarray}
Hence, we have
\begin{eqnarray}
\tilde\a&=&\int_{-\pi}^{\pi}\int_{-\pi/2}^{\pi/2}\zeta(\eta',\phi',\psi)(1-\chi(\phi_{\ast}))
\frac{\p\v(\eta',\phi_{\ast})}{\p\eta'}\cos\phi_{\ast}\ud{\phi_{\ast}}\ud{\psi}\\
&=&-\int_{-\pi}^{\pi}\int_{-\pi/2}^{\pi/2}\zeta(\eta',\phi',\psi)(1-\chi(\phi_{\ast}))
\frac{1}{\sin\phi_{\ast}}
\bigg(\v(\eta',\phi_{\ast},\psi)-\bar\v(\eta')-S(\eta',\phi_{\ast},\psi)\bigg)\cos\phi_{\ast}\ud{\phi_{\ast}}\ud{\psi}\no\\
&&-\int_{-\pi}^{\pi}\int_{-\pi/2}^{\pi/2}\zeta(\eta',\phi',\psi)(1-\chi(\phi_{\ast}))
\frac{1}{\sin\phi_{\ast}}
F(\eta',\psi)\cos\phi_{\ast}\frac{\p\v(\eta',\phi_{\ast},\psi)}{\p\phi_{\ast}}\cos\phi_{\ast}\ud{\phi_{\ast}}\ud{\psi}\no\\
&=&\tilde\a_1+\tilde\a_2.\no
\end{eqnarray}
We may directly obtain
\begin{eqnarray}
\\
\abs{\tilde\a_1}&\leq&\int_{-\pi}^{\pi}\int_{-\pi/2}^{\pi/2}\zeta(\eta',\phi',\psi)(1-\chi(\phi_{\ast}))
\frac{1}{\sin\phi_{\ast}}
\bigg(\v(\eta',\phi_{\ast},\psi)-\bar\v(\eta')-S(\eta',\phi_{\ast},\psi)\bigg)\cos\phi_{\ast}\ud{\phi_{\ast}}\ud{\psi}\no\\
&\leq&\frac{\rk}{\d}\abs{\int_{-\pi}^{\pi}\int_{-\pi/2}^{\pi/2}
\bigg(\v(\eta',\phi_{\ast},\psi)-\bar\v(\eta')-S(\eta',\phi_{\ast},\psi)\bigg)\cos\phi_{\ast}\ud{\phi_{\ast}}\ud{\psi}}\no\\
&\leq&C(\d)\bigg(\lnnm{\v}+\lnnm{S}\bigg).\no
\end{eqnarray}
On the other hand, an integration by parts yields
\begin{eqnarray}
\tilde\a_2&=&\int_{-\pi}^{\pi}\int_{-\pi/2}^{\pi/2}\frac{\p}{\p\phi_{\ast}}\bigg(\zeta(\eta',\phi',\psi)(1-\chi(\phi_{\ast}))
\frac{1}{\sin\phi_{\ast}}
F(\eta',\psi)\cos\phi_{\ast}\bigg)\v(\eta',\phi_{\ast},\psi)\cos\phi_{\ast}\ud{\phi_{\ast}}\ud{\psi},
\end{eqnarray}
which further implies
\begin{eqnarray}
\abs{\tilde\a_2}&\leq&\frac{C\e}{\d^2}\lnnm{\v}\leq C(\d)\lnnm{\v}.
\end{eqnarray}
Since we can use substitution to show
\begin{eqnarray}
\int_0^{\eta}\frac{1}{\sin\phi'}\exp(-G_{\eta,\eta',\psi})\ud{\eta'}\leq 1,
\end{eqnarray}
we have
\begin{eqnarray}
\abs{I_2}&\leq&C(\d)\bigg(\lnnm{\v}+\lnnm{S}\bigg)\int_0^{\eta}
\frac{1}{\sin\phi'}\exp(-G_{\eta,\eta',\psi})\ud{\eta'}\\
&\leq&C(\d)\bigg(\lnnm{\v}+\lnnm{S}\bigg).\no
\end{eqnarray}
\ \\
Step 3: Estimate of $I_3$ for $0\leq\sin\phi\leq\d_0$, $\chi(\phi_{\ast})=1$ and $\sqrt{\e\eta'}\geq\sin\phi'$.\\
Based on (\ref{pt 01}), this implies
\begin{eqnarray}
\zeta(\eta',\phi',\psi)\leq C\sqrt{\e\eta'}.\no
\end{eqnarray}
Then combining this with (\ref{pt 02}), we can directly obtain
\begin{eqnarray}
\\
\int_{-\pi}^{\pi}\int_{-\pi/2}^{\pi/2}\frac{\zeta(\eta',\phi',\psi)}{\zeta(\eta',\phi_{\ast},\psi)}\chi(\phi_{\ast})
\a(\eta',\phi_{\ast},\psi)\cos\phi_{\ast}\ud{\phi_{\ast}}\ud{\psi}&\leq&C\int_{-\pi}^{\pi}\int_{-\d}^{\d}
\a(\eta',\phi_{\ast},\psi)\cos\phi_{\ast}\ud{\phi_{\ast}}\ud{\psi}\leq C\d\lnnm{\a}.\no
\end{eqnarray}
Hence, we have
\begin{eqnarray}
\abs{I_3}&\leq&C\d\lnnm{\a}\int_0^{\eta}\frac{1}{\sin\phi'}\exp(-G_{\eta,\eta',\psi})\ud{\eta'}\leq C\d\lnnm{\a}.
\end{eqnarray}
\ \\
Step 4: Estimate of $I_4$ for $0\leq\sin\phi\leq\d_0$, $\chi(\phi_{\ast})=1$, $\sqrt{\e\eta'}\leq\sin\phi'$ and $\sin^2\phi\leq\e(\eta-\eta')$.\\
Based on (\ref{pt 01}), this implies
\begin{eqnarray}
\zeta(\eta',\phi',\psi)\leq C\sin\phi'.
\end{eqnarray}
Based on (\ref{pt 03}), we have
\begin{eqnarray}
-G_{\eta,\eta',\psi}=-\int_{\eta'}^{\eta}\frac{1}{\sin\phi'(y)}\ud{y}&\leq&-\frac{\eta-\eta'}{2\sqrt{\e(\eta-\eta')}}\leq-C\sqrt{\frac{\eta-\eta'}{\e}}.
\end{eqnarray}
Hence, we know
\begin{eqnarray}
\abs{I_4}&\leq&C\int_0^{\eta}\bigg(\int_{-\pi}^{\pi}\int_{-\pi/2}^{\pi/2}\frac{\zeta(\eta',\phi',\psi)}{\zeta(\eta',\phi_{\ast},\psi)}\chi(\phi_{\ast})
\a(\eta',\phi_{\ast},\psi)\cos\phi_{\ast}\ud{\phi_{\ast}}\ud{\psi}\bigg)
\frac{1}{\sin\phi'}\exp(-G_{\eta,\eta',\psi})\ud{\eta'}\\
&\leq&C\int_0^{\eta}\bigg(\int_{-\pi}^{\pi}\int_{-\d}^{\d}\frac{1}{\zeta(\eta',\phi_{\ast},\psi)}
\a(\eta',\phi_{\ast},\psi)\cos\phi_{\ast}\ud{\phi_{\ast}}\ud{\psi}\bigg)
\frac{\zeta(\eta',\phi',\psi)}{\sin\phi'}\exp(-G_{\eta,\eta',\psi})\ud{\eta'}\no\\
&\leq&C\d\lnnm{\a}\int_0^{\eta}\frac{1}{\sqrt{\e\eta'}}\exp(-G_{\eta,\eta',\psi})\ud{\eta'}\no\\
&\leq&C\d\lnnm{\a}\int_0^{\eta}\frac{1}{\sqrt{\e\eta'}}\exp\bigg(-C\sqrt{\frac{\eta-\eta'}{\e}}\bigg)\ud{\eta'}\no
\end{eqnarray}
Define $z=\dfrac{\eta'}{\e}$, which implies $\ud{\eta'}=\e\ud{z}$. Substituting this into above integral, we have
\begin{eqnarray}
\abs{I_4}&\leq&C\d\lnnm{\a}\int_0^{\eta/\e}\frac{1}{\sqrt{z}}\exp\bigg(-C\sqrt{\frac{\eta}{\e}-z}\bigg)\ud{z}\\
&=&C\d\lnnm{\a}\Bigg(\int_0^{1}\frac{1}{\sqrt{z}}\exp\bigg(-C\sqrt{\frac{\eta}{\e}-z}\bigg)\ud{z}
+\int_1^{\eta/\e}\frac{1}{\sqrt{z}}\exp\bigg(-C\sqrt{\frac{\eta}{\e}-z}\bigg)\ud{z}\Bigg).\no
\end{eqnarray}
We can estimate these two terms separately.
\begin{eqnarray}
\int_0^{1}\frac{1}{\sqrt{z}}\exp\bigg(-C\sqrt{\frac{\eta}{\e}-z}\bigg)\ud{z}&\leq&\int_0^{1}\frac{1}{\sqrt{z}}\ud{z}=2.
\end{eqnarray}
\begin{eqnarray}
\int_1^{\eta/\e}\frac{1}{\sqrt{z}}\exp\bigg(-C\sqrt{\frac{\eta}{\e}-z}\bigg)\ud{z}&\leq&\int_1^{\eta/\e}\exp\bigg(-C\sqrt{\frac{\eta}{\e}-z}\bigg)\ud{z}
\overset{t^2=\frac{\eta}{\e}-z}{\leq}2\int_0^{\infty}t\ue^{-Ct}\ud{t}<\infty.
\end{eqnarray}
Hence, we know
\begin{eqnarray}
\abs{I_4}&\leq&C\d\lnnm{\a}.
\end{eqnarray}
\ \\
Step 5: Estimate of $I_5$ for $0\leq\sin\phi\leq\d_0$, $\chi(\phi_{\ast})=1$, $\sqrt{\e\eta'}\leq\sin\phi'$ and $\sin^2\phi\geq\e(\eta-\eta')$.\\
Based on (\ref{pt 01}), this implies
\begin{eqnarray}
\zeta(\eta',\phi',\psi)\leq C\sin\phi'.\no
\end{eqnarray}
Based on (\ref{pt 03}), we have
\begin{eqnarray}
-G_{\eta,\eta',\psi}=-\int_{\eta'}^{\eta}\frac{1}{\sin\phi'(y)}\ud{y}&\leq&-\frac{C(\eta-\eta')}{\sin\phi}.
\end{eqnarray}
Hence, we have
\begin{eqnarray}
\abs{I_5}&\leq& C\lnnm{\a}\int_0^{\eta}\bigg(\int_{-\pi}^{\pi}\int_{-\d}^{\d}\frac{1}{\zeta(\eta',\phi_{\ast},\psi)}
\cos\phi_{\ast}\ud{\phi_{\ast}}\ud{\psi}\bigg)
\exp\left(-\frac{C(\eta-\eta')}{\sin\phi}\right)\ud{\eta'}
\end{eqnarray}
Here, we use a different way to estimate the inner integral. We use substitution to find
\begin{eqnarray}
\int_{-\d}^{\d}\frac{1}{\zeta(\eta',\phi_{\ast},\psi)}\cos\phi_{\ast}
\ud{\phi_{\ast}}
&\leq&\int_{-\d}^{\d}\frac{1}{\bigg(\rk^2-(\rk-\e\eta')^2\cos\phi_{\ast}^2\bigg)^{1/2}}
\ud{\phi_{\ast}}\\
&\overset{\sin\phi_{\ast}\ small}{\leq}&C\int_{-\d}^{\d}\frac{\cos\phi_{\ast}}{\bigg(\rk^2-(\rk-\e\eta')^2\cos\phi_{\ast}^2\bigg)^{1/2}}
\ud{\phi_{\ast}}\no\\
&=&C\int_{-\d}^{\d}\frac{\cos\phi_{\ast}}{\bigg(\rk^2-(\rk-\e\eta')^2+(\rk-\e\eta')^2\sin\phi_{\ast}^2\bigg)^{1/2}}
\ud{\phi_{\ast}}\no\\
&\overset{y=\sin\phi_{\ast}}{=}&C\int_{-\d}^{\d}\frac{1}{\bigg(\rk^2-(\rk-\e\eta')^2+(\rk-\e\eta')^2y^2\bigg)^{1/2}}
\ud{y}.\no
\end{eqnarray}
Define
\begin{eqnarray}
p&=&\sqrt{\rk^2-(\rk-\e\eta')^2}=\sqrt{2\rk\e\eta'-\e^2\eta'^2}\leq C\sqrt{\e\eta'},\\
q&=&\rk-\e\eta'\geq C,\\
r&=&\frac{p}{q}\leq C\sqrt{\e\eta'}.
\end{eqnarray}
Then we have
\begin{eqnarray}
\int_{-\d}^{\d}\frac{1}{\zeta(\eta',\phi_{\ast},\psi)}\ud{\phi_{\ast}}&\leq&C\int_{-\d}^{\d}\frac{1}{(p^2+q^2y^2)^{1/2}}\ud{y}\\
&\leq&C\int_{-2}^{2}\frac{1}{(p^2+q^2y^2)^{1/2}}\ud{y}\leq C\int_{-2}^{2}\frac{1}{(r^2+y^2)^{1/2}}\ud{y}\no\\
&\leq&C\int_{0}^{2}\frac{1}{(r^2+y^2)^{1/2}}\ud{y}=\bigg(\ln(y+\sqrt{r^2+y^2})-\ln(r)\bigg)\bigg|_0^{2}\no\\
&\leq&C\bigg(\ln(2+\sqrt{r^2+4})-\ln{r}\bigg)\leq C\bigg(1+\ln(r)\bigg)\no\\
&\leq&C\bigg(1+\abs{\ln(\e)}+\abs{\ln(\eta')}\bigg).\no
\end{eqnarray}
Hence, we know
\begin{eqnarray}
\abs{I_5}&\leq&C\lnnm{\a}\int_0^{\eta}\bigg(1+\abs{\ln(\e)}+\abs{\ln(\eta')}\bigg)
\exp\left(-\frac{C(\eta-\eta')}{\sin\phi}\right)\ud{\eta'}
\end{eqnarray}
We may directly compute
\begin{eqnarray}
\abs{\int_0^{\eta}\bigg(1+\abs{\ln(\e)}\bigg)
\exp\left(-\frac{C(\eta-\eta')}{\sin\phi}\right)\ud{\eta'}}\leq C\sin\phi(1+\abs{\ln(\e)}).
\end{eqnarray}
Hence, we only need to estimate
\begin{eqnarray}
\abs{\int_0^{\eta}\abs{\ln(\eta')}
\exp\left(-\frac{C(\eta-\eta')}{\sin\phi}\right)\ud{\eta'}}.
\end{eqnarray}
If $\eta\leq 2$, using Cauchy's inequality, we have
\begin{eqnarray}
\abs{\int_0^{\eta}\abs{\ln(\eta')}
\exp\left(-\frac{C(\eta-\eta')}{\sin\phi}\right)\ud{\eta'}}
&\leq&\bigg(\int_0^{\eta}\ln^2(\eta')\ud{\eta'}\bigg)^{1/2}\bigg(\int_0^{\eta}
\exp\left(-\frac{2C(\eta-\eta')}{\sin\phi}\right)\ud{\eta'}\bigg)^{1/2}\\
&\leq&\bigg(\int_0^{2}\ln^2(\eta')\ud{\eta'}\bigg)^{1/2}\bigg(\int_0^{\eta}
\exp\left(-\frac{2C(\eta-\eta')}{\sin\phi}\right)\ud{\eta'}\bigg)^{1/2}\no\\
&\leq&\sqrt{\sin\phi}.\no
\end{eqnarray}
If $\eta\geq 2$, we decompose and apply Cauchy's inequality to obtain
\begin{eqnarray}
&&\abs{\int_0^{\eta}\abs{\ln(\eta')}
\exp\left(-\frac{C(\eta-\eta')}{\sin\phi}\right)\ud{\eta'}}\\
&\leq&\abs{\int_0^{2}\abs{\ln(\eta')}
\exp\left(-\frac{C(\eta-\eta')}{\sin\phi}\right)\ud{\eta'}}+\abs{\int_2^{\eta}\ln(\eta')
\exp\left(-\frac{C(\eta-\eta')}{\sin\phi}\right)\ud{\eta'}}\no\\
&\leq&\bigg(\int_0^{2}\ln^2(\eta')\ud{\eta'}\bigg)^{1/2}\bigg(\int_0^{2}
\exp\left(-\frac{2C(\eta-\eta')}{\sin\phi}\right)\ud{\eta'}\bigg)^{1/2}+\ln(2)\abs{\int_2^{\eta}
\exp\left(-\frac{C(\eta-\eta')}{\sin\phi}\right)\ud{\eta'}}\no\\
&\leq&C\bigg(\sqrt{\sin\phi}+\sin\phi\bigg)\leq C\sqrt{\sin\phi}.\no
\end{eqnarray}
Hence, we have
\begin{eqnarray}
\abs{I_5}\leq C(1+\abs{\ln(\e)})\sqrt{\d_0}\lnnm{\a}.
\end{eqnarray}
\ \\
Step 6: Synthesis.\\
Collecting all the terms in previous steps, we have proved
\begin{eqnarray}
\abs{I}&\leq&C(1+\abs{\ln(\e)})\sqrt{\d_0}\lnnm{\a}+C\d\lnnm{\a}\\
&&+\frac{C}{\d_0^4}\bigg(\lnnm{\v}+\lnm{\frac{\p p}{\p\phi}}+\lnnm{\frac{\p S}{\p\phi}}+\lnnm{\frac{\p S}{\p\eta}}\bigg)+C(\d)\bigg(\lnnm{\v}+\lnnm{S}\bigg).\no
\end{eqnarray}
Therefore, we know
\begin{eqnarray}
\abs{\a}_{I}&\leq&\lnm{p_{\a}}+\lnnm{S_{\a}}+C(1+\abs{\ln(\e)})\sqrt{\d_0}\lnnm{\a}+C\d\lnnm{\a}\\
&&+\frac{C}{\d_0^4}\bigg(\lnnm{\v}+\lnm{\frac{\p p}{\p\phi}}+\lnnm{\frac{\p S}{\p\phi}}+\lnnm{\frac{\p S}{\p\eta}}\bigg)+C(\d)\bigg(\lnnm{\v}+\lnnm{S}\bigg).\no
\end{eqnarray}

\subsubsection{Region II: $\sin\phi<0$ and $\abs{E(\eta,\phi,\psi)}\leq \ue^{-V(L)}$}

\begin{eqnarray}
\k[p_{\a}]&=&p_{\a}(\phi'(0),\psi)\exp(-G_{L,0,\psi}-G_{L,\eta,\psi})\\
\t[\tilde\a+S_{\a}]&=&\int_0^{L}\frac{(\tilde\a+S)(\eta',\phi'(\eta'),\psi)}{\sin(\phi'(\eta'))}
\exp(-G_{L,\eta',\psi}-G_{L,\eta,\psi})\ud{\eta'}\\
&&+\int_{\eta}^{L}\frac{(\tilde\a+S)(\eta',\rr\phi'(\eta'),\psi)}{\sin(\phi'(\eta'))}\exp(-G_{\eta',\eta,\psi})\ud{\eta'}\nonumber.
\end{eqnarray}
A direct computation reveals
\begin{eqnarray}
\abs{\k[p_{\a}]}&\leq&\lnm{p_{\a}},\\
\abs{\t[S_{\a}]}&\leq&\lnnm{S_{\a}}.
\end{eqnarray}
Hence, we only need to estimate $II=\t[\tilde\a]$. In particular, we can decompose
\begin{eqnarray}
\t[\tilde\a]&=&\int_0^{L}\frac{\tilde\a(\eta',\phi'(\eta'),\psi)}{\sin(\phi'(\eta'))}\exp(-G_{L,\eta',\psi}-G_{L,\eta,\psi})\ud{\eta'}
+\int_{\eta}^{L}\frac{\tilde\a(\eta',\rr\phi'(\eta'),\psi)}{\sin(\phi'(\eta'))}\exp(-G_{\eta',\eta,\psi})\ud{\eta'}\\
&=&\int_0^{\eta}\frac{\tilde\a(\eta',\phi'(\eta'),\psi)}{\sin(\phi'(\eta'))}\exp(-G_{L,\eta',\psi}-G_{L,\eta,\psi})\ud{\eta'}\no\\
&&+\int_{\eta}^{L}\frac{\tilde\a(\eta',\phi'(\eta'),\psi)}{\sin(\phi'(\eta'))}\exp(-G_{L,\eta',\psi}-G_{L,\eta,\psi})\ud{\eta'}
+\int_{\eta}^{L}\frac{\tilde\a(\eta',\rr\phi'(\eta'),\psi)}{\sin(\phi'(\eta'))}\exp(-G_{\eta',\eta,\psi})\ud{\eta'}.\no
\end{eqnarray}
The integral $\displaystyle\int_0^{\eta}\cdots$ can be estimated as in Region I, so we only need to estimate the integral $\displaystyle\int_{\eta}^L\cdots$. Also, noting that fact that \begin{eqnarray}
\exp(-G_{L,\eta',\psi}-G_{L,\eta,\psi})\leq \exp(-G_{\eta',\eta,\psi}),
\end{eqnarray}
we only need to estimate
\begin{eqnarray}
\int_{\eta}^{L}\frac{\tilde\a(\eta',\rr\phi'(\eta'),\psi)}{\sin(\phi'(\eta'))}\exp(-G_{\eta',\eta,\psi})\ud{\eta'}.
\end{eqnarray}
Here the proof is almost identical to that in Region I, so we only point out the key differences.\\
\ \\
Step 0: Preliminaries.\\
We need to update one key result. For $0\leq\eta\leq\eta'$,
\begin{eqnarray}
\sin\phi'&=&\sqrt{1-\cos^2\phi'}\leq\sqrt{1-\bigg(\frac{\rk-\e\eta}{\rk-\e\eta'}\bigg)^2\cos^2\phi}\\
&=&\frac{\sqrt{(\rk-\e\eta')^2\sin^2\phi+(2\rk-\e\eta-\e\eta')(\e\eta'-\e\eta)\cos^2\phi}}{\rk-\e\eta'}\no\\
&\leq&\abs{\sin\phi}.
\end{eqnarray}
Then we have
\begin{eqnarray}\label{pt 04}
-\int_{\eta}^{\eta'}\frac{1}{\sin\phi'(y)}\ud{y}&\leq&-\frac{\eta'-\eta}{\abs{\sin\phi}}.
\end{eqnarray}
In the following, we will divide the estimate of $II$ into several cases based on the value of $\sin\phi$, $\sin\phi'$ and $\e\eta'$. We write
\begin{eqnarray}
II&=&\int_{\eta}^L\id_{\{\sin\phi\leq-\d_0\}}+\int_{\eta}^L\id_{\{-\d_0\leq\sin\phi\leq0\}}\id_{\{\chi(\phi_{\ast})<1\}}\\
&&+\int_{\eta}^L\id_{\{-\d_0\leq\sin\phi\leq0\}}\id_{\{\chi(\phi_{\ast})=1\}}\id_{\{\sqrt{\e\eta'}\geq\sin\phi'\}}
+\int_{\eta}^L\id_{\{-\d_0\leq\sin\phi\leq0\}}\id_{\{\chi(\phi_{\ast})=1\}}\id_{\{\sqrt{\e\eta'}\leq\sin\phi'\}}\no\\
&=&II_1+II_2+II_3+II_4.\no
\end{eqnarray}
\ \\
Step 1: Estimate of $II_1$ for $\sin\phi\leq-\d_0$.\\
We first estimate $\sin\phi'$. Along the characteristics, we know
\begin{eqnarray}
\ue^{-V(\eta',\psi)}\cos\phi'=\ue^{-V(\eta,\psi)}\cos\phi,
\end{eqnarray}
which implies
\begin{eqnarray}
\cos\phi'&=&\ue^{V(\eta',\psi)-V(\eta,\psi)}\cos\phi\leq \ue^{V(L,\psi)-V(0,\psi)}\cos\phi= \ue^{V(L,\psi)-V(0,\psi)}\sqrt{1-\d_0^2}.
\end{eqnarray}
Based on Lemma \ref{rt lemma 1}, we can further deduce that
\begin{eqnarray}
\cos\phi'\leq \bigg(1-\frac{\e^{1-n}}{\rk'}\bigg)^{-1}\sqrt{1-\d_0^2}.
\end{eqnarray}
Then we have
\begin{eqnarray}
\sin\phi'\geq\sqrt{1-\bigg(1-\frac{\e^{1-n}}{\rk'}\bigg)^{-2}(1-\d_0^2)}\geq \d_0-\e^{\frac{1}{2}-\frac{n}{2}}>\frac{\d_0}{2},
\end{eqnarray}
when $\e$ is sufficiently small.
Based on Lemma \ref{pt lemma 2}, we know
\begin{eqnarray}
\abs{\sin\phi\frac{\p\v}{\p\eta}}\leq C\bigg(1+\frac{1}{\d_0^3}\bigg)\bigg(\lnnm{\v}+\lnnm{\frac{\p S}{\p\phi}}+\lnnm{\frac{\p S}{\p\eta}}\bigg).
\end{eqnarray}
Hence, we have
\begin{eqnarray}
\abs{II_1}\leq  \frac{1}{\abs{\sin\phi}}\abs{\frac{\p\v}{\p\eta}}\leq \frac{C}{\d_0^4}\bigg(\lnnm{\v}+\lnnm{\frac{\p S}{\p\phi}}+\lnnm{\frac{\p S}{\p\eta}}\bigg).
\end{eqnarray}
\ \\
Step 2: Estimate of $II_2$ for $-\d_0\leq\sin\phi\leq0$ and $\chi(\phi_{\ast})<1$.\\
This is similar to the estimate of $I_2$ based on the integral
\begin{eqnarray}
\int_{\eta}^{L}\frac{1}{\sin\phi'}\exp(-G_{\eta',\eta,\psi})\ud{\eta'}\leq 1.
\end{eqnarray}
Then we have
\begin{eqnarray}
\abs{II_2}
&\leq&C(\d)\bigg(\lnnm{\v}+\lnnm{S}\bigg).
\end{eqnarray}
\ \\
Step 3: Estimate of $II_3$ for $-\d_0\leq\sin\phi\leq0$, $\chi(\phi_{\ast})=1$ and $\sqrt{\e\eta'}\geq\sin\phi'$.\\
This is identical to the estimate of $I_4$, we have
\begin{eqnarray}
\abs{II_3}&\leq&C\d\lnnm{\a}.
\end{eqnarray}
\ \\
Step 4: Estimate of $II_4$ for $-\d_0\leq\sin\phi\leq0$, $\chi(\phi_{\ast})=1$ and $\sqrt{\e\eta'}\leq\sin\phi'$.\\
This step is different. We do not need to further decompose the cases.
Based on (\ref{pt 04}), we have,
\begin{eqnarray}
-G_{\eta,\eta'}&\leq&-\frac{\eta'-\eta}{\abs{\sin\phi}}.
\end{eqnarray}
Then following the same argument in estimating $I_5$, we obtain
\begin{eqnarray}
\abs{II_4}&\leq&C\lnnm{\a}\int_{\eta}^{L}\bigg(1+\abs{\ln(\e)}+\abs{\ln(\eta')}\bigg)
\exp\left(-\frac{\eta'-\eta}{\abs{\sin\phi}}\right)\ud{\eta'}
\end{eqnarray}
If $\eta\geq 2$, we directly obtain
\begin{eqnarray}
\abs{\int_{\eta}^{L}\abs{\ln(\eta')}
\exp\left(-\frac{\eta'-\eta}{\abs{\sin\phi}}\right)\ud{\eta'}}&\leq& \abs{\int_{2}^{L}\ln(\eta')
\exp\left(-\frac{\eta'-\eta}{\abs{\sin\phi}}\right)\ud{\eta'}}\\
&\leq&\ln(2)\abs{\int_{2}^{L}
\exp\left(-\frac{\eta'-\eta}{\abs{\sin\phi}}\right)\ud{\eta'}}\no\\
&\leq&C\sqrt{\abs{\sin\phi}}.\no
\end{eqnarray}
If $\eta\leq 2$, we decompose as
\begin{eqnarray}
&&\abs{\int_{\eta}^{L}\abs{\ln(\eta')}
\exp\left(-\frac{\eta'-\eta}{\abs{\sin\phi}}\right)\ud{\eta'}}\\
&\leq&\abs{\int_{\eta}^{2}\abs{\ln(\eta')}
\exp\left(-\frac{\eta'-\eta}{\abs{\sin\phi}}\right)\ud{\eta'}}+\abs{\int_{2}^{L}\abs{\ln(\eta')}
\exp\left(-\frac{\eta'-\eta}{\abs{\sin\phi}}\right)\ud{\eta'}}.\no
\end{eqnarray}
The second term is identical to the estimate in $\eta\geq2$. We apply Cauchy's inequality to the first term
\begin{eqnarray}
\abs{\int_{\eta}^{2}\abs{\ln(\eta')}
\exp\left(-\frac{\eta'-\eta}{\abs{\sin\phi}}\right)\ud{\eta'}}
&\leq&\bigg(\int_{\eta}^{2}\ln^2(\eta')\ud{\eta'}\bigg)^{1/2}\bigg(\int_{\eta}^{2}
\exp\left(-\frac{2(\eta'-\eta)}{\abs{\sin\phi}}\right)\ud{\eta'}\bigg)^{1/2}\\
&\leq&\bigg(\int_0^{2}\ln^2(\eta')\ud{\eta'}\bigg)^{1/2}\bigg(\int_{\eta}^{2}
\exp\left(-\frac{2(\eta'-\eta)}{\abs{\sin\phi}}\right)\ud{\eta'}\bigg)^{1/2}\no\\
&\leq&C\sqrt{\abs{\sin\phi}}.\no
\end{eqnarray}
Hence, we have
\begin{eqnarray}
\abs{II_4}\leq C(1+\abs{\ln(\e)})\sqrt{\d_0}\lnnm{\a}.
\end{eqnarray}
\ \\
Step 5: Synthesis.\\
Collecting all the terms in previous steps, we have proved
\begin{eqnarray}
\abs{II}&\leq&C(1+\abs{\ln(\e)})\sqrt{\d_0}\lnnm{\a}+C\d\lnnm{\a}\\
&&+\frac{C}{\d_0^4}\bigg(\lnnm{\v}+\lnm{\frac{\p p}{\p\phi}}+\lnnm{\frac{\p S}{\p\phi}}+\lnnm{\frac{\p S}{\p\eta}}\bigg)+C(\d)\bigg(\lnnm{\v}+\lnnm{S}\bigg).\no
\end{eqnarray}
Therefore, we know
\begin{eqnarray}
\abs{\a}_{II}&\leq&\lnnm{S_{\a}}+\lnm{p_{\a}}+C(1+\abs{\ln(\e)})\sqrt{\d_0}\lnnm{\a}+C\d\lnnm{\a}\\
&&+\frac{C}{\d_0^4}\bigg(\lnnm{\v}+\lnm{\frac{\p p}{\p\phi}}+\lnnm{\frac{\p S}{\p\phi}}+\lnnm{\frac{\p S}{\p\eta}}\bigg)+C(\d)\bigg(\lnnm{\v}+\lnnm{S}\bigg).\no
\end{eqnarray}

\subsubsection{Region III: $\sin\phi<0$ and $\abs{E(\eta,\phi,\psi)}\geq \ue^{-V(L)}$}

We still ignore $\psi$ dependence. Based on \cite[Lemma 4.7, Lemma 4.8]{AA003}, we still have
\begin{eqnarray}
\abs{\k[p_{\a}]}&\leq&\lnm{p_{\a}},\\
\abs{\t[S_{\a}]}&\leq&\lnnm{S_{\a}}.
\end{eqnarray}
Hence, we only need to estimate $III=\t[\tilde\a]$. Note that $\abs{E(\eta,\phi,\psi)}\geq \ue^{-V(L,\psi)}$ implies
\begin{eqnarray}
\ue^{-V(\eta,\psi)}\cos\phi\geq \ue^{-V(L,\psi)}.
\end{eqnarray}
Hence, based on Lemma \ref{rt lemma 1}, we can further deduce that
\begin{eqnarray}
\cos\phi&\geq&\ue^{V(\eta,\psi)-V(L,\psi)}\geq \ue^{V(0,\psi)-V(L,\psi)}\geq \bigg(1-\frac{\e^{1-n}}{\rk'}\bigg).
\end{eqnarray}
Hence, we know
\begin{eqnarray}
\abs{\sin\phi}\leq\sqrt{1-\bigg(1-\frac{\e^{1-n}}{\rk'}\bigg)^2}\leq \e^{\frac{1}{2}-\frac{n}{2}}.
\end{eqnarray}
Hence, when $\e$ is sufficiently small, we always have
\begin{eqnarray}
\abs{\sin\phi}\leq \e^{\frac{1}{2}-\frac{n}{2}}\leq \d_0.
\end{eqnarray}
This means we do not need to bother with the estimate of $\sin\phi\leq-\d_0$ as Step 1 in estimating $I$ and $II$.
Since we can decompose
\begin{eqnarray}
\t[\tilde\a]&=&\int_0^{\eta}\frac{\tilde\a(\eta',\phi'(\eta'),\psi)}{\sin(\phi'(\eta'))}
\exp(-G_{\eta^+,\eta',\psi}-G_{\eta^+,\eta,\psi})\ud{\eta'}\\
&&\bigg(\int_{\eta}^{\eta^+}\frac{\tilde\a(\eta',\phi'(\eta'),\psi)}{\sin(\phi'(\eta'))}
\exp(-G_{\eta^+,\eta',\psi}-G_{\eta^+,\eta,\psi})\ud{\eta'}\no\\
&&+
\int_{\eta}^{\eta^+}\frac{(\tilde\a+S_{\a})(\eta',\rr\phi'(\eta'),\psi)}{\sin(\phi'(\eta'))}\exp(-G_{\eta',\eta,\psi})\ud{\eta'}\bigg)\nonumber.
\end{eqnarray}
Then the integral $\displaystyle\int_0^{\eta}(\cdots)$ is similar to the argument in Region I, and the integral $\displaystyle\int_{\eta}^{\eta^+}(\cdots)$ is similar to the argument in Region II. Hence, combining the methods in Region I and Region II, we can show the desired result, i.e.
\begin{eqnarray}
\abs{\a}_{III}&\leq&\lnm{p_{\a}}+\lnnm{S_{\a}}+C(1+\abs{\ln(\e)})\sqrt{\d_0}\lnnm{\a}+C\d\lnnm{\a}\\
&&+C(\d)\bigg(\lnnm{\v}+\lnnm{S}\bigg).\no
\end{eqnarray}

\subsubsection{Estimate of Normal Derivative}

Combining the analysis in these three regions, we have
\begin{eqnarray}
\abs{\a}&\leq&\lnm{p_{\a}}+\lnnm{S_{\a}}+C(1+\abs{\ln(\e)})\sqrt{\d_0}\lnnm{\a}+C\d\lnnm{\a}\\
&&+\frac{C}{\d_0^4}\bigg(\lnnm{\v}+\lnm{\frac{\p p}{\p\phi}}+\lnnm{\frac{\p S}{\p\phi}}+\lnnm{\frac{\p S}{\p\eta}}\bigg)+C(\d)\bigg(\lnnm{\v}+\lnnm{S}\bigg).\no
\end{eqnarray}
Taking supremum over all $(\eta,\phi,\psi)$, we have
\begin{eqnarray}\label{pt 05}
\lnnm{\a}&\leq&\lnm{p_{\a}}+\lnnm{S_{\a}}+C(1+\abs{\ln(\e)})\sqrt{\d_0}\lnnm{\a}+C\d\lnnm{\a}\\
&&+\frac{C}{\d_0^4}\bigg(\lnnm{\v}+\lnm{\frac{\p p}{\p\phi}}+\lnnm{\frac{\p S}{\p\phi}}+\lnnm{\frac{\p S}{\p\eta}}\bigg)\no\\
&&+C(\d)\bigg(\lnnm{\v}+\lnnm{S}\bigg).\no
\end{eqnarray}
Then we choose these constants to perform absorbing argument. First we choose $0<\d<<1$ sufficiently small such that
\begin{eqnarray}
C\d\leq\frac{1}{4}.
\end{eqnarray}
Then we take $\d_0=\d\abs{\ln(\e)}^{-2}$ such that
\begin{eqnarray}
C(1+\abs{\ln(\e)})\sqrt{\d_0}\leq 2C\d\leq\frac{1}{2}.
\end{eqnarray}
for $\e$ sufficiently small. Note that this mild decay of $\d_0$ with respect to $\e$ also justifies the assumption in Case III and the proof of Lemma \ref{pt lemma 2} that
\begin{eqnarray}
\e^{\frac{1}{2}-\frac{n}{2}}\leq \frac{\d_0}{2},
\end{eqnarray}
for $\e$ sufficiently small. Here since $\d$ and $C$ are independent of $\e$, there is no circulant argument. Hence, we can absorb all the term related to $\lnnm{\a}$ on the right-hand side of (\ref{pt 05}) to the left-hand side to obtain
\begin{eqnarray}
\lnnm{\a}&\leq&C\bigg(\lnm{p_{\a}}+\lnnm{S_{\a}}\bigg)\\
&&+C\abs{\ln(\e)}^8\bigg(\lnnm{\v}+\lnm{\frac{\p p}{\p\phi}}+\lnnm{S}+\lnnm{\frac{\p S}{\p\phi}}+\lnnm{\frac{\p S}{\p\eta}}\bigg).\no
\end{eqnarray}

\subsection{Mild Formulation of Velocity Derivative}

Consider the general $\e$-Milne problem for $\b=\zeta\dfrac{\p\v}{\p\phi}$ as
\begin{eqnarray}
\left\{
\begin{array}{rcl}\displaystyle
\sin\phi\frac{\p\b}{\p\eta}+F(\eta,\psi)\cos\phi\frac{\p
\b}{\p\phi}+\b&=&S_{\b},\\
\b(0,\phi,\psi)&=&p_{\b}(\phi,\psi)\ \ \text{for}\ \ \sin\phi>0,\\
\b(L,\phi,\psi)&=&\b(L,\rr\phi,\psi),
\end{array}
\right.
\end{eqnarray}
where $p_{\b}$ and $S_{\b}$ will be specified later. This is much simpler than normal derivative, since we do not have $\tilde\b$. Then by a direct argument that
\begin{eqnarray}
\abs{\k[p_{\b}]}&\leq&\lnm{p_{\b}},\\
\abs{\t[S_{\b}]}&\leq&\lnnm{S_{\b}}.
\end{eqnarray}
we can get the desired result.
\begin{lemma}\label{pt lemma 4}
We have
\begin{eqnarray}
\lnnm{\b}&\leq&\lnm{p_{\b}}+\lnnm{S_{\b}}.
\end{eqnarray}
\end{lemma}

\subsection{Estimate of Derivatives}

In this subsection, we combine above a priori estimates of normal and velocity derivatives.
\begin{theorem}\label{pt additional 1}
Assume (\ref{Milne bounded}), (\ref{Milne decay}), (\ref{Milne bounded.}) and (\ref{Milne decay.}). The normal and velocity derivatives of $\v$ are well-defined a.e. and satisfy
\begin{eqnarray}
\lnnm{\zeta\frac{\p\v}{\p\eta}}+\lnnm{\zeta\frac{\p\v}{\p\eta}}\leq C\abs{\ln(\e)}^8.
\end{eqnarray}
\end{theorem}
\begin{proof}
Based on the analysis in \cite{AA007}, derivatives are a.e. well-defined. Collecting the estimates for $\a$ and $\b$ in Lemma \ref{pt lemma 3} and Lemma \ref{pt lemma 4}, we have
\begin{eqnarray}\label{pt 06}
\lnnm{\a}&\leq&C\bigg(\lnm{p_{\a}}+\lnnm{S_{\a}}\bigg)\label{pt 06}\\
&&+C\abs{\ln(\e)}^8\bigg(\lnnm{\v}+\lnm{\frac{\p p}{\p\phi}}+\lnnm{S}+\lnnm{\frac{\p S}{\p\phi}}+\lnnm{\frac{\p S}{\p\eta}}\bigg),\no\\
\lnnm{\b}&\leq&\lnm{p_{\b}}+\lnnm{S_{\b}}.\label{pt 07}
\end{eqnarray}
Taking derivatives on both sides of (\ref{Milne difference problem}) and multiplying $\zeta$, based on Lemma \ref{rt lemma 2}, we have
\begin{eqnarray}
p_{\a}&=&\e\cos\phi\frac{\p p}{\p\phi}+p-\bar\v(0),\\
p_{\b}&=&\sin\phi\frac{\p p}{\p\phi},\\
S_{\a}&=&\frac{\p{F}}{\p{\eta}}\b\cos\phi+\zeta\frac{\p S}{\p\eta},\\
S_{\b}&=&\a\cos\phi+F\b\sin\phi+\zeta\frac{\p S}{\p\phi}.
\end{eqnarray}
Since $\abs{F(\eta)}+\abs{\dfrac{\p F}{\p\eta}}\leq \e$, by absorbing $\a$ and $\b$ on the right-hand side of (\ref{pt 06}) and (\ref{pt 07}), we derive
\begin{eqnarray}
\a&\leq&C\abs{\ln(\e)}^8,\\
\b&\leq&C\abs{\ln(\e)}^8.
\end{eqnarray}
\end{proof}

\begin{theorem}\label{pt additional 2}
Assume (\ref{Milne bounded}), (\ref{Milne decay}), (\ref{Milne bounded.}) and (\ref{Milne decay.}). The normal and velocity derivatives of $\v$ are well-defined a.e. and satisfy
\begin{eqnarray}
\lnnm{\ue^{K_0\eta}\zeta\frac{\p\v}{\p\eta}}+\lnnm{\ue^{K_0\eta}\zeta\frac{\p\v}{\p\phi}}\leq C\abs{\ln(\e)}^8.
\end{eqnarray}
\end{theorem}
\begin{proof}
This proof is almost identical to Theorem \ref{pt additional 1}. The only difference is that $S_{\a}$ is added by $K_0\a\sin\phi$ and $S_{\b}$ added by $K_0\b\sin\phi$. When $K_0$ is sufficiently small, we can also absorb them into the left-hand side. Hence, this is obvious.
\end{proof}

\begin{corollary}\label{pt corollary}
Assume (\ref{Milne bounded}), (\ref{Milne decay}), (\ref{Milne bounded.}) and (\ref{Milne decay.}). We have
\begin{eqnarray}
\lnnm{\ue^{K_0\eta}\sin\phi\frac{\p\v}{\p\eta}}\leq C\abs{\ln(\e)}^8.
\end{eqnarray}
\end{corollary}
\begin{proof}
This is a natural result of Theorem \ref{pt additional 2} since $\zeta(\eta,\phi,\psi)\geq \abs{\sin\phi}$.
\end{proof}

Now we pull $\tau_i$ for $i=1,2$ and $\psi$ dependence back and study the tangential derivatives and velocity derivative.
\begin{theorem}\label{Milne tangential}
Assume (\ref{Milne bounded}), (\ref{Milne decay}), (\ref{Milne bounded.}) and (\ref{Milne decay.}). We have for $i=1,2$,
\begin{eqnarray}
\lnnm{\ue^{K_0\eta}\frac{\p\v}{\p\tau_i}}&\leq& C\abs{\ln(\e)}^8,\\
\lnnm{\ue^{K_0\eta}\frac{\p\v}{\p\psi}}&\leq& C\abs{\ln(\e)}^8
\end{eqnarray}
\end{theorem}
\begin{proof}
Following a similar fashion in proof of Theorem \ref{pt additional 2}, using iteration and characteristics, we can show $\dfrac{\p\v}{\p\tau_i}$ is a.e. well-defined, so here we focus on the a priori estimate. Let $\w=\dfrac{\p\v}{\p\tau_i}$. Taking $\tau_i$ derivative on both sides of (\ref{Milne difference problem}), we have $\w$ satisfies the equation
\begin{eqnarray}\label{Milne tangential problem}
\left\{
\begin{array}{rcl}\displaystyle
\sin\phi\frac{\p \w}{\p\eta}+F(\eta,\psi)\cos\phi\frac{\p
\w}{\p\phi}+\w-\bar\w&=&\dfrac{\p S}{\p\tau_i}+\e\bigg(\dfrac{\p_{\tau_i}\rk_1\sin^2\psi}{(\rk_1-\e\eta)^2}+\dfrac{\p_{\tau_i}\rk_2\cos^2\psi}{(\rk_2-\e\eta)^2}\bigg)\bigg(\cos\phi\dfrac{\p
\v}{\p\phi}\bigg),\\\rule{0ex}{1.5em}
\w(0,\phi,\psi)&=&\dfrac{\p p}{\p\tau_i}(\phi,\psi)\ \ \text{for}\ \ \sin\phi>0,\\\rule{0ex}{1.5em}
\w(L,\phi,\psi)&=&\w(L,\rr\phi,\psi).
\end{array}
\right.
\end{eqnarray}
Our assumptions on $S$ verify
\begin{eqnarray}
\lnnm{\ue^{K_0\eta}\frac{\p S}{\p\tau_i}}&\leq&C.
\end{eqnarray}
For $\eta\in[0,L]$, we have
\begin{eqnarray}
\lnnm{\dfrac{\p_{\tau_i}\rk_1\sin^2\psi}{(\rk_1-\e\eta)^2}+\dfrac{\p_{\tau_i}\rk_2\cos^2\psi}{(\rk_2-\e\eta)^2}}\leq C,
\end{eqnarray}
and
\begin{eqnarray}
C_1\e\leq F(\eta,\psi)\leq C_2\e.
\end{eqnarray}
Based on Corollary \ref{pt corollary} and the equation (\ref{Milne difference problem}), we know
\begin{eqnarray}
\lnnm{\ue^{K_0\eta}\bigg(\e\cos\phi\frac{\p\v}{\p\phi}\bigg)}\leq C\abs{\ln(\e)}^8.
\end{eqnarray}
Therefore, the source term in the equation (\ref{Milne tangential problem}) is in $L^{\infty}$ and decays exponentially. By Theorem \ref{Milne Decay Theorem}, we have
\begin{eqnarray}
\lnnm{\ue^{K_0\eta}(\w-\w_L)}\leq C\abs{\ln(\e)}^8,
\end{eqnarray}
for some constant $\w_L$. It is easy to see this $\w_L$ must be zero due to decay of $\v$. Similarly, let $\w'=\dfrac{\p\v}{\p\psi}$. Taking $\psi$ derivative on both sides of (\ref{Milne difference problem}), we have $\w'$ satisfies the equation
\begin{eqnarray}
\left\{
\begin{array}{rcl}\displaystyle
\sin\phi\frac{\p \w'}{\p\eta}+F(\eta,\psi)\cos\phi\frac{\p
\w'}{\p\phi}+\w'&=&\dfrac{\p S}{\p\psi}+\e\bigg(\dfrac{\sin(2\psi)}{\rk_1-\e\eta}-\dfrac{\sin(2\psi)}{\rk_2-\e\eta}\bigg)\bigg(\cos\phi\dfrac{\p
\v}{\p\phi}\bigg),\\\rule{0ex}{1.5em}
\w'(0,\phi,\psi)&=&\dfrac{\p p}{\p\psi}(\phi,\psi)\ \ \text{for}\ \ \sin\phi>0,\\\rule{0ex}{1.5em}
\w'(L,\phi,\psi)&=&\w'(L,\rr\phi,\psi).
\end{array}
\right.
\end{eqnarray}
We may directly estimate
\begin{eqnarray}
\lnnm{\dfrac{\sin(2\psi)}{\rk_1-\e\eta}-\dfrac{\sin(2\psi)}{\rk_2-\e\eta}}\leq C,
\end{eqnarray}
which means the source term is in $L^{\infty}$ and decays exponentially. This equation does not involve $\bar\w'$ term, which makes it even simpler. Hence, we get
\begin{eqnarray}
\lnnm{\ue^{K_0\eta}\w'}\leq C\abs{\ln(\e)}^8,
\end{eqnarray}
for some constant $\w'_L$. Naturally we have $\w'_L$ must be zero.
\end{proof}
We finally come to the $\e$-Milne problem with diffusive boundary.
\begin{theorem}\label{Milne tangential.}
Assume (\ref{Milne bounded}), (\ref{Milne decay}), (\ref{Milne bounded.}) and (\ref{Milne decay.}). There exists $K_0>0$ such that the unique solution $f(\eta,\phi,\psi)$ to the
$\e$-Milne problem (\ref{Milne problem.}) with the normalization
condition (\ref{Milne normalization}) satisfies for $i=1,2$,
\begin{eqnarray}
\lnnm{\ue^{K_0\eta}\frac{\p(f-f_L)}{\p\tau_i}}&\leq& C\abs{\ln(\e)}^8,\\
\lnnm{\ue^{K_0\eta}\frac{\p(f-f_L)}{\p\psi}}&\leq& C\abs{\ln(\e)}^8
\end{eqnarray}
\end{theorem}

\newpage

\appendix

\makeatletter
\renewcommand \theequation {%
A.%
\ifnum\c@subsection>\z@\@arabic\c@subsection.%
\fi\@arabic\c@equation} \@addtoreset{equation}{section}
\@addtoreset{equation}{subsection} \makeatother

\section{Remainder Estimate}

In this section, we consider the remainder equation for $u(\vx,\vw)$ as
\begin{eqnarray}\label{neutron}
\left\{
\begin{array}{rcl}\displaystyle
\e\vw\cdot\nx u+u-\bar
u&=&f(\vx,\vw)\ \ \text{in}\ \ \Omega,\\
u(\vx_0,\vw)&=&\pp[u](\vx_0)+h(\vx_0,\vw)\ \ \text{for}\ \
\vw\cdot\vn<0\ \ \text{and}\ \ \vx_0\in\p\Omega,
\end{array}
\right.
\end{eqnarray}
where
\begin{eqnarray}
\bar u(\vx)=\frac{1}{4\pi}\int_{\s^2}u(\vx,\vw)\ud{\vw},
\end{eqnarray}
\begin{eqnarray}
\pp[u](\vx_0)=\frac{1}{4\pi}\int_{\vw\cdot\vn>0}u(\vx_0,\vw)(\vw\cdot\vn)\ud{\vw},
\end{eqnarray}
$\vn$ is the outward unit normal vector, with the Knudsen number $0<\e<<1$. To guarantee uniqueness, we need the normalization condition
\begin{eqnarray}\label{normalization.}
\int_{\Omega\times\s^2}u(\vx,\vw)\ud{\vw}\ud{\vx}=0.
\end{eqnarray}
Also, the data $f$ and $h$ satisfy the compatibility condition
\begin{eqnarray}\label{compatibility.}
\int_{\Omega\times\s^2}f(\vx,\vw)\ud{\vw}\ud{\vx}+\e\int_{\p\Omega}\int_{\vw\cdot\vn<0}h(\vx_0,\vw)(\vw\cdot\vn)\ud{\vw}\ud{\vx_0}=0.
\end{eqnarray}
Based on the flow direction, we can divide the boundary $\Gamma=\{(\vx,\vw): \vx\in\p\Omega\}$ into
the in-flow boundary $\Gamma^-$, the out-flow boundary $\Gamma^+$
and the grazing set $\Gamma^0$ as
\begin{eqnarray}
\Gamma^{-}&=&\{(\vx,\vw): \vx\in\p\Omega, \vw\cdot\vn<0\}\\
\Gamma^{+}&=&\{(\vx,\vw): \vx\in\p\Omega, \vw\cdot\vn>0\}\\
\Gamma^{0}&=&\{(\vx,\vw): \vx\in\p\Omega, \vw\cdot\vn=0\}
\end{eqnarray}
It is easy to see $\Gamma=\Gamma^+\cup\Gamma^-\cup\Gamma^0$.
Hence, the boundary condition is only given for $\Gamma^{-}$.
We define the $L^p$ norm with $1\leq p<\infty$ and $L^{\infty}$ norms in $\Omega\times\s^2$ as
usual:
\begin{eqnarray}
\nm{f}_{L^p(\Omega\times\s^2)}&=&\bigg(\int_{\Omega}\int_{\s^2}\abs{f(\vx,\vw)}^p\ud{\vw}\ud{\vx}\bigg)^{1/p},\\
\nm{f}_{L^{\infty}(\Omega\times\s^2)}&=&\sup_{(\vx,\vw)\in\Omega\times\s^2}\abs{f(\vx,\vw)}.
\end{eqnarray}
Define the $L^p$ norm with $1\leq p<\infty$ and $L^{\infty}$ norms on the boundary as follows:
\begin{eqnarray}
\nm{f}_{L^p(\Gamma)}&=&\bigg(\iint_{\Gamma}\abs{f(\vx,\vw)}^p\abs{\vw\cdot\vn}\ud{\vw}\ud{\vx}\bigg)^{1/p},\\
\nm{f}_{L^p(\Gamma^{\pm})}&=&\bigg(\iint_{\Gamma^{\pm}}\abs{f(\vx,\vw)}^p\abs{\vw\cdot\vn}\ud{\vw}\ud{\vx}\bigg)^{1/p},\\
\nm{f}_{L^{\infty}(\Gamma)}&=&\sup_{(\vx,\vw)\in\Gamma}\abs{f(\vx,\vw)},\\
\nm{f}_{L^{\infty}(\Gamma^{\pm})}&=&\sup_{(\vx,\vw)\in\Gamma^{\pm}}\abs{f(\vx,\vw)}.
\end{eqnarray}

The direct application of energy method as in \cite{AA003} and \cite{AA006}, we may obtain
\begin{lemma}\label{LT estimate}
Assume $f(\vx,\vw)\in L^{\infty}(\Omega\times\s^2)$ and $h(x_0,\vw)\in
L^{\infty}(\Gamma^-)$. Then for the transport equation (\ref{neutron}),
there exists a unique solution $u(\vx,\vw)\in L^2(\Omega\times\s^2)$
satisfying
\begin{eqnarray}
\frac{1}{\e}\nm{(1-\pp)[u]}^2_{L^2(\Gamma^+)}+\nm{u}_{L^2(\Omega\times\s^2)}\leq C\bigg(\frac{1}{\e^2}\nm{f}_{L^2(\Omega\times\s^2)}+\frac{1}{\e}\nm{h}_{L^2(\Gamma^-)}\bigg),
\end{eqnarray}
\end{lemma}
Based on classical $L^2-L^{\infty}$ framework, we are able to show
\begin{eqnarray}
\im{u}{\Omega\times\s^2}&\leq&C(\Omega)\bigg(\frac{1}{\e^{\frac{3}{2}}}\nm{u}_{L^2(\Omega\times\s^2)}+\im{f}{\Omega\times\s^2}+\im{h}{\Gamma^-}\bigg).
\end{eqnarray}
Therefore, it is natural to deduce $L^{\infty}$ estimate.
\begin{theorem}\label{LI estimate.}
Assume $f(\vx,\vw)\in L^{\infty}(\Omega\times\s^2)$ and
$h(x_0,\vw)\in L^{\infty}(\Gamma^-)$. Then the solution $u(\vx,\vw)$ to the transport
equation (\ref{neutron}) satisfies
\begin{eqnarray}
\im{u}{\Omega\times\s^2}\leq C(\Omega)\bigg(\frac{1}{\e^{\frac{7}{2}}}\tm{f}{\Omega\times\s^2}+\frac{1}{\e^{\frac{5}{2}}}\tm{h}{\Gamma^-}+\im{f}{\Omega\times\s^2}+\im{h}{\Gamma^-}\bigg).\no
\end{eqnarray}
\end{theorem}
However, the estimates here is not strong enough to close the diffusive limit, so we must further improve them.

\subsection{$L^{2m}$ Estimate}

In this subsection, we try to improve previous estimates. In the following, we assume $m$ is an integer and let $o(1)$ denote a sufficiently small constant.
\begin{theorem}\label{LN estimate}
Assume $f(\vx,\vw)\in L^{\infty}(\Omega\times\s^2)$ and
$h(x_0,\vw)\in L^{\infty}(\Gamma^-)$. Then for $\dfrac{3}{2}\leq m\leq 3$,
$u(\vx,\vw)$ satisfies
\begin{eqnarray}
&&\frac{1}{\e^{\frac{1}{2}}}\nm{(1-\pp)[u]}_{L^2(\Gamma^+)}+\nm{
\bar u}_{L^{2m}(\Omega\times\s^2)}+\frac{1}{\e}\nm{u-\bar
u}_{L^2(\Omega\times\s^2)}\\
&\leq&
C\bigg(o(1)\e^{\frac{3}{2m}}\nm{u}_{L^{\infty}(\Gamma^+)}+\frac{1}{\e}\tm{f}{\Omega\times\s^2}+
\frac{1}{\e^2}\nm{f}_{L^{\frac{2m}{2m-1}}(\Omega\times\s^2)}+\frac{1}{\e}\nm{h}_{L^2(\Gamma^-)}+\nm{h}_{L^{\frac{4m}{3}}(\Gamma^-)}\bigg).\nonumber
\end{eqnarray}
\end{theorem}
\begin{proof}
We divide the proof into several steps:\\
\ \\
Step 1: Kernel Estimate.\\
Applying Green's identity to the
equation (\ref{neutron}). Then for any
$\phi\in L^2(\Omega\times\s^2)$ satisfying $\vw\cdot\nx\phi\in
L^2(\Omega\times\s^2)$ and $\phi\in L^2(\Gamma)$, we have
\begin{eqnarray}\label{well-posedness temp 4.}
\e\int_{\Gamma}u\phi\ud{\gamma}
-\e\iint_{\Omega\times\s^2}(\vw\cdot\nx\phi)u+\iint_{\Omega\times\s^2}(u-\bar
u)\phi=\iint_{\Omega\times\s^2}f\phi.
\end{eqnarray}
Our goal is to choose a particular test function $\phi$. We first
construct an auxiliary function $\zeta$. Naturally $u\in
L^{\infty}(\Omega\times\s^2)$ implies $\bar u\in
L^{2m}(\Omega)$ which further leads to $(\bar u)^{2m-1}\in
L^{\frac{2m}{2m-1}}(\Omega)$. We define $\zeta(\vx)$ on $\Omega$ satisfying
\begin{eqnarray}\label{test temp 1.}
\left\{
\begin{array}{rcl}
\Delta \zeta&=&(\bar u)^{2m-1}-\displaystyle\frac{1}{\abs{\Omega}}\int_{\Omega}(\bar u)^{2m-1}\ud{\vx}\ \ \text{in}\ \
\Omega,\\\rule{0ex}{1.0em} \dfrac{\p\zeta}{\p\vn}&=&0\ \ \text{on}\ \ \p\Omega.
\end{array}
\right.
\end{eqnarray}
In the bounded domain $\Omega$, based on the standard elliptic
estimate, we have a unique $\zeta$ satisfying
\begin{eqnarray}\label{test temp 3.}
\nm{\zeta}_{W^{2,\frac{2m}{2m-1}}(\Omega)}\leq C\nm{(\bar
u)^{2m-1}}_{L^{\frac{2m}{2m-1}}(\Omega)}= C\nm{\bar
u}_{L^{2m}(\Omega)}^{2m-1},
\end{eqnarray}
and
\begin{eqnarray}
\int_{\Omega}\zeta(\vx)\ud{\vx}=0.
\end{eqnarray}
We plug the test function
\begin{eqnarray}\label{test temp 2.}
\phi=-\vw\cdot\nx\zeta
\end{eqnarray}
into the weak formulation (\ref{well-posedness temp 4.}) and estimate
each term there. By Sobolev embedding theorem, we have for $1\leq m\leq3$,
\begin{eqnarray}
\nm{\phi}_{L^2(\Omega)}&\leq& C\nm{\zeta}_{H^1(\Omega)}\leq C\nm{\zeta}_{W^{2,\frac{2m}{2m-1}}(\Omega)}\leq
C\nm{\bar
u}_{L^{2m}(\Omega)}^{2m-1},\label{test temp 6.}\\
\nm{\phi}_{L^{\frac{2m}{2m-1}}(\Omega)}&\leq&C\nm{\zeta}_{W^{1,\frac{2m}{2m-1}}(\Omega)}\leq
C\nm{\bar
u}_{L^{2m}(\Omega)}^{2m-1}.\label{test temp 4.}
\end{eqnarray}
Easily we can decompose
\begin{eqnarray}\label{test temp 5.}
-\e\iint_{\Omega\times\s^2}(\vw\cdot\nx\phi)u_{\l}&=&-\e\iint_{\Omega\times\s^2}(\vw\cdot\nx\phi)\bar
u_{\l}-\e\iint_{\Omega\times\s^2}(\vw\cdot\nx\phi)(u_{\l}-\bar
u_{\l}).
\end{eqnarray}
We estimate the two term on the right-hand side of (\ref{test temp 5.}) separately. By
(\ref{test temp 1.}) and (\ref{test temp 2.}), we have
\begin{eqnarray}\label{wellposed temp 1.}
-\e\iint_{\Omega\times\s^2}(\vw\cdot\nx\phi)\bar
u&=&\e\iint_{\Omega\times\s^2}\bar
u\bigg(w_1(w_1\p_{11}\zeta+w_2\p_{12}\zeta)+w_2(w_1\p_{12}\zeta+w_2\p_{22}\zeta)\bigg)\\
&=&\e\iint_{\Omega\times\s^2}\bar
u\bigg(w_1^2\p_{11}\zeta+w_2^2\p_{22}\zeta\bigg)\nonumber\\
&=&2\e\pi\int_{\Omega}\bar u(\p_{11}\zeta+\p_{22}\zeta)\nonumber\\
&=&\e\nm{\bar u}_{L^{2m}(\Omega)}^{2m}\nonumber.
\end{eqnarray}
In the second equality, above cross terms vanish due to the symmetry
of the integral over $\s^2$. On the other hand, for the second term
in (\ref{test temp 5.}), H\"older's inequality and the elliptic
estimate imply
\begin{eqnarray}\label{wellposed temp 2.}
-\e\iint_{\Omega\times\s^2}(\vw\cdot\nx\phi)(u-\bar
u)&\leq&C\e\nm{u-\bar u}_{L^{2m}(\Omega\times\s^2)}\nm{\nx\phi}_{L^{\frac{2m}{2m-1}}(\Omega)}\\
&\leq&C\e\nm{u-\bar u}_{L^{2m}(\Omega\times\s^2)}\nm{\zeta}_{W^{2,\frac{2m}{2m-1}}(\Omega)}\no\\
&\leq&C\e\nm{u-\bar
u}_{L^{2m}(\Omega\times\s^2)}\nm{\bar
u}_{L^{2m}(\Omega)}^{2m-1}\nonumber.
\end{eqnarray}
Based on (\ref{test temp 3.}), (\ref{test temp 6.}), (\ref{test temp 4.}), Sobolev embedding theorem and the trace theorem, we have
\begin{eqnarray}
\\
\nm{\nx\zeta}_{L^{\frac{4m}{4m-3}}(\Gamma)}\leq C\nm{\nx\zeta}_{W^{\frac{1}{2m},\frac{2m}{2m-1}}(\Gamma)}\leq C\nm{\nx\zeta}_{W^{1,\frac{2m}{2m-1}}(\Omega)}\leq C\nm{\zeta}_{W^{2,\frac{2m}{2m-1}}(\Omega)}\leq
C\nm{\bar
u}_{L^{2m}(\Omega)}^{2m-1}.\no
\end{eqnarray}
We may also decompose $\vw=(\vw\cdot\vn)\vn+\vw_{\perp}$ to obtain
\begin{eqnarray}
\e\int_{\Gamma}u\phi\ud{\gamma}&=&\e\int_{\Gamma}u(\vw\cdot\nx\zeta)\ud{\gamma}\\
&=&\e\int_{\Gamma}u(\vn\cdot\nx\zeta)(\vw\cdot\vn)\ud{\gamma}+\e\int_{\Gamma}u(\vw_{\perp}\cdot\nx\zeta)\ud{\gamma}\no\\
&=&\e\int_{\Gamma}u(\vw_{\perp}\cdot\nx\zeta)\ud{\gamma}.\no
\end{eqnarray}
Based on (\ref{test temp 3.}), (\ref{test temp 4.}) and
H\"older's inequality, we have
\begin{eqnarray}
\e\int_{\Gamma}u\phi\ud{\gamma}&=&\e\int_{\Gamma}u(\vw_{\perp}\cdot\nx\zeta)\ud{\gamma}\\
&=&\e\int_{\Gamma}\pp[u](\vw_{\perp}\cdot\nx\zeta)\ud{\gamma}+\e\int_{\Gamma^+}(1-\pp)[u](\vw_{\perp}\cdot\nx\zeta)\ud{\gamma}
+\e\int_{\Gamma^-}h(\vw_{\perp}\cdot\nx\zeta)\ud{\gamma}\no\\
&=&\e\int_{\Gamma^+}(1-\pp)[u](\vw_{\perp}\cdot\nx\zeta)\ud{\gamma}
+\e\int_{\Gamma^-}h(\vw_{\perp}\cdot\nx\zeta)\ud{\gamma}\no\\
&\leq&C\e\nm{\nx\zeta}_{L^{\frac{4m}{4m-3}}(\Gamma)}\bigg(\nm{(1-\pp)[u]}_{L^{\frac{4m}{3}}(\Gamma^+)}+\nm{h}_{L^{\frac{4m}{3}}(\Gamma^-)}\bigg)\no\\
&\leq&C\e\nm{\bar u}_{L^{2m}(\Omega\times\s^2)}^{2m-1}\bigg(\nm{(1-\pp)[u]}_{L^{\frac{4m}{3}}(\Gamma^+)}+\nm{h}_{L^{\frac{4m}{3}}(\Gamma^-)}\bigg).\no
\end{eqnarray}
Hence, we know
\begin{eqnarray}\label{wellposed temp 3.}
\e\int_{\Gamma}u\phi\ud{\gamma}&\leq&C\e\nm{\bar u}_{L^{2m}(\Omega\times\s^2)}^{2m-1}\bigg(\nm{(1-\pp)[u]}_{L^{\frac{4m}{3}}(\Gamma^+)}
+\nm{h}_{L^{\frac{4m}{3}}(\Gamma^-)}\bigg).
\end{eqnarray}
Also, we have
\begin{eqnarray}\label{wellposed temp 5.}
\iint_{\Omega\times\s^2}(u-\bar u)\phi\leq
C\nm{\phi}_{L^2(\Omega\times\s^2)}\nm{u-\bar
u}_{L^2(\Omega\times\s^2)}\leq
C\nm{\bar
u}_{L^{2m}(\Omega)}^{2m-1}\nm{u-\bar
u}_{L^2(\Omega\times\s^2)},
\end{eqnarray}
\begin{eqnarray}\label{wellposed temp 6.}
\iint_{\Omega\times\s^2}f\phi\leq C\nm{\phi}_{L^2(\Omega\times\s^2)}\nm{f}_{L^2(\Omega\times\s^2)}\leq C\nm{\bar
u}_{L^{2m}(\Omega)}^{2m-1}\nm{f}_{L^2(\Omega\times\s^2)}.
\end{eqnarray}
Collecting terms in (\ref{wellposed temp 1.}), (\ref{wellposed temp
2.}), (\ref{wellposed temp 3.}),
(\ref{wellposed temp 5.}) and (\ref{wellposed temp 6.}), we obtain
\begin{eqnarray}\label{improve temp 1.}
\e\nm{\bar u}_{L^{2m}(\Omega\times\s^2)}&\leq&
C\bigg(\e\nm{u-\bar
u}_{L^{2m}(\Omega\times\s^2)}+\nm{u-\bar
u}_{L^2(\Omega\times\s^2)}+\nm{f}_{L^2(\Omega\times\s^2)}\\
&&+\e\nm{(1-\pp)[u]}_{L^{\frac{4m}{3}}(\Gamma^+)}+\e\nm{h}_{L^{\frac{4m}{3}}(\Gamma^-)}\bigg)\nonumber,
\end{eqnarray}
\ \\
Step 2: Energy Estimate.\\
In the weak formulation (\ref{well-posedness temp 4.}), we may take
the test function $\phi=u$ to get the energy estimate
\begin{eqnarray}\label{well-posedness temp 9.}
\half\e\int_{\Gamma}\abs{u}^2\ud{\gamma}+\nm{u-\bar
u}_{L^2(\Omega\times\s^2)}^2=\iint_{\Omega\times\s^2}fu.
\end{eqnarray}
Hence, by $L^2$ estimate as in Theorem \ref{LT estimate} and \cite{AA007}, this naturally implies
\begin{eqnarray}\label{well-posedness temp 5.}
\e\nm{(1-\pp)[u]}_{L^2(\Gamma^+)}^2+\nm{u-\bar
u}_{L^2(\Omega\times\s^2)}^2\leq \nm{f}_{L^2(\Omega\times\s^2)}^2+\iint_{\Omega\times\s^2}fu+\nm{h}_{L^2(\Gamma^-)}^2.
\end{eqnarray}
On the other hand, we can square on both sides of
(\ref{improve temp 1.}) to obtain
\begin{eqnarray}\label{well-posedness temp 6.}
\e^2\nm{\bar u}_{L^{2m}(\Omega\times\s^2)}^2&\leq&
C\bigg(\e^2\nm{u-\bar
u}_{L^{2m}(\Omega\times\s^2)}^2+\nm{u-\bar
u}_{L^2(\Omega\times\s^2)}^2+\nm{f}_{L^2(\Omega\times\s^2)}^2\\
&&+\e^2\nm{(1-\pp)[u]}_{L^{\frac{4m}{3}}(\Gamma^+)}+\e^2\nm{h}_{L^{\frac{4m}{3}}(\Gamma^-)}^2\bigg)\nonumber,
\end{eqnarray}
Multiplying a sufficiently small constant on both sides of
(\ref{well-posedness temp 6.}) and adding it to (\ref{well-posedness
temp 5.}) to absorb
$\nm{u-\bar u}_{L^2(\Omega\times\s^2)}^2$ and $\e^2\nm{\bar u}_{L^2(\Omega\times\s^2)}^2$, we deduce
\begin{eqnarray}\label{wt 03.}
&&\e\nm{(1-\pp)[u]}_{L^2(\Gamma^+)}^2+\e^2\nm{\bar
u}_{L^{2m}(\Omega\times\s^2)}^2+\nm{u-\bar
u}_{L^2(\Omega\times\s^2)}^2\\
&\leq&
C\bigg(\e^2\nm{u-\bar
u}_{L^{2m}(\Omega\times\s^2)}^2+\e^2\nm{(1-\pp)[u]}_{L^{\frac{4m}{3}}(\Gamma^+)}\no\\
&&+\tm{f}{\Omega\times\s^2}^2+
\iint_{\Omega\times\s^2}fu+\nm{h}_{L^2(\Gamma^-)}^2+\e^2\nm{h}_{L^{\frac{4m}{3}}(\Gamma^-)}^2\bigg).\nonumber
\end{eqnarray}
By interpolation estimate and Young's inequality, for $\dfrac{3}{2}\leq m\leq 3$, we have
\begin{eqnarray}
\nm{(1-\pp)[u]}_{L^{\frac{4m}{3}}(\Gamma^+)}&\leq&\nm{(1-\pp)[u]}_{L^2(\Gamma^+)}^{\frac{3}{2m}}\nm{(1-\pp)[u]}_{L^{\infty}(\Gamma^+)}^{\frac{2m-3}{2m}}\\
&=&\bigg(\frac{1}{\e^{\frac{6m-9}{4m^2}}}\nm{(1-\pp)[u]}_{L^2(\Gamma^+)}^{\frac{3}{2m}}\bigg)
\bigg(\e^{\frac{6m-9}{4m^2}}\nm{(1-\pp)[u]}_{L^{\infty}(\Gamma^+)}^{\frac{2m-3}{2m}}\bigg)\no\\
&\leq&C\bigg(\frac{1}{\e^{\frac{6m-9}{4m^2}}}\nm{(1-\pp)[u]}_{L^2(\Gamma^+)}^{\frac{3}{2m}}\bigg)^{\frac{2m}{3}}+o(1)
\bigg(\e^{\frac{6m-9}{4m^2}}\nm{(1-\pp)[u]}_{L^{\infty}(\Gamma^+)}^{\frac{2m-3}{2m}}\bigg)^{\frac{2m}{2m-3}}\no\\
&\leq&\frac{C}{\e^{\frac{2m-3}{2m}}}\nm{(1-\pp)[u]}_{L^2(\Gamma^+)}+o(1)\e^{\frac{3}{2m}}\nm{(1-\pp)[u]}_{L^{\infty}(\Gamma^+)}\no\\
&\leq&\frac{C}{\e^{\frac{2m-3}{2m}}}\nm{(1-\pp)[u]}_{L^2(\Gamma^+)}+o(1)\e^{\frac{3}{2m}}\nm{u}_{L^{\infty}(\Omega\times\s^2)}.\no
\end{eqnarray}
Similarly, we have
\begin{eqnarray}
\nm{u-\bar u}_{L^{2m}(\Omega\times\s^2)}&\leq&\nm{u-\bar u}_{L^2(\Omega\times\s^2)}^{\frac{1}{m}}\nm{u-\bar u}_{L^{\infty}(\Omega\times\s^2)}^{\frac{m-1}{m}}\\
&=&\bigg(\frac{1}{\e^{\frac{3m-3}{2m^2}}}\nm{u-\bar u}_{L^2(\Omega\times\s^2)}^{\frac{1}{m}}\bigg)\bigg(\e^{\frac{3m-3}{2m^2}}\nm{u-\bar u}_{L^{\infty}(\Omega\times\s^2)}^{\frac{m-1}{m}}\bigg)\no\\
&\leq&C\bigg(\frac{1}{\e^{\frac{3m-3}{2m^2}}}\nm{u-\bar u}_{L^2(\Omega\times\s^2)}^{\frac{1}{m}}\bigg)^{m}+o(1)\bigg(\e^{\frac{3m-3}{2m^2}}\nm{u-\bar u}_{L^{\infty}(\Omega\times\s^2)}^{\frac{m-1}{m}}\bigg)^{\frac{m}{m-1}}\no\\
&\leq&\frac{C}{\e^{\frac{3m-3}{2m}}}\nm{u-\bar u}_{L^2(\Omega\times\s^2)}+o(1)\e^{\frac{3}{2m}}\nm{u-\bar u}_{L^{\infty}(\Omega\times\s^2)}.\no
\end{eqnarray}
We need this extra $\e^{\frac{3}{2m}}$ for the convenience of $L^{\infty}$ estimate.
Then we know for sufficiently small $\e$ and $\dfrac{3}{2}\leq m\leq 3$, ,
\begin{eqnarray}
\e^2\nm{(1-\pp)[u]}_{L^{m}(\Gamma^+)}^2
&\leq&C\e^{2-\frac{2m-3}{m}}\nm{(1-\pp)[u]}_{L^2(\Gamma^+)}^2+o(1)\e^{2+\frac{3}{m}}\nm{u}_{L^{\infty}(\Gamma^+)}^2\\
&\leq&\e^{\frac{3}{m}}\nm{(1-\pp)[u]}_{L^2(\Gamma^+)}^2+o(1)\e^{2+\frac{3}{m}}\nm{u}_{L^{\infty}(\Gamma^+)}^2,\\
&\leq&o(1)\e\nm{(1-\pp)[u]}_{L^2(\Gamma^+)}^2+o(1)\e^{2+\frac{3}{m}}\nm{u}_{L^{\infty}(\Gamma^+)}^2.\no
\end{eqnarray}
Similarly, we have
\begin{eqnarray}
\e^2\nm{u-\bar
u}_{L^{2m}(\Omega\times\s^2)}^2&\leq&\e^{2-\frac{3m-3}{m}}\nm{u-\bar u}_{L^2(\Omega\times\s^2)}^2+o(1)\e^{2+\frac{3}{m}}\nm{u}_{L^{\infty}(\Omega\times\s^2)}^2\\
&\leq& \e^{\frac{3}{m}-1}\nm{u-\bar u}_{L^2(\Omega\times\s^2)}^2+o(1)\e^{2+\frac{3}{m}}\nm{u}_{L^{\infty}(\Omega\times\s^2)}^2,\\
&\leq& o(1)\nm{u-\bar u}_{L^2(\Omega\times\s^2)}^2+o(1)\e^{2+\frac{3}{m}}\nm{u}_{L^{\infty}(\Omega\times\s^2)}^2.\no
\end{eqnarray}
By (\ref{well-posedness temp 5.}), we can absorb $\nm{u-\bar u}_{L^2(\Omega\times\s^2)}$ and $\e\nm{(1-\pp)[u]}_{L^2(\Gamma^+)}^2$ into left-hand side to obtain
\begin{eqnarray}\label{wt 04.}
&&\e\nm{(1-\pp)[u]}_{L^2(\Gamma^+)}^2+\e^2\nm{\bar
u}_{L^{2m}(\Omega\times\s^2)}^2+\nm{u-\bar
u}_{L^2(\Omega\times\s^2)}^2\\
&\leq&
C\bigg(o(1)\e^{2+\frac{3}{m}}\nm{u}_{L^{\infty}(\Omega\times\s^2)}^2+\tm{f}{\Omega\times\s^2}^2+
\iint_{\Omega\times\s^2}fu+\nm{h}_{L^2(\Gamma^-)}^2+\e^2\nm{h}_{L^{\frac{4m}{3}}(\Gamma^-)}^2\bigg).\no
\end{eqnarray}
We can decompose
\begin{eqnarray}
\iint_{\Omega\times\s^2}fu=\iint_{\Omega\times\s^2}f\bar u+\iint_{\Omega\times\s^2}f(u-\bar u).
\end{eqnarray}
H\"older's inequality and Cauchy's inequality imply
\begin{eqnarray}
\iint_{\Omega\times\s^2}f\bar u\leq\nm{f}_{L^{\frac{2m}{2m-1}}(\Omega\times\s^2)}\nm{\bar u}_{L^{2m}(\Omega\times\s^2)}
\leq\frac{C}{\e^{2}}\nm{f}_{L^{\frac{2m}{2m-1}}(\Omega\times\s^2)}^2+o(1)\e^2\nm{\bar u}_{L^{2m}(\Omega\times\s^2)}^2,
\end{eqnarray}
and
\begin{eqnarray}
\iint_{\Omega\times\s^2}f(u-\bar u)\leq C\nm{f}_{L^{2}(\Omega\times\s^2)}^2+o(1)\nm{u-\bar u}_{L^2(\Omega\times\s^2)}^2.
\end{eqnarray}
Hence, absorbing $\e^2\nm{\bar u}_{L^{2m}(\Omega\times\s^2)}^2$ and $\nm{u-\bar u}_{L^2(\Omega\times\s^2)}^2$ into left-hand side of (\ref{wt 04.}), we get
\begin{eqnarray}\label{wt 06.}
&&\e\nm{(1-\pp)[u]}_{L^2(\Gamma^+)}^2+\e^2\nm{\bar
u}_{L^{2m}(\Omega\times\s^2)}^2+\nm{u-\bar
u}_{L^2(\Omega\times\s^2)}^2\\
&\leq&
C\bigg(o(1)\e^{2+\frac{3}{m}}\nm{u}_{L^{\infty}(\Gamma^+)}^2+\tm{f}{\Omega\times\s^2}^2+
\frac{1}{\e^2}\nm{f}_{L^{\frac{2m}{2m-1}}(\Omega\times\s^2)}^2+\nm{h}_{L^2(\Gamma^-)}^2+\e^2\nm{h}_{L^{\frac{4m}{3}}(\Gamma^-)}^2\bigg),\nonumber
\end{eqnarray}
which implies
\begin{eqnarray}\label{wt 07.}
&&\frac{1}{\e^{\frac{1}{2}}}\nm{(1-\pp)[u]}_{L^2(\Gamma^+)}+\nm{
\bar u}_{L^{2m}(\Omega\times\s^2)}+\frac{1}{\e}\nm{u-\bar
u}_{L^2(\Omega\times\s^2)}\\
&\leq&
C\bigg(o(1)\e^{\frac{3}{2m}}\nm{u}_{L^{\infty}(\Gamma^+)}+\frac{1}{\e}\tm{f}{\Omega\times\s^2}+
\frac{1}{\e^2}\nm{f}_{L^{\frac{2m}{2m-1}}(\Omega\times\s^2)}+\frac{1}{\e}\nm{h}_{L^2(\Gamma^-)}+\nm{h}_{L^{\frac{4m}{3}}(\Gamma^-)}\bigg).\nonumber
\end{eqnarray}

\end{proof}

\subsection{$L^{\infty}$ Estimate}

In this subsection, we prove the $L^{\infty}$ estimate. We consider the characteristics that reflect several times on the boundary.
\begin{definition}(Stochastic Cycle)
For fixed point $(t,\vx,\vw)$ with $(\vx,\vw)\notin\Gamma^0$, let
$(t_0,\vx_0,\vw_0)=(0,\vx,\vw)$. For $\vw_{k+1}$ such that
$\vw_{k+1}\cdot\vn(\vx_{k+1})>0$, define the $(k+1)$-component of
the back-time cycle as
\begin{eqnarray}
(t_{k+1},\vx_{k+1},\vw_{k+1})=(t_k+t_b(\vx_k,\vw_k),\vx_b(\vx_k,\vw_k),\vw_{k+1})
\end{eqnarray}
where
\begin{eqnarray}
t_b(\vx,\vw)&=&\inf\{t>0:\vx-\e t\vw\notin\Omega\}\\
x_b(\vx,\vw)&=&\vx-\e t_b(\vx,\vw)\vw\notin\Omega
\end{eqnarray}
Set
\begin{eqnarray}
\xc(s;t,\vx,\vw)&=&\sum_{k}\id_{t_{k+1}\leq s<t_k}\bigg(\vx_k-\e(t_k-s)\vw_k\bigg)\\
\wc(s;t,\vx,\vw)&=&\sum_{k}\id_{t_{k+1}\leq s<t_k}\vw_k
\end{eqnarray}
Define $\mu_{k+1}=\{\vw\in \s^2:\vw\cdot\vn(\vx_{k+1})>0\}$, and
let the iterated integral for $k\geq2$ be defined as
\begin{eqnarray}
\int_{\prod_{k=1}^{k-1}\mu_j}\prod_{j=1}^{k-1}\ud{\sigma_j}=\int_{\mu_1}\ldots\bigg(\int_{\mu_{k-1}}\ud{\sigma_{k-1}}\bigg)\ldots\ud{\sigma_1}
\end{eqnarray}
where $\ud{\sigma_j}=(\vn(\vx_j)\cdot\vw)\ud{\vw}$ is a
probability measure.
\end{definition}
\begin{lemma}\label{well-posedness lemma 6}
For $T_0>0$ sufficiently large, there exists constants $C_1,C_2>0$
independent of $T_0$, such that for $k=C_1T_0^{5/4}$,
\begin{eqnarray}
\int_{\prod_{j=1}^{k-1}\mu_j}{\bf{1}}_{t_k(t,\vx,\vw,\vw_1,\ldots,\vw_{k-1})<T_0}\prod_{j=1}^{k-1}\ud{\sigma_j}\leq
\bigg(\frac{1}{2}\bigg)^{C_2T_0^{5/4}}
\end{eqnarray}
\end{lemma}
\begin{proof}
See \cite[Lemma 4.1]{Esposito.Guo.Kim.Marra2013}.
\end{proof}
\begin{theorem}\label{LI estimate}
Assume $f(\vx,\vw)\in L^{\infty}(\Omega\times\s^2)$ and
$h(x_0,\vw)\in L^{\infty}(\Gamma^-)$. Then for the steady neutron
transport equation (\ref{neutron}), there exists a unique solution
$u(\vx,\vw)\in L^{\infty}(\Omega\times\s^2)$ satisfying
\begin{eqnarray}
\im{u}{\Omega\times\s^2}
&\leq& C\bigg(\frac{1}{\e^{1+\frac{3}{2m}}}\tm{f}{\Omega\times\s^2}+
\frac{1}{\e^{2+\frac{3}{2m}}}\nm{f}_{L^{\frac{2m}{2m-1}}(\Omega\times\s^2)}+\im{f}{\Omega\times\s^2}\\
&&+\frac{1}{\e^{1+\frac{3}{2m}}}\nm{h}_{L^2(\Gamma^-)}+\frac{1}{\e^{\frac{3}{2m}}}\nm{h}_{L^{\frac{4m}{3}}(\Gamma^-)}+\im{h}{\Gamma^-}\bigg).\no
\end{eqnarray}
\end{theorem}
\begin{proof}
We divide the proof into several steps:\\
\ \\
Step 1: Mild formulation.\\
We rewrite the equation (\ref{neutron}) along the characteristics as
\begin{eqnarray}
u(\vx,\vw)&=&h(\vx-\e t_1\vw,\vw)\ue^{- t_1}+\pp[u](\vx-\e t_1\vw,\vw)\ue^{- t_1}\\
&&+\int_0^{t_1}f(\vx-\e(t_1-s_1)\vw,\vw)\ue^{- (t_1-s_1)}\ud{s_1}+\int_0^{t_1}\bar u(\vx-\e(t_1-s_1)\vw)\ue^{- (t_1-s_1)}\ud{s_1}.\no
\end{eqnarray}
Note that here $\pp[u]$ is an integral over $\mu_1$ at $\vx_1$, using stochastic cycle, we may rewrite it again along the characteristics to $\vx_2$. This process can continue to arbitrary $\vx_k$. Then we get
\begin{eqnarray}\label{wt 21}
u(\vx,\vw)&=&\ue^{- t_1}H+\sum_{l=1}^{k-1}\bigg(\int_{\prod_{j=1}^l}\ue^{- t_{l+1}}G\prod_{j=1}^l\ud{\sigma_j}\bigg)
+\sum_{l=1}^{k-1}\bigg(\int_{\prod_{j=1}^l}\ue^{- t_{l+1}}\pp[u](\vx_k,\vw_{k-1})\prod_{j=1}^l\ud{\sigma_j}\bigg)\\
&=&I+II+III.\no
\end{eqnarray}
where
\begin{eqnarray}
H&=&h(\vx-\e t_1\vw,\vw)\\
&&+\int_0^{t_1}f(\vx-\e(t_1-s_1)\vw,\vw)\ue^{ s_1}\ud{s_1}+\int_0^{t_1}\bar u(\vx-\e(t_1-s_1)\vw)\ue^{ s_1}\ud{s_1},\no\\
G&=&h(\vx_l-\e t_{l+1}\vw_l,\vw_l)\\
&&+\int_0^{t_l}f(\vx_l-\e(t_{l+1}-s_{l+1})\vw_l,\vw_l)\ue^{ s_{l+1}}\ud{s_{l+1}}+\int_0^{t_l}\bar u(\vx_l-\e(t_{l+1}-s_{l+1})\vw_l)\ue^{ s_{l+1}}\ud{s_{l+1}}.\no
\end{eqnarray}
We need to estimate each term on the right-hand side of (\ref{wt 21}).\\
\ \\
Step 2: Estimate of mild formulation.\\
We first consider $III$. We may decompose it as
\begin{eqnarray}
III&=&\sum_{l=1}^{k-1}\int_{\prod_{j=1}^l}\pp[u](\vx_k,\vw_{k-1})\ue^{- t_{l+1}}\prod_{j=1}^l\ud{\sigma_j}\\
&=&\sum_{l=1}^{k-1}\int_{\prod_{j=1}^l}{\bf{1}}_{t_k\leq T_0}\pp[u](\vx_k,\vw_{k-1})\ue^{- t_{l+1}}\prod_{j=1}^l\ud{\sigma_j}\no\\
&&+\sum_{l=1}^{k-1}\int_{\prod_{j=1}^l}{\bf{1}}_{t_k\geq T_0}\pp[u](\vx_k,\vw_{k-1})\ue^{- t_{l+1}}\prod_{j=1}^l\ud{\sigma_j},\no\\
&=&III_1+III_2,\no
\end{eqnarray}
where $T_0>0$ is defined as in Lemma \ref{well-posedness lemma 6}. Then we take $k=C_1T_0^{5/4}$. By Lemma \ref{well-posedness lemma 6}, we deduce
\begin{eqnarray}
\abs{III_1}&\leq&C\bigg(\frac{1}{2}\bigg)^{C_2T_0^{5/4}}\im{u}{\Omega\times\s^2}.
\end{eqnarray}
Also, we may directly estimate
\begin{eqnarray}
\abs{III_2}&\leq&C\ue^{- T_0}\im{u}{\Omega\times\s^2}.
\end{eqnarray}
Then taking $T_0$ sufficiently large, we know
\begin{eqnarray}\label{wt 22}
\abs{III}\leq \delta\im{u}{\Omega\times\s^2},
\end{eqnarray}
for $\d>0$ small.
On the other hand, we may directly estimate the terms in $I$ and $II$ related to $h$ and $f$, which we denote as $I_1$ and $II_1$. For fixed $T$, it is easy to see
\begin{eqnarray}\label{wt 23}
\abs{I_1}+\abs{II_1}\leq \im{f}{\Omega\times\s^2}+\im{h}{\Gamma^-}.
\end{eqnarray}
\ \\
Step 3: Estimate of $\bar u$ term.\\
The most troubling terms are related to $\bar u$. Here, we use the trick as in \cite{Esposito.Guo.Kim.Marra2013} and \cite{AA003}. Collecting the results in (\ref{wt 22}) and (\ref{wt 23}), we obtain
\begin{eqnarray}
\abs{u}&\leq&A+\abs{\int_0^{t_1}\bar u(\vx-\e(t_1-s_1)\vw)\ue^{- (t_1-s_1)}\ud{s_1}}\\
&&+\abs{\sum_{l=1}^{k-1}\Bigg(\int_{\prod_{j=1}^l}\bigg(\int_0^{t_l}\bar u(\vx_l-\e(t_{l+1}-s_{l+1})\vw_l)\ue^{- (t_{l+1}-s_{l+1})}\ud{s_{l+1}}\bigg)\prod_{j=1}^l\ud{\sigma_j}\Bigg)},\no\\
&=&A+I_2+II_2,\no
\end{eqnarray}
where
\begin{eqnarray}
A=\im{f}{\Omega\times\s^2}+\im{h}{\Gamma^-}+\delta\im{u}{\Omega\times\s^2}.
\end{eqnarray}
By definition, we know
\begin{eqnarray}
\abs{I_2}&=&\abs{\int_0^{t_1}\bigg(\int_{\s^2}u(\vx-\e(t_1-s_1)\vw,\vw_{s_1})\ud{\vw_{s_1}}\bigg)\ue^{- (t_1-s_1)}\ud{s_1}},
\end{eqnarray}
where $\vw_{s_1}\in\s^2$ is a dummy variable.
Then we can utilize the mild formulation (\ref{wt 21}) to rewrite $u(\vx-\e(t_1-s_1)\vw,\vw_{s_1})$ along the characteristics. We denote the stochastic cycle as $(t_k',\vx_k',\vw_k')$ correspondingly and $(t_0',\vx_0',\vw_0')=(0,\vx-\e(t_1-s_1)\vw,\vw_{s_1})$. Then
\begin{eqnarray}
\abs{I_2}&\leq& \abs{\int_0^{t_1}\bigg(\int_{\s^2}A\ud{\vw_{s_1}}\bigg)\ue^{- (t_1-s_1)}\ud{s_1}}\\
&&+\abs{\int_0^{t_1}\bigg(\int_{\s^2}\int_0^{t_1'}\bar u(\vx'-\e(t_1'-s_1')\vw_{s_1})\ue^{- (t_1'-s_1')}\ud{s_1'}\ud{\vw_{s_1}}\bigg)\ue^{- (t_1-s_1)}\ud{s_1}}\no\\
&&+\Bigg\vert\int_0^{t_1}\bigg(\int_{\s^2}\sum_{l'=1}^{k-1}\int_{\prod_{j'=1}^{l'}}\bigg(\int_0^{t_{l'}'}\bar u(\vx_{l'}-\e(t_{l'+1}'-s_{l'+1}')\vw_{l'})\ue^{- (t_{l'+1}'-s_{l'+1}')}\ud{s_{l'+1}'}\bigg)\prod_{j'=1}^{l'}
\ud{\sigma_{j'}}\ud{\vw_{s_1}}\bigg)\no\\
&&\ue^{- (t_1-s_1)}\ud{s_1}\Bigg\vert,\no\\
&=&\abs{I_{2,1}}+\abs{I_{2,2}}+\abs{I_{2,3}}.\no
\end{eqnarray}
It is obvious that
\begin{eqnarray}
\abs{I_{2,1}}&=&\abs{\int_0^{t_1}\bigg(\int_{\s^2}A\ud{\vw_{s_1}}\bigg)\ue^{- (t_1-s_1)}\ud{s_1}}\leq A\\
&\leq& \im{f}{\Omega\times\s^2}+\im{h}{\Gamma^-}+\delta\im{u}{\Omega\times\s^2}.\no
\end{eqnarray}
Then by definition, we know
\begin{eqnarray}
\abs{I_{2,2}}&=&\abs{\int_0^{t_1}\bigg(\int_{\s^2}\int_0^{t_1'}\bar u(\vx'-\e(t_1'-s_1')\vw_{s_1})\ue^{- (t_1'-s_1')}\ud{s_1'}\ud{\vw_{s_1}}\bigg)\ue^{- (t_1-s_1)}\ud{s_1}}.
\end{eqnarray}
We may decompose this integral
\begin{eqnarray}
\int_0^{t_1}\int_{\s^2}\int_0^{t_1'}&=&\int_0^{t_1}\int_{\s^2}\int_{t_1'-s_1'\leq\d}
+\int_0^{t_1}\int_{\s^2}\int_{t_1'-s_1'\geq\d}=I_{2,2,1}+I_{2,2,2}.
\end{eqnarray}
For $I_{2,2,1}$, since the integral is defined in the small domain $[t_1'-\d, t_1']$, it is easy to see
\begin{eqnarray}
\abs{I_{2,2,1}}\leq \delta\im{u}{\Omega\times\s^2}.
\end{eqnarray}
For $I_{2,2,2}$, applying H\"{o}lder's inequality, we get
\begin{eqnarray}
\abs{I_{2,2,2}}&\leq&\abs{\int_0^{t_1}\int_{\s^2}\int_{t_1'-s_1'\geq\d}\bar u(\vx'-\e(t_1'-s_1')\vw_{s_1})\ue^{- (t_1'-s_1')}\ue^{- (t_1-s_1)}\ud{s_1'}\ud{\vw_{s_1}}\ud{s_1}}\\
&\leq&\bigg(\int_0^{t_1}\int_{\s^2}\int_{t_1'-s_1'\geq\d}\ue^{- (t_1'-s_1')}\ue^{- (t_1-s_1)}\ud{s_1'}\ud{\vw_{s_1}}\ud{s_1}\bigg)^{\frac{2m-1}{2m}}\no\\
&&\bigg(\int_0^{t_1}\int_{\s^2}\int_{t_1'-s_1'\geq\d}{\bf{1}}_{\vx'-\e(t_1'-s_1')\vw_{s_1}\in\Omega}\abs{\bar u}^{2m}(\vx'-\e(t_1'-s_1')\vw_{s_1})\ue^{- (t_1'-s_1')}\ue^{- (t_1-s_1)}\ud{s_1'}\ud{\vw_{s_1}}\ud{s_1}\bigg)^{\frac{1}{2m}}\no\\
&\leq&\bigg(\int_0^{t_1}\int_{\s^2}\int_{t_1'-s_1'\geq\d}{\bf{1}}_{\vx'-\e(t_1'-s_1')\vw_{s_1}\in\Omega}\abs{\bar u}^{2m}(\vx'-\e(t_1'-s_1')\vw_{s_1})\ue^{- (t_1'-s_1')}\ue^{- (t_1-s_1)}\ud{s_1'}\ud{\vw_{s_1}}\ud{s_1}\bigg)^{\frac{1}{2m}}.\no
\end{eqnarray}
Since $\vw_{s_1}\in\s^2$, we can express it as $(\sin\phi\cos\psi,\sin\phi\sin\psi,\cos\phi)$. Then considering $\vx'-\e(t_1'-s_1')\vw_{s_1}\in\bar\Omega$, we apply the substitution $(\phi,\psi,s'_1)\rt(y_1,y_2,y_3)$ as
\begin{eqnarray}
\vec y=\vx'-\e(t_1'-s_1')\vw_{s_1},
\end{eqnarray}
whose Jacobian is
\begin{eqnarray}
\abs{\frac{\p(y_1,y_2,y_3)}{\p(\phi,\psi,r')}}&=&\abs{\abs{\begin{array}{ccc}
-\e(t_1'-s_1')\cos\phi\cos\psi&\e(t_1'-s_1')\sin\phi\sin\psi&\e\sin\phi\cos\psi\\
-\e(t_1'-s_1')\cos\phi\sin\psi&-\e(t_1'-s_1')\sin\phi\cos\psi&\e\sin\phi\sin\psi\\
\e(t_1'-s_1')\sin\phi&0&\e\cos\phi
\end{array}}}\\
&=&\e^3(t_1'-s_1')^2\sin\phi.
\end{eqnarray}
Hence, we can further decompose $I_{2,2,2}$ into $\abs{\sin\phi}\leq\d$ and $\abs{\sin\phi}\geq\d$. In the first part, the integral can be bounded by $\delta\im{u}{\Omega\times\s^2}$. Then for the second part, we have
\begin{eqnarray}
\abs{\frac{\p(y_1,y_2,y_3)}{\p(\phi,\psi,r')}}\geq\e^3\d^3.
\end{eqnarray}
Therefore, we know
\begin{eqnarray}
\abs{I_{2,2,2}}&\leq&\frac{1}{\e^{\frac{3}{2m}}\d^{\frac{3}{2m}}}\nm{\bar u}_{L^{2m}(\Omega\times\s^2)}.
\end{eqnarray}
Hence, we have shown
\begin{eqnarray}
\abs{I_{2,2}}\leq \delta\im{u}{\Omega\times\s^2}+\frac{1}{\e^{\frac{3}{2m}}\d^{\frac{3}{2m}}}\nm{\bar u}_{L^{2m}(\Omega\times\s^2)}.
\end{eqnarray}
After a similar but tedious computation, we can show
\begin{eqnarray}
\abs{I_{2,3}}\leq \delta\im{u}{\Omega\times\s^2}+\frac{1}{\e^{\frac{3}{2m}}\d^{\frac{3}{2m}}}\nm{\bar u}_{L^{2m}(\Omega\times\s^2)}.
\end{eqnarray}
Hence, we have proved
\begin{eqnarray}
\abs{I_{2}}\leq \delta\im{u}{\Omega\times\s^2}+\frac{1}{\e^{\frac{3}{2m}}\d^{\frac{3}{2m}}}\nm{\bar u}_{L^{2m}(\Omega\times\s^2)}+\im{f}{\Omega\times\s^2}+\im{h}{\Gamma^-}.
\end{eqnarray}
In a similar fashion, we can show
\begin{eqnarray}
\abs{II_{2}}\leq \delta\im{u}{\Omega\times\s^2}+\frac{1}{\e^{\frac{3}{2m}}\d^{\frac{3}{2m}}}\nm{\bar u}_{L^{2m}(\Omega\times\s^2)}+\im{f}{\Omega\times\s^2}+\im{h}{\Gamma^-}.
\end{eqnarray}
\ \\
Step 4: Synthesis.\\
Summarizing all above, we have shown
\begin{eqnarray}
\abs{u}&\leq& \delta\im{u}{\Omega\times\s^2}+\frac{1}{\d^{\frac{3}{2m}}\e^{\frac{3}{2m}}}\nm{\bar u}_{L^2(\Omega\times\s^2)}+\im{f}{\Omega\times\s^2}+\im{h}{\Gamma^-}\\
&\leq&\delta\im{u}{\Omega\times\s^2}+\frac{1}{\d^{\frac{3}{2m}}\e^{\frac{3}{2m}}}\nm{u}_{L^2(\Omega\times\s^2)}+\im{f}{\Omega\times\s^2}+\im{h}{\Gamma^-}.\no
\end{eqnarray}
Since $(\vx,\vw)$ are arbitrary and $\d$ is small, taking supremum on both sides and applying Lemma \ref{LT estimate}, we have
\begin{eqnarray}
\im{u}{\Omega\times\s^2}\leq C\bigg(
\frac{1}{\e^{\frac{3}{2m}}}\nm{\bar u}_{L^{2m}(\Omega\times\s^2)}+\im{f}{\Omega\times\s^2}+\im{g}{\Gamma^-}\bigg).
\end{eqnarray}
Considering Theorem \ref{LN estimate}, we obtain
\begin{eqnarray}
\im{u}{\Omega\times\s^2}&\leq& C\bigg(\frac{1}{\e^{1+\frac{3}{2m}}}\tm{f}{\Omega\times\s^2}+
\frac{1}{\e^{2+\frac{3}{2m}}}\nm{f}_{L^{\frac{2m}{2m-1}}(\Omega\times\s^2)}+\im{f}{\Omega\times\s^2}\\
&&+\frac{1}{\e^{1+\frac{3}{2m}}}\nm{h}_{L^2(\Gamma^-)}+\frac{1}{\e^{\frac{3}{2m}}}\nm{h}_{L^{\frac{4m}{3}}(\Gamma^-)}+\im{h}{\Gamma^-}\bigg)+o(1)\im{u}{\Gamma^+}.\no
\end{eqnarray}
Absorbing $\im{u}{\Omega\times\s^2}$ into the left-hand side, we obtain
\begin{eqnarray}
\im{u}{\Omega\times\s^2}
&\leq& C\bigg(\frac{1}{\e^{1+\frac{3}{2m}}}\tm{f}{\Omega\times\s^2}+
\frac{1}{\e^{2+\frac{3}{2m}}}\nm{f}_{L^{\frac{2m}{2m-1}}(\Omega\times\s^2)}+\im{f}{\Omega\times\s^2}\\
&&+\frac{1}{\e^{1+\frac{3}{2m}}}\nm{h}_{L^2(\Gamma^-)}+\frac{1}{\e^{\frac{3}{2m}}}\nm{h}_{L^{\frac{4m}{3}}(\Gamma^-)}+\im{h}{\Gamma^-}\bigg).\no
\end{eqnarray}
This is the desired estimate.
\end{proof}

\newpage

\makeatletter
\renewcommand \theequation {%
B.%
\ifnum\c@subsection>\z@\@arabic\c@subsection.%
\fi\@arabic\c@equation} \@addtoreset{equation}{section}
\@addtoreset{equation}{subsection} \makeatother

\section{Diffusive Limit}

\begin{corollary}\label{diffusive limit}
Assume $g(\vx_0,\vw)\in C^2(\Gamma^-)$ satisfying (\ref{compatibility}). Then for the steady neutron
transport equation (\ref{transport}), there exists a unique solution
$u^{\e}(\vx,\vw)\in L^{\infty}(\Omega\times\s^2)$ satisfying (\ref{normalization}). Moreover, for any $0<\d<<1$, the solution obeys the estimate
\begin{eqnarray}
\im{u^{\e}-\u_0}{\Omega\times\s^2}\leq C(\d,\Omega)\e^{\frac{1}{3}-\d},
\end{eqnarray}
where $\u_0$ is defined in (\ref{expansion temp 11}).
\end{corollary}
\begin{proof}
Based on Theorem \ref{LI estimate}, we know there exists a unique $u^{\e}(\vx,\vw)\in L^{\infty}(\Omega\times\s^2)$, so we focus on the diffusive limit. We can divide the proof into several steps:\\
\ \\
Step 1: Remainder definitions.\\
We define the remainder as
\begin{eqnarray}\label{pf 1_}
R&=&u^{\e}-\sum_{k=0}^{2}\e^k\u_k-\sum_{k=0}^{1}\e^k\ub_k=u^{\e}-\q-\qb,
\end{eqnarray}
where
\begin{eqnarray}
\q&=&\u_0+\e\u_1+\e^2\u_2,\\
\qb&=&\ub_0+\e\ub_1.
\end{eqnarray}
Noting the equation (\ref{transport temp}) is equivalent to the
equation (\ref{transport}), we write $\ll$ to denote the neutron
transport operator as follows:
\begin{eqnarray}
\ll[u]&=&\e\vw\cdot\nx u+ u-\bar u\\
&=&\sin\phi\frac{\p u}{\p\eta}-\e\bigg(\dfrac{\sin^2\psi}{R_1-\e\eta}+\dfrac{\cos^2\psi}{R_2-\e\eta}\bigg)\cos\phi\dfrac{\p u}{\p\phi}
+u-\bar u+G[u],\nonumber
\end{eqnarray}
where
\begin{eqnarray}
G[u]&=&\e\bigg(\dfrac{\cos\phi\sin\psi}{P_1(1-\e\kappa_1\eta)}\dfrac{\p u}{\p\tau_1}+\dfrac{\cos\phi\cos\psi}{P_2(1-\e\kappa_2\eta)}\dfrac{\p u}{\p\tau_2}\bigg)\\\rule{0ex}{2.0em}
&&+\e\bigg(\dfrac{\sin\psi}{1-\e\kappa_1\eta}(\vt_2\times(\p_{21}\vr\times\vt_2))\cdot\vt_2
+\dfrac{\cos\psi}{1-\e\kappa_2\eta}(\vt_1\times(\p_{12}\vr\times\vt_1))\cdot\vt_1\bigg)\dfrac{\cos\phi}{P_1P_2}\dfrac{\p u}{\p\psi}.\no
\end{eqnarray}
\ \\
Step 2: Estimates of $\ll[\q]$.\\
The interior contribution can be estimated as
\begin{eqnarray}
\ll[\q]=\e\vw\cdot\nx \q+ \q-\bar
\q&=&\e^{3}\vw\cdot\nx \u_2.
\end{eqnarray}
Based on classical elliptic estimates, we have
\begin{eqnarray}
\im{\ll[\q]}{\Omega\times\s^2}&\leq&\im{\e^{3}\vw\cdot\nx \u_2}{\Omega\times\s^2}\leq C\e^{3}\im{\nx\u_2}{\Omega\times\s^2}\leq
C\e^{3}.
\end{eqnarray}
This implies
\begin{eqnarray}\label{pf 2_}
\tm{\ll[\q]}{\Omega\times\s^2}&\leq& C\e^{3},\\
\nm{\ll[\q]}_{L^{\frac{2m}{2m-1}}(\Omega\times\s^2)}&\leq& C\e^{3},\\
\im{\ll[\q]}{\Omega\times\s^2}&\leq& C\e^{3}.
\end{eqnarray}
\ \\
Step 3: Estimates of $\ll \qb$.\\
Since $\ub_0=0$, we only need to estimate $\ub_1=(f_1^{\e}-f^{\e}_{1,L})\cdot\psi_0=\v\psi_0$ where
$f_1^{\e}(\eta,\tau_1,\tau_2,\phi,\psi)$ solves the $\e$-Milne problem and $\v=f_1^{\e}-f^{\e}_{1,L}$. The boundary layer contribution can be
estimated as
\begin{eqnarray}\label{remainder temp 1}
\ll[\e\ub_1]&=&\sin\phi\frac{\p
(\e\ub_1)}{\p\eta}-\e\bigg(\dfrac{\sin^2\psi}{R_1-\e\eta}+\dfrac{\cos^2\psi}{R_2-\e\eta}\bigg)\cos\phi\frac{\p
(\e\ub_1)}{\p\phi}+ (\e\ub_1)-
(\e\bar\ub_1)+G[\e\ub_1]\\
&=&\e\Bigg(\sin\phi\bigg(\psi_0\frac{\p
\v}{\p\eta}+\v\frac{\p\psi_0}{\p\eta}\bigg)-\e\psi_0\bigg(\dfrac{\sin^2\psi}{R_1-\e\eta}+\dfrac{\cos^2\psi}{R_2-\e\eta}\bigg)\cos\phi\frac{\p
\v}{\p\phi}+\frac{\p \v}{\p\tau}+ \psi_0\v-\psi_0\bar\v+\psi_0G[\v]\Bigg)\nonumber\\
&=&\e\psi_0\bigg(\sin\phi\frac{\p
\v}{\p\eta}-\e\bigg(\dfrac{\sin^2\psi}{R_1-\e\eta}+\dfrac{\cos^2\psi}{R_2-\e\eta}\bigg)\cos\phi\frac{\p
\v}{\p\phi}+\v-\bar\v\bigg)+\e\Bigg(\sin\phi
\frac{\p\psi_0}{\p\eta}\v-\e\psi_0G[\v]\Bigg)\nonumber\\
&=&\e\Bigg(\sin\phi
\frac{\p\psi_0}{\p\eta}\v-\e\psi_0G[\v]\Bigg)\nonumber.
\end{eqnarray}
Since $\psi_0=1$ when $\eta\leq R_{\min}/(4\e^{n})$, the effective region
of $\px\psi_0$ is $\eta\geq R_{\min}/(4\e^{n})$ which is further and further
from the origin as $\e\rt0$. By Theorem \ref{Milne theorem 3.}, the
first term in (\ref{remainder temp 1}) can be bounded as
\begin{eqnarray}
\im{\e\sin\phi\frac{\p\psi_0}{\p\eta}\v}{\Omega\times\s^2}&\leq&
C\e \ue^{-\frac{K_0}{\e^{n}}}\leq C\e^3.
\end{eqnarray}
Then we turn to the crucial estimate in the second term of (\ref{remainder temp 1}), by Theorem \ref{Milne tangential.}, we have
\begin{eqnarray}
\im{\e\psi_0G[\v]}{\Omega\times\s^2}&\leq&C\e^2\bigg(\im{\frac{\p \v}{\p\tau_1}}{\Omega\times\s^2}+\im{\frac{\p \v}{\p\tau_2}}{\Omega\times\s^2}
+\im{\frac{\p \v}{\p\psi}}{\Omega\times\s^2}\bigg)\\
&\leq& C\e^{2}\abs{\ln(\e)}^8.\no
\end{eqnarray}
Also, the exponential decay of $\dfrac{\p\v}{\p\tau}$ by Theorem \ref{Milne tangential.} and the rescaling $\eta=\mu/\e$ implies
\begin{eqnarray}
\tm{\e\psi_0G[\v]}{\Omega\times\s^2}&\leq& \e^2\bigg(\tm{\frac{\p \v}{\p\tau_1}}{\Omega\times\s^2}+\tm{\frac{\p \v}{\p\tau_2}}{\Omega\times\s^2}
+\tm{\frac{\p \v}{\p\psi}}{\Omega\times\s^2}\bigg)\\
&\leq&\e^2\Bigg(\int_{-\pi}^{\pi}\int_0^1(1-\mu)\Big(\lnm{\frac{\p\v}{\p\tau_1}(\mu,\tau)}^2
+\lnm{\frac{\p\v}{\p\tau_2}(\mu,\tau)}^2+\lnm{\frac{\p\v}{\p\psi}(\mu,\tau)}^2\Big)\ud{\mu}\ud{\tau}\Bigg)^{1/2}\no\\
&\leq&\e^{\frac{5}{2}}\Bigg(\int_{-\pi}^{\pi}\int_0^{1/\e}(1-\e\eta)\Big(\lnm{\frac{\p\v}{\p\tau_1}(\eta,\tau)}^2
+\lnm{\frac{\p\v}{\p\tau_2}(\eta,\tau)}^2+\lnm{\frac{\p\v}{\p\psi}(\eta,\tau)}^2\Big)\ud{\eta}\ud{\tau}\Bigg)^{1/2}\no\\
&\leq&C\e^{\frac{5}{2}}\Bigg(\int_{-\pi}^{\pi}\int_0^{1/\e}\ue^{-2K_0\eta}\abs{\ln(\e)}^{16}\ud{\eta}\ud{\tau}\Bigg)^{1/2}\no\\
&\leq& C\e^{\frac{5}{2}}\abs{\ln(\e)}^8.\no
\end{eqnarray}
Similarly, we have
\begin{eqnarray}
\nm{\e\psi_0G[\v]}_{L^{\frac{2m}{2m-1}}(\Omega\times\s^2)}&\leq&C\e^{3-\frac{1}{2m}}\abs{\ln(\e)}^8.
\end{eqnarray}
In total, we have
\begin{eqnarray}
\tm{\ll[\qb]}{\Omega\times\s^2}&\leq& C\e^{\frac{5}{2}}\abs{\ln(\e)}^8,\\
\nm{\ll[\qb]}_{L^{\frac{2m}{2m-1}}(\Omega\times\s^2)}&\leq& C\e^{3-\frac{1}{2m}}\abs{\ln(\e)}^8,\\
\im{\ll[\qb]}{\Omega\times\s^2}&\leq& C\e^{2}\abs{\ln(\e)}^8.
\end{eqnarray}
\ \\
Step 4: Diffusive Limit.\\
In summary, since $\ll[u^{\e}]=0$, collecting estimates in Step 2 and Step 3, we can prove
\begin{eqnarray}
\tm{\ll[R]}{\Omega\times\s^2}&\leq& C\e^{\frac{5}{2}}\abs{\ln(\e)}^8,\\
\nm{\ll[R]}_{L^{\frac{2m}{2m-1}}(\Omega\times\s^2)}&\leq& C\e^{3-\frac{1}{2m}}\abs{\ln(\e)}^8,\\
\im{\ll[R]}{\Omega\times\s^2}&\leq& C\e^{2}\abs{\ln(\e)}^8.
\end{eqnarray}
Also, based on our construction, it is easy to see
\begin{eqnarray}
R-\pp[R]=-\e^2(\vw\cdot\nx\u_1-\pp[\vw\cdot\nx\u_1]),
\end{eqnarray}
which further implies
\begin{eqnarray}
\tm{R-\pp[R]}{\Gamma^-}&\leq& C\e^2,\\
\nm{R-\pp[R]}_{L^{m}(\Gamma^-)}&\leq&C\e^2,\\
\im{R-\pp[R]}{\Gamma^-}&\leq& C\e^2
\end{eqnarray}
Hence, the remainder $R$ satisfies the equation
\begin{eqnarray}
\left\{
\begin{array}{rcl}
\e \vw\cdot\nabla_x R+R-\bar R&=&\ll[R]\ \ \text{for}\ \ \vx\in\Omega,\\
R-\pp[R]&=&R-\pp[R]\ \ \text{for}\ \ \vw\cdot\vn<0\ \ \text{and}\ \
\vx_0\in\p\Omega.
\end{array}
\right.
\end{eqnarray}
It is easy to verify $R$ satisfies the normalization condition (\ref{normalization.}) and the data satisfies the compatibility condition (\ref{compatibility.}). By Theorem \ref{LI estimate}, we have for $2<m\leq3$,
\begin{eqnarray}
\im{R}{\Omega\times\s^2}
&\leq& C\bigg(\frac{1}{\e^{1+\frac{3}{2m}}}\tm{\ll[R]}{\Omega\times\s^2}+
\frac{1}{\e^{2+\frac{3}{2m}}}\nm{\ll[R]}_{L^{\frac{2m}{2m-1}}(\Omega\times\s^2)}+\im{\ll[R]}{\Omega\times\s^2}\\
&&+\frac{1}{\e^{1+\frac{3}{2m}}}\nm{R-\pp[R]}_{L^2(\Gamma^-)}+\frac{1}{\e^{\frac{3}{2m}}}\nm{R-\pp[R]}_{L^{\frac{4m}{3}}(\Gamma^-)}+\im{R-\pp[R]}{\Gamma^-}\bigg)\no\\
&\leq& C\Bigg(\frac{1}{\e^{1+\frac{3}{2m}}}\bigg(\e^{\frac{5}{2}}\abs{\ln(\e)}^8\bigg)+
\frac{1}{\e^{2+\frac{3}{2m}}}\bigg(\e^{3-\frac{1}{2m}}\abs{\ln(\e)}^8\bigg)+\bigg(\e^{2}\abs{\ln(\e)}^8\bigg)\no\\
&&+\frac{1}{\e^{1+\frac{3}{2m}}}(\e^2)+\frac{1}{\e^{\frac{3}{2m}}}(\e^2)+(\e^2)\Bigg)\no\\
&\leq&C\e^{1-\frac{2}{m}}\abs{\ln(\e)}^8\leq C\e^{\frac{1}{3}-\d}
\end{eqnarray}
Note that the constant $C$ might depend on $m$ and thus depend on $\d$.
Since it is easy to see
\begin{eqnarray}
\im{\sum_{k=1}^{2}\e^k\u_k+\sum_{k=0}^{1}\e^k\ub_k}{\Omega\times\s^2}\leq C\e,
\end{eqnarray}
our result naturally follows. This completes the proof of diffusive limit.
\end{proof}

\newpage


\bibliographystyle{siam}
\bibliography{Reference}

\end{document}